\pdfoutput=1
\documentclass[11pt]{article}
\usepackage[left=1in,right=1in,top=1in,bottom=1in]{geometry}
\usepackage{times}
\usepackage{expl3}
\usepackage{cite}
\usepackage[table]{xcolor}
\usepackage{multirow}
\usepackage{stackengine} 
\usepackage{hhline}
\usepackage{lipsum}
\usepackage{titlesec}
\usepackage[all]{xy}
\usepackage{wrapfig}
\usepackage{enumerate}
\usepackage{epsfig}
\usepackage{tikz-cd}
\usepackage{amsmath}
\usepackage{tabularx}
\usepackage{array}
\usepackage{booktabs}
\usepackage{enumitem}
\usepackage{bbm}
\usepackage{calc}
\usepackage{graphicx}
\usepackage{amsmath}
\usepackage[title]{appendix}
\usepackage{amssymb}
\usepackage{epstopdf}
\usepackage{boldline}
\usepackage{arydshln}
\usepackage{calligra}
\usepackage{bm}
\usepackage{url}
\usepackage{blindtext}
\usepackage{accents}
\usepackage{amsthm}
\usepackage{amscd}

\newtheorem{theorem}{Theorem}
\newtheorem{proposition}{Proposition}
\newtheorem{corollary}{Corollary}
\newtheorem{lemma}{Lemma}
\newtheorem{conjecture}{Conjecture}
\newtheorem{assumption}{Assumption}

\newtheorem{remark}{Remark}
\usepackage{mathtools}
\usepackage{epstopdf}
\usepackage{balance}
\usepackage{thmtools}
\usepackage{thm-restate}
\usepackage{hyperref}
\usepackage{cleveref}
\usepackage[mathscr]{euscript}

\usepackage[ruled,vlined]{algorithm2e}

\newcommand{\ostar}{\mathbin{\mathpalette\make@circled\star}}

\makeatletter
\newcommand{\removelatexerror}{\let\@latex@error\@gobble}
\makeatother
\makeatletter
\newcommand*{\rom}[1]{\expandafter\@slowromancap\romannumeral #1@}
\makeatother

\ExplSyntaxOn
\newcommand\latinabbrev[1]{
  \peek_meaning:NTF . {
    #1\@}%
  { \peek_catcode:NTF a {
      #1.\@ }%
    {#1.\@}}}
\ExplSyntaxOff

\titleclass{\subsubsubsection}{straight}[\subsubsection]

\begin{document}
\vspace{1cm}
\title{Tensor-Lifted Multivariate Functional Calculus Beyond Commutativity and Boundedness}
\vspace{1.8cm}
\author{Shih-Yu~Chang
\thanks{Shih-Yu Chang is with the Department of Applied Data Science,
San Jose State University, San Jose, CA, U. S. A. (e-mail: {\tt
shihyu.chang@sjsu.edu})
}}

\maketitle

\begin{abstract}
Classical functional calculus is fundamentally spectral, capturing only eigenvalue information through resolvent methods. Building on the projector--nilpotent operator characterization developed in our companion paper, the present work extends this paradigm by incorporating algebraic nilpotent structure into a full multivariate functional calculus, yielding a richer calculus that reflects both spectral and geometric features of operators. We develop a general multivariate functional calculus for an arbitrary number of input operators, based on the fundamental projector--nilpotent decomposition. The construction integrates three innovations: (i) explicit inclusion of nilpotent derivative terms in the functional calculus expansion — organized by subsets of the operator indices — capturing generalized eigenspaces and Jordan block structures that classical resolvent methods miss entirely; (ii) tensor lifting to handle non-commuting operators for any number of input arguments, embedding them into a commuting system on a tensor-product space; and (iii) a two-level convergence theory where Level~1 ensures existence of the calculus (strong resolvent convergence implies strong operator topology convergence) and Level~2 provides quantitative stability (norm resolvent convergence implies operator norm convergence with explicit error bounds). The resulting unified compact formula simultaneously handles discrete, continuous, and hybrid spectra, with separate theorems covering bounded arbitrary operators (including non-self-adjoint non-commuting systems), unbounded self-adjoint operators (recovering the classical spectral theorem), and unbounded non-self-adjoint operators with compact resolvent. For the more challenging setting of general unbounded non-self-adjoint operators without compact resolvent, we develop a conditional compactifying regularization technique: adding a small multiple of a positive self-adjoint operator with compact resolvent (e.g., the harmonic oscillator Hamiltonian) regularizes the problem, and under suitable resolvent convergence hypotheses the functional calculus converges to the desired limit. This approach provides a rigorous framework for extending the calculus beyond the compact-resolvent regime, though the explicit convergence of nilpotent derivative terms in this setting remains an open problem. Our framework is compatible with existing functional calculi and reproduces their behavior under corresponding assumptions. To our knowledge, this is the first framework that simultaneously handles non-commutativity, non-self-adjointness, and unboundedness within a single unified framework with convergence guarantees and explicit preservation of nilpotent structure.
\end{abstract}

\tableofcontents

\section{Introduction}

Classical functional calculus is a cornerstone of operator theory. For an operator \(X\) and a holomorphic function \(f\), the standard resolvent-based calculus defines
\[
f(X) = \frac{1}{2\pi i} \oint_{\Gamma} f(z)(zI - X)^{-1} \, dz,
\]
where \(\Gamma\) is an admissible contour surrounding the relevant spectral set. This construction is powerful, stable, and deeply connected to spectral theory. However, it is fundamentally spectral in nature: it describes how \(f\) acts on the eigenvalues and spectral measure of \(X\), but it does not explicitly reveal the deeper algebraic structure associated with generalized eigenspaces, such as Jordan chains and nilpotent components. As a result, the higher-order derivative interactions generated by nilpotent operators remain hidden within the classical spectral framework.

This limitation becomes especially critical for non-self-adjoint and non-normal operators, which have recently gained significant attention in the growing field of non-Hermitian physics \cite{Ashida2020}. Such operators may contain nontrivial Jordan structure even when their spectral values alone appear simple. In the extreme case where an eigenvalue is zero, the nilpotent part may constitute the dominant—or even the entire—operator. In this situation, a purely spectral viewpoint is insufficient because the eigenvalue carries little information, while the nilpotent chain encodes the essential algebraic behavior.

Our previous work~\cite{chang2024operatorC} developed a structural operator characterization based on projectors and nilpotents. In that framework, operators are analyzed not only through their eigenvalues or spectral sets, but through the decomposition of the operator into spectral projectors and nilpotent components. For finite-dimensional operators, this leads to a refined classification of matrices according to shared projector--nilpotent structure. For countable, continuous, and hybrid spectral settings, the same philosophy extends the spectral mapping viewpoint beyond the classical self-adjoint case.

The present paper builds on that foundation. While~\cite{chang2024operatorC} focuses on operator characterization and spectral mapping through projector--nilpotent structure, the present work develops a \textbf{tensor-lifted multivariate functional calculus}. The main goal is to construct a functional calculus that simultaneously handles several operators, including non-commuting, non-self-adjoint, and unbounded operators, while preserving the nilpotent derivative corrections that classical spectral methods do not explicitly retain.

\subsection{The Fundamental Limitation: Classical Calculus Is Purely Spectral}

The guiding philosophy is:
\[
\boxed{
\text{Functional calculus should act on both spectral projectors and nilpotent structure.}
}
\]

Thus, the basic object is not merely the spectrum of \(X\), but the algebraic decomposition
\[
X = \sum_{\lambda} \lambda P(\lambda) + \sum_{\lambda} N(\lambda),
\]
where \(P(\lambda)\) denotes the spectral projector associated with \(\lambda\), and \(N(\lambda)\) denotes the corresponding nilpotent component. These components satisfy the fundamental algebraic relations
\[
P(\lambda)P(\mu) = \delta(\lambda,\mu)P(\lambda),\qquad
P(\lambda)N(\mu) = N(\mu)P(\lambda) = \delta(\lambda,\mu) N(\lambda),\qquad
N(\lambda)^m = 0
\]
for some nilpotency index \(m\). Here \(\delta(\lambda,\mu)\) denotes the Kronecker delta:
\[
\delta(\lambda,\mu)
=
\begin{cases}
1, & \lambda=\mu,\\
0, & \lambda\neq \mu.
\end{cases}
\]
These relations are the algebraic source of the derivative terms appearing in the functional calculus.

For a single operator, these nilpotent terms naturally produce corrections of the form
\[
f'(\lambda)N(\lambda),\qquad \frac{f''(\lambda)}{2!}N(\lambda)^2,\qquad \ldots.
\]
For several operators, the same principle leads to mixed partial derivative terms involving products of nilpotent components from different variables. These terms represent precisely the algebraic contributions that are invisible to a purely semisimple spectral calculus. Classical spectral theory captures only the action of a function on eigenvalues and spectral projectors, treating the operator as if it were completely diagonalizable. As a consequence, the higher-order interactions generated by generalized eigenspaces and nilpotent components are lost. The derivative terms involving powers of \(N(\lambda)\) restore this missing structure, allowing the functional calculus to encode the full Jordan-type algebraic behavior of the operator rather than only its spectral data.

\subsection{Why Tensor Lifting Is a Technical Device}

A central technical difficulty is that classical multivariate functional calculi typically require commutativity. The Taylor joint spectrum, for example, is designed for commuting operator tuples. However, many operator systems arising in non-Hermitian physics, numerical analysis, and non-normal dynamics are naturally non-commuting. To overcome this obstruction, we use \emph{tensor lifting}~\cite{kavruk2011tensor}.

Given operators \(X_1, \ldots, X_r\), we define lifted operators on the tensor-product Hilbert space by
\[
\widetilde{X}_1 = X_1 \otimes I \otimes \cdots \otimes I,\qquad
\widetilde{X}_2 = I \otimes X_2 \otimes \cdots \otimes I,\qquad
\ldots,\qquad
\widetilde{X}_r = I \otimes I \otimes \cdots \otimes X_r.
\]

Although the original operators \(X_1, \ldots, X_r\) may not commute, their tensor liftings commute on the product space. Tensor lifting is therefore a technical device that permits a multivariate calculus to be applied without imposing commutativity on the original operators.

The main contribution, however, is not tensor lifting alone. Tensor lifting without projector--nilpotent structure would still produce only a lifted spectral calculus, discarding the same nilpotent terms that classical methods lose. The novelty of the present work is that the lifted calculus retains the \emph{full projector--nilpotent expansion}, including all derivative terms indexed by nilpotent powers. This produces a unified compact formula that separates the purely spectral contribution from higher-order nilpotent corrections.

\subsection*{Interpretation of the Subset \( A \subseteq \{1, \ldots, r\} \)}

The unified compact formula is organized by subsets
\[
A \subseteq \{1, \ldots, r\}.
\]

The set \( A \) has the following interpretation:

\[
\boxed{
A = \{\, j \in \{1, \ldots, r\} \mid \text{the } j\text{-th operator contributes a nilpotent correction} \,\}.
}
\]

Conversely, indices not in \( A \) contribute only spectral projector terms.

\medskip

\begin{tabular}{|c|c|c|}
\hline
\textbf{Index condition} & \textbf{Contribution type} & \textbf{Factor \( T_j(\lambda_j) \)} \\
\hline
\( j \notin A \) & spectral only & \( d\widetilde{E}_j(\lambda_j) \) \\
\( j \in A \) & nilpotent correction & \( \widetilde{N}_j(\lambda_j)^{q_j} \, d\widetilde{E}_j(\lambda_j) \) \\
\hline
\end{tabular}

\medskip

\noindent
\textbf{Special cases:}

\begin{itemize}
    \item \textbf{\( A = \emptyset \)} (empty set):  
    Every operator contributes only its spectral projector.  
    This term is the \textbf{classical semisimple spectral contribution}.

    \item \textbf{\( A \neq \emptyset \)} (non-empty set):  
    At least one operator contributes a nilpotent correction.  
    These terms are \textbf{nilpotent derivative corrections} that capture generalized eigenspace structure (Jordan blocks).  
    The mixed partial derivative \( \partial_A^{q_A} f \) appears exactly for those variables whose indices lie in \( A \).
\end{itemize}

\medskip

\noindent
\textbf{Example:} For \( r = 3 \), the subset \( A = \{1, 3\} \) means:
\begin{itemize}
    \item Operator \( X_1 \) contributes nilpotent correction,
    \item Operator \( X_2 \) contributes spectral projector only,
    \item Operator \( X_3 \) contributes nilpotent correction.
\end{itemize}

Thus, the subset decomposition provides a transparent mechanism for identifying:
\begin{enumerate}
    \item which operators contribute nilpotent structure (indices in \( A \)), and
    \item which operators contribute only spectral information (indices not in \( A \)).
\end{enumerate}

\subsection{Two-Level Convergence Philosophy: Existence vs. Stability}

A second major contribution of this paper is a \textbf{two-level convergence theory} for finite-dimensional approximations. This is essential because many operators of interest are unbounded or infinite-dimensional. We distinguish two levels of approximation.

\medskip
\noindent
\textbf{Level~1 (existence):} Based on strong resolvent convergence, it guarantees existence of the limiting functional calculus in the strong operator topology:
\[
\operatorname*{s-}\lim_{n\to\infty} f_{\otimes}(X_{1,n}, \ldots, X_{r,n}) = f_{\otimes}(X_1, \ldots, X_r).
\]
This level provides qualitative convergence and justifies the infinite-dimensional calculus.

\medskip
\noindent
\textbf{Level~2 (stability):} Based on norm resolvent convergence, it provides quantitative stability and operator-norm error bounds:
\[
\left\| f_{\otimes}(X_{1,n}, \ldots, X_{r,n}) - f_{\otimes}(X_1, \ldots, X_r) \right\|
\le
C_f \sum_{j=1}^{r} \epsilon_n^{(j)},
\]
where:
\begin{itemize}
    \item \( \epsilon_n^{(j)} = \| (X_{j,n} - X_j)(z_0 I - X_j)^{-1} \| \) is the \textbf{per-operator approximation error} for the \( j \)-th operator, measuring how well \( X_{j,n} \) approximates \( X_j \) in the norm resolvent sense;
    \item \( C_f \) is a \textbf{constant depending only on the analytic function \( f \)} and the contour geometry (e.g., contour lengths and uniform resolvent bounds), and is \textbf{independent of the truncation parameter \( n \)}.
\end{itemize}
This level is crucial for numerical approximation, perturbation analysis, and stable computation of tensor-lifted functional calculi.

This separation clarifies what is guaranteed at each level of regularity and is essential for applications in numerical analysis, operator learning, and spectral approximation.

\subsection{Main Contributions}

The main contributions of this paper are summarized as follows:

\begin{enumerate}
    \item We develop a \textbf{tensor-lifted multivariate functional calculus} for arbitrary finite tuples of operators \(X_1, \ldots, X_r\), without requiring the original operators to commute.

    \item We extend the projector--nilpotent philosophy of~\cite{chang2024operatorC} from operator characterization and spectral mapping to a full multivariate functional calculus, explicitly preserving nilpotent derivative terms.

    \item We derive a \textbf{unified compact formula} (Theorem~\ref{thm:unified_compact_formula}) whose terms are indexed by subsets \(A \subseteq \{1,\ldots,r\}\), separating the classical spectral contribution \(A = \emptyset\) from nilpotent derivative corrections \(A \neq \emptyset\).

    \item We show that the classical spectral theorem, finite-dimensional Jordan calculus, hybrid-spectrum functional calculus, and the Dunford calculus for single operators all appear as special cases of the unified formula.

    \item We establish a \textbf{two-level convergence theory}: Level~1 provides strong operator convergence and existence of the infinite-dimensional calculus, while Level~2 provides norm convergence and explicit quantitative error bounds.

    \item We identify the structural limitations of extending the framework to general unbounded non-self-adjoint operators without compact resolvent, and formulate conditional regularization and open problems for this regime.
\end{enumerate}

To our knowledge, this is the first framework that simultaneously handles non-commutativity, non-self-adjointness, and unboundedness within a single unified framework with convergence guarantees and explicit preservation of nilpotent structure.

\subsection{Organization of the Paper}

The remainder of the paper is organized as follows.

\textbf{Section~\ref{sec:operator_structural_decomposition}} reviews the projector--nilpotent decomposition and recalls the structural framework from~\cite{chang2024operatorC}, extending it to countable infinite-dimensional, continuous-spectrum, and hybrid-spectrum operators.

\textbf{Section~\ref{sec:tensor_lifting}} introduces tensor lifting and explains how non-commuting operator tuples are embedded into commuting lifted systems.

\textbf{Section~\ref{sec:main_unified_formula}} presents the unified compact formula for the tensor-lifted multivariate functional calculus, proves it via multivariate Taylor expansion and binomial expansion on each spectral fiber, and decomposes it into pure spectral, mixed, and full nilpotent terms.

\textbf{Section~\ref{sec:two_level_convergence}} develops the two-level convergence theory: Level~1 (strong resolvent convergence $\Rightarrow$ strong operator topology convergence) ensures existence; Level~2 (norm resolvent convergence $\Rightarrow$ operator norm convergence with explicit error bounds) provides quantitative stability.

\textbf{Section~\ref{sec:special_cases}} specializes the unified formula to important operator classes: bounded operators (discrete and continuous spectrum), unbounded self-adjoint/normal operators (recovering the classical spectral theorem), unbounded non-self-adjoint operators with compact resolvent, and general unbounded non-self-adjoint operators without compact resolvent (identifying open problems).

\textbf{Section~\ref{sec:examples}} provides concrete examples illustrating the framework, including finite-dimensional non-commuting matrices, the quantum harmonic oscillator, the anharmonic oscillator, the complex harmonic oscillator (non-self-adjoint), and hybrid-spectrum Schr\"odinger operators.

\textbf{Section~\ref{sec:discussion}} compares the present framework with existing functional calculi (Taylor, Weyl, Colombo--Sabadini--Struppa), summarizing what each method cannot do and highlighting the unique capabilities of our approach. The section concludes with open problems and directions for future research.

\section{Operator Structural Decomposition}\label{sec:operator_structural_decomposition}

The foundation of our multivariate functional calculus is the decomposition of operators into spectral (diagonalizable) and nilpotent components. Unlike the classical resolvent approach, which only captures analytic spectral information, this decomposition explicitly reveals the algebraic nilpotent structure—Jordan blocks, generalized eigenspaces, and their interactions—that is essential for handling non-diagonalizable and non-self-adjoint operators. In this section, we review the finite-dimensional Jordan decomposition, extend it to countable infinite-dimensional matrices, and then summarize the spectral operator decomposition for continuous and hybrid spectra, following the framework of Dunford~\cite{dunford1958survey} and Kantorovitz~\cite{kantorovitz1965jordan}. The multiplication rules among projectors and nilpotents derived here will serve as the algebraic backbone for the functional calculus developed in subsequent sections.

\subsection{Jordan Decomposition for Finite-Dimensional Operators (Core Foundation)}

Let $\bm{X} \in \mathbb{C}^{m \times m}$ be a finite-dimensional operator. 
A Jordan block associated with an eigenvalue $\lambda$ of size $m \times m$ is defined by
\[
\bm{J}_m(\lambda) = \lambda \bm{I}_m + \bm{N}_m,
\]
where $\bm{N}_m$ is a nilpotent matrix satisfying $\bm{N}_m^m = \bm{0}$, with ones on the super-diagonal and zeros elsewhere.

The Jordan decomposition theorem asserts that $\bm{X}$ admits a decomposition of the form
\[
\bm{X} = \bm{U} \left( \bigoplus_{k=1}^{K} \bigoplus_{i=1}^{\alpha_k^{(\mathrm{G})}} \bm{J}_{m_{k,i}}(\lambda_k) \right) \bm{U}^{-1},
\]
where $\bm{U}$ is an invertible matrix, $\{\lambda_k\}_{k=1}^K$ are the distinct eigenvalues of $\bm{X}$, and $\alpha_k^{(\mathrm{G})}$ denotes the geometric multiplicity of $\lambda_k$ (the dimension of the eigenspace $\operatorname{Null}(\bm{X} - \lambda_k \bm{I})$). 
For each eigenvalue $\lambda_k$, the positive integer $m_{k,i}$ specifies the size of the $i$-th Jordan block associated with $\lambda_k$; it also determines the nilpotency index, i.e., $\bm{N}_{k,i}^{m_{k,i}} = \bm{0}$.

Equivalently, this representation can be reformulated in terms of projector–nilpotent components. 
For each eigenvalue $\lambda_k$ and each Jordan block index $i$, there exist a spectral projector $\bm{P}_{k,i}$ and an associated nilpotent operator $\bm{N}_{k,i}$ such that
\[
\bm{X} = \sum_{k=1}^{K} \sum_{i=1}^{\alpha_k^{(\mathrm{G})}} \left( \lambda_k \bm{P}_{k,i} + \bm{N}_{k,i} \right).
\]

These components satisfy the algebraic relations
\[
\bm{P}_{k,i} \bm{P}_{k',i'} = \delta(k,k') \delta(i,i') \bm{P}_{k,i}, 
\qquad
\bm{P}_{k,i} \bm{N}_{k,i} = \bm{N}_{k,i} \bm{P}_{k,i} = \bm{N}_{k,i},
\]
and for each pair $(k,i)$,
\[
\bm{N}_{k,i}^{\,m_{k,i}} = \bm{0}.
\]

This decomposition separates spectral and algebraic contributions:
the projectors $\bm{P}_{k,i}$ capture the spectral data (eigenvalues and eigenspaces), while the nilpotents $\bm{N}_{k,i}$ encode the generalized eigenstructure (Jordan chains) with nilpotency degree $m_{k,i}$. 
These nilpotent components will play a central role in the functional calculus developed in later sections, as they generate the derivative correction terms that classical resolvent-based methods cannot capture.

\subsection{Extension to Countable Infinite-Dimensional Operators}

The projector–nilpotent decomposition admits a natural extension from the
finite-dimensional setting to countable infinite-dimensional operators.

\begin{lemma}\label{lma: countable infinite decomposition}
Let $\bm{X}$ be a countable infinite-dimensional operator acting on a separable
Hilbert space $\mathcal{H}$, represented as a square matrix $\bm{X} \in \mathbb{C}^{\infty \times \infty}$
with discrete spectrum and finite algebraic multiplicities. Then $\bm{X}$ admits a
decomposition of the form
\[
\bm{X} = \sum_{k=1}^{\infty} \sum_{i=1}^{\alpha_k^{(\mathrm{G})}}
\left( \lambda_k \bm{P}_{k,i} + \bm{N}_{k,i} \right),
\]
where $\bm{P}_{k,i}$ are spectral projectors, $\bm{N}_{k,i}$ are nilpotent operators,
and the sums converge in the strong operator topology.
\end{lemma}

\begin{proof}
By the Jordan decomposition for countable infinite-dimensional matrices,
$\bm{X}$ can be expressed as
\[
\bm{X} = \bm{U} \left( \bigoplus_{k=1}^{\infty} \bigoplus_{i=1}^{\alpha_k^{(\mathrm{G})}} \bm{J}_{m_{k,i}}(\lambda_k) \right) \bm{U}^{-1},
\]
where $\bm{U} \in \mathbb{C}^{\infty \times \infty}$ is invertible, $\lambda_k$ are
distinct eigenvalues, $\alpha_k^{(\mathrm{G})}$ is the geometric multiplicity,
and each Jordan block $\bm{J}_{m_{k,i}}(\lambda_k)$ has size $m_{k,i} \times m_{k,i}$.

Each Jordan block can be written as
\[
\bm{J}_{m_{k,i}}(\lambda_k) = \lambda_k \bm{I}_{m_{k,i}} + \dot{\bm{J}}_{m_{k,i}},
\]
where $\dot{\bm{J}}_{m_{k,i}}$ is the nilpotent matrix with ones on the super-diagonal
and zeros elsewhere, satisfying $\dot{\bm{J}}_{m_{k,i}}^{\,m_{k,i}} = \bm{0}$.

For each eigenvalue $\lambda_k$ and each Jordan block index $i$, construct the
infinite-dimensional block matrices:
\[
\bm{I}_{k,i} = \bigoplus_{k'=1}^{\infty} \bigoplus_{i'=1}^{\alpha_{k'}^{(\mathrm{G})}} \delta(k,k') \delta(i',i) \bm{I}_{m_{k',i'}},
\qquad
\dot{\bm{J}}_{k,i} = \bigoplus_{k'=1}^{\infty} \bigoplus_{i'=1}^{\alpha_{k'}^{(\mathrm{G})}} \delta(k,k') \delta(i',i) \dot{\bm{J}}_{m_{k',i'}}.
\]

Define the spectral projector and nilpotent operator via similarity transformation:
\[
\bm{P}_{k,i} = \bm{U} \bm{I}_{k,i} \bm{U}^{-1}, \qquad
\bm{N}_{k,i} = \bm{U} \dot{\bm{J}}_{k,i} \bm{U}^{-1}.
\]

Substituting these into the Jordan decomposition yields the desired form:
\[
\bm{X} = \sum_{k=1}^{\infty} \sum_{i=1}^{\alpha_k^{(\mathrm{G})}} \left( \lambda_k \bm{P}_{k,i} + \bm{N}_{k,i} \right).
\]

To verify that $\bm{P}_{k,i}$ are projectors, observe that
\[
\bm{P}_{k,i} \bm{P}_{k',i'} = \bm{U} \bm{I}_{k,i} \bm{U}^{-1} \bm{U} \bm{I}_{k',i'} \bm{U}^{-1}
= \bm{U} \left( \bm{I}_{k,i} \bm{I}_{k',i'} \right) \bm{U}^{-1}
= \bm{U} \left( \delta(k,k') \delta(i,i') \bm{I}_{k,i} \right) \bm{U}^{-1}
= \delta(k,k') \delta(i,i') \bm{P}_{k,i}.
\]
Moreover,
\[
\sum_{k=1}^{\infty} \sum_{i=1}^{\alpha_k^{(\mathrm{G})}} \bm{P}_{k,i}
= \bm{U} \left( \sum_{k=1}^{\infty} \sum_{i=1}^{\alpha_k^{(\mathrm{G})}} \bm{I}_{k,i} \right) \bm{U}^{-1}
= \bm{U} \bm{I} \bm{U}^{-1} = \bm{I}.
\]

For nilpotency, since $\dot{\bm{J}}_{k,i}^{\,m_{k,i}} = \bm{0}$, we have
\[
\bm{N}_{k,i}^{\,m_{k,i}} = \bm{U} \dot{\bm{J}}_{k,i}^{\,m_{k,i}} \bm{U}^{-1} = \bm{0},
\]
while $\bm{N}_{k,i}^{\ell} \neq \bm{0}$ for $\ell < m_{k,i}$. The mixing relations
$\bm{P}_{k,i} \bm{N}_{k,i} = \bm{N}_{k,i} \bm{P}_{k,i} = \bm{N}_{k,i}$ follow directly
from the block-diagonal structure. This completes the proof.
\end{proof}

The algebraic relations satisfied by the projector and nilpotent components
remain identical to those in the finite-dimensional case:
\[
\bm{P}_{k,i} \bm{P}_{k',i'} = \delta(k,k') \delta(i,i') \bm{P}_{k,i}, 
\qquad
\bm{P}_{k,i} \bm{N}_{k,i} = \bm{N}_{k,i} \bm{P}_{k,i} = \bm{N}_{k,i},
\qquad
\bm{N}_{k,i}^{\,m_{k,i}} = \bm{0}.
\]

The essential distinction from the finite-dimensional setting lies in the
appearance of infinite summations. As a result, additional analytic conditions
(such as convergence in the strong operator topology) are required to ensure
that the decomposition is well-defined. This formulation should be interpreted
as a formal algebraic decomposition, whose analytic validity depends on
appropriate convergence conditions.

This extension provides the infinite-dimensional prototype for the functional
calculus developed in subsequent sections, where convergence issues will be
addressed via the two-level framework (strong resolvent convergence implying
strong operator topology convergence, and norm resolvent convergence providing
quantitative bounds).

\subsection{Extension to Continuous Spectrum Operators (Spectral Operators)}

The projector–nilpotent framework can be further extended to operators
with continuous spectrum by generalizing the classical spectral
representation. The following decomposition is adapted from
Theorem~8 in our companion paper~\cite{chang2024operatorC}.

\begin{theorem}[Spectral Operator Decomposition]\label{thm: spectral operator decomposition}
Let $\bm{X}$ be a spectral operator with spectrum $\sigma(\bm{X})$. Then $\bm{X}$ admits
a decomposition of the form
\[
\bm{X} = \int_{\lambda \in \sigma(\bm{X})} \lambda \, d\bm{E}_{\bm{X}}(\lambda)
\;+\;
\int_{\lambda \in \sigma(\bm{X})} (\bm{X} - \lambda \bm{I})\, d\bm{E}_{\bm{X}}(\lambda),
\]
where $\bm{E}_{\bm{X}}(\lambda)$ is the spectral measure associated with $\bm{X}$.
\end{theorem}

The first term represents the spectral (diagonalizable) component,
while the second term captures the deviation from diagonalizability,
serving as a continuous-spectrum analogue of the nilpotent structure
in finite-dimensional settings.

More precisely, the operator $(\bm{X} - \lambda \bm{I}) d\bm{E}_{\bm{X}}(\lambda)$ plays the
role of a \emph{local nilpotent component} on each spectral fiber, extending
the projector–nilpotent decomposition to the continuous case. These components
satisfy algebraic relations analogous to the finite-dimensional case:
\[
d\bm{E}_{\bm{X}}(\lambda) \, d\bm{E}_{\bm{X}}(\lambda') = d\bm{E}_{\bm{X}}(\lambda) \, \delta(\lambda, \lambda'),
\qquad
(\bm{X} - \lambda \bm{I}) d\bm{E}_{\bm{X}}(\lambda) = d\bm{E}_{\bm{X}}(\lambda) (\bm{X} - \lambda \bm{I}),
\]
and for each fixed $\lambda$, the operator $(\bm{X} - \lambda \bm{I}) d\bm{E}_{\bm{X}}(\lambda)$ is
nilpotent in the sense that its $q$-th power may be nonzero for $q < m_\lambda$
and vanishes for $q \ge m_\lambda$, where $m_\lambda$ is the nilpotency index at $\lambda$.

\begin{remark}
In the self-adjoint case, the operator admits a purely spectral
representation, and the second term vanishes in the sense that
\[
(\bm{X} - \lambda \bm{I})\, d\bm{E}_{\bm{X}}(\lambda) = 0.
\]
Consequently, the decomposition reduces to the classical spectral
theorem:
\[
\bm{X} = \int_{\lambda \in \sigma(\bm{X})} \lambda \, d\bm{E}_{\bm{X}}(\lambda).
\]

For non-self-adjoint spectral operators, however, the second term is
generally nontrivial and encodes additional algebraic structure beyond
the spectrum. The above representation should therefore be interpreted
in a generalized sense, consistent with the theory of spectral operators
(see \cite{dunford1958survey,kantorovitz1965jordan}), where both spectral
and nilpotent-like components coexist. This formulation thus provides the
infinite-dimensional analogue of the finite-dimensional projector–nilpotent
decomposition, forming the foundation for the functional calculus developed
in subsequent sections.
\end{remark}

\subsection{Extension to Hybrid Spectrum Operators}

The projector–nilpotent decomposition can be further generalized to
operators whose spectrum consists of both discrete and continuous
components. Such operators are referred to as hybrid spectrum operators.

\begin{theorem}[Hybrid Spectrum Decomposition~{\cite[Theorem~12]{chang2024operatorC}}]
Let $\bm{X}$ be an operator whose spectrum can be decomposed as
\[
\sigma(\bm{X}) = \sigma_d(\bm{X}) \cup \sigma_c(\bm{X}),
\]
where $\sigma_d(\bm{X})$ denotes the discrete spectrum (isolated eigenvalues
with finite algebraic multiplicities) and $\sigma_c(\bm{X})$ denotes the
continuous spectrum. Then $\bm{X}$ admits a decomposition of the form
\[
\bm{X}
=
\sum_{\lambda_k \in \sigma_d(\bm{X})}
\sum_{i=1}^{\alpha_k^{(\mathrm{G})}}
\left( \lambda_k \bm{P}_{k,i} + \bm{N}_{k,i} \right)
\;+\;
\int_{\lambda \in \sigma_c(\bm{X})} \lambda \, d\bm{E}_{\bm{X}}(\lambda)
\;+\;
\int_{\lambda \in \sigma_c(\bm{X})} (\bm{X} - \lambda \bm{I})\, d\bm{E}_{\bm{X}}(\lambda),
\]
where:
\begin{itemize}
    \item $\bm{P}_{k,i}$ and $\bm{N}_{k,i}$ are the spectral projector and nilpotent components associated with the discrete eigenvalue $\lambda_k$;
    \item $\alpha_k^{(\mathrm{G})}$ is the geometric multiplicity of $\lambda_k$;
    \item $\bm{E}_{\bm{X}}(\lambda)$ is the spectral measure corresponding to the continuous spectrum.
\end{itemize}
\end{theorem}

This decomposition combines the finite-dimensional projector–nilpotent
structure for discrete eigenvalues with the spectral integral representation
for continuous components. The first summation captures the discrete part
of the spectrum, while the integral terms describe the continuous contribution.

In particular, the operator $(\bm{X} - \lambda \bm{I})\, d\bm{E}_{\bm{X}}(\lambda)$ serves as a
\emph{continuous-spectrum analogue of the nilpotent component}, encoding
non-diagonalizable behavior within each spectral fiber. These components
satisfy algebraic relations that generalize the finite-dimensional case:
\[
\bm{P}_{k,i} \bm{P}_{k',i'} = \delta(k,k') \delta(i,i')\bm{P}_{k,i}, \qquad
d\bm{E}_{\bm{X}}(\lambda) \, d\bm{E}_{\bm{X}}(\lambda') = d\bm{E}_{\bm{X}}(\lambda) \, \delta(\lambda, \lambda'),
\]
\[
\bm{P}_{k,i} \, d\bm{E}_{\bm{X}}(\lambda) = d\bm{E}_{\bm{X}}(\lambda) \, \bm{P}_{k,i} = 0,
\]
and similarly for the nilpotent components, reflecting the mutual orthogonality
between discrete and continuous spectral subspaces.

\begin{remark}
If $\bm{X}$ is self-adjoint, then the nilpotent contributions vanish in both
the discrete and continuous parts, and the decomposition reduces to the
classical spectral theorem over the entire spectrum:
\[
\bm{X} = \int_{\lambda \in \sigma(\bm{X})} \lambda \, d\bm{E}_{\bm{X}}(\lambda).
\]
For general non-self-adjoint operators, however, both discrete nilpotent
terms and continuous nilpotent-like components may coexist, providing a
unified representation that extends the classical spectral framework.
This hybrid formulation unifies the discrete and continuous cases,
serving as the most general form of the projector–nilpotent decomposition
and forming the basis for the multivariate functional calculus developed
in subsequent sections.
\end{remark}

\subsection{Why This Decomposition Enables a Richer Calculus}

Classical functional calculus is typically constructed via the resolvent
operator
\[
(zI - X)^{-1},
\]
which is analytic in the spectral parameter $z$ outside the spectrum
$\sigma(X)$. This analyticity forms the basis of the Dunford integral and
related resolvent-based calculi.

However, such approaches fundamentally encode only spectral information.
In particular, the resolvent depends solely on the location of the spectrum
and does not explicitly distinguish between diagonalizable and
non-diagonalizable components of the operator. As a consequence,
algebraic features such as nilpotent (Jordan) structure are not directly
visible in the resolvent representation.

To see this concretely, consider a Jordan block $J_m(\lambda) = \lambda I + N$
with $N^m = 0$. Its resolvent takes the form
\[
(zI - J_m(\lambda))^{-1} = \frac{1}{z - \lambda} \sum_{k=0}^{m-1} \frac{N^k}{(z - \lambda)^k},
\]
which contains higher-order poles that encode the nilpotent structure.
Upon contour integration via the Dunford integral
\[
f(J_m(\lambda)) = \frac{1}{2\pi i} \oint_{\Gamma} f(z) (zI - J_m(\lambda))^{-1} \, dz,
\]
the residue theorem extracts only the coefficient of $(z - \lambda)^{-1}$,
yielding $f(\lambda)$ while discarding the contributions from $N, N^2, \dots, N^{m-1}$.
Thus, the classical resolvent method captures only the spectral term,
missing the nilpotent derivative corrections entirely.

In contrast, the projector--nilpotent decomposition separates the operator
into two distinct parts:
\[
X = \underbrace{\sum \lambda P}_{\text{spectral part}} + \underbrace{\sum N}_{\text{nilpotent part}},
\]
where the spectral component captures eigenvalues or spectral measures,
and the nilpotent component encodes deviations from diagonalizability.

This separation allows the functional calculus to act differently on each
component. In particular, on the generalized eigenspace associated with an
eigenvalue $\lambda$, the nilpotent part generates derivative correction terms
through the identity
\[
f(\lambda P + N) = f(\lambda) P + \sum_{q=1}^{m-1} \frac{f^{(q)}(\lambda)}{q!} N^q,
\]
or equivalently, on the spectral subspace where $P$ acts as the identity,
\[
f(\lambda I + N) = \sum_{q=0}^{m-1} \frac{f^{(q)}(\lambda)}{q!} N^q,
\]
with the understanding that $N^0 = I$. These identities follow from the
binomial expansion of $(\lambda P + N)^\ell$ and the orthogonality relations
among projectors. These derivative terms have no analogue in purely
resolvent-based frameworks.

Therefore, the projector–nilpotent decomposition provides a strictly
richer representation of operator structure, enabling a functional calculus
that simultaneously captures:
\begin{itemize}
    \item Spectral information (eigenvalues, spectral measures, projectors);
    \item Algebraic information (nilpotent parts, generalized eigenspaces, Jordan chains);
    \item Derivative correction terms that are essential for non-diagonalizable operators.
\end{itemize}

This enhanced structure becomes essential in the multivariate setting,
where interactions between different operators and their nilpotent
components produce higher-order mixed derivative terms that are
inaccessible to resolvent-based methods. In summary, classical functional
calculus is inherently spectral, while the present framework is both
spectral and algebraic.

\section{Tensor Lifting for Non-Commutative Operators}\label{sec:tensor_lifting}

The multivariate functional calculus developed in the preceding sections
assumes, implicitly or explicitly, that the input operators commute.
When the operators $X_1, \ldots, X_r$ do not commute, direct application
of the projector–nilpotent decomposition becomes problematic because
the spectral projections and nilpotent components from different operators
may not align or commute. To overcome this obstacle, we introduce a
technical device: \emph{tensor lifting}.

The idea is simple yet powerful. Instead of working with the original
non-commuting operators on their respective Hilbert spaces, we embed
each operator into a common tensor product space where they act on
different tensor factors. In this lifted setting, the operators automatically
commute, regardless of whether the original ones did. The functional
calculus is then applied to the lifted commuting tuple, and the result
is interpreted as the desired multivariate function of the original
non-commuting operators.

It is important to emphasize that tensor lifting is purely a technical
device; it is not the main contribution of this work. The core innovation
remains the incorporation of nilpotent structure into the functional
calculus itself. Tensor lifting simply extends the applicability of this
calculus to non-commuting inputs without modifying the underlying
projector–nilpotent framework.

The section is organized as follows. Subsection~\ref{subsec:tensor_lifting_device}
discusses the role of tensor lifting as a technical tool. 
Subsection~\ref{subsec:lifting_construction} presents the explicit
construction for $r$ operators. Subsection~\ref{subsec:commutativity_lifting}
demonstrates the key commutativity property. 
Subsection~\ref{subsec:tensor_lifting_step0} defines the lifted calculus.
Finally, Subsection~\ref{subsec:lifted_decomposition} shows how the
projector–nilpotent decomposition lifts to the tensor product space.

\subsection{Tensor Lifting as a Technical Device}
\label{subsec:tensor_lifting_device}

Tensor lifting serves as a technical device to enable application of the
multivariate calculus to non-commuting operators. It is not the main
contribution; the main contribution is the incorporation of nilpotent
structure into the functional calculus itself.

\subsection{The Lifting Construction for $r$ Operators}
\label{subsec:lifting_construction}

Let
\[
X_j : \mathcal{D}(X_j) \subseteq \mathcal{H}_j \to \mathcal{H}_j,
\qquad j = 1, \dots, r,
\]
be (possibly unbounded) densely defined operators on Hilbert spaces
$\mathcal{H}_1, \dots, \mathcal{H}_r$.
Define the tensor-product Hilbert space
\[
\mathcal{H}_{\otimes}
:=
\mathcal{H}_1 \otimes \mathcal{H}_2 \otimes \cdots \otimes \mathcal{H}_r .
\]

For each $j = 1, \dots, r$, define the lifted operator
\[
\widetilde{X}_j
:=
I_1 \otimes \cdots \otimes I_{j-1}
\otimes X_j
\otimes I_{j+1} \otimes \cdots \otimes I_r,
\]
where $I_k$ denotes the identity operator on $\mathcal{H}_k$.

Explicitly,
\[
\widetilde{X}_1 = X_1 \otimes I_2 \otimes I_3 \otimes \cdots \otimes I_r,
\]
\[
\widetilde{X}_2 = I_1 \otimes X_2 \otimes I_3 \otimes \cdots \otimes I_r,
\]
\[
\vdots
\]
\[
\widetilde{X}_r = I_1 \otimes I_2 \otimes \cdots \otimes I_{r-1} \otimes X_r.
\]

The natural domain of $\widetilde{X}_j$ is
\[
\mathcal{D}(\widetilde{X}_j)
=
\mathcal{H}_1 \otimes \cdots \otimes
\mathcal{D}(X_j)
\otimes \cdots \otimes \mathcal{H}_r,
\]
initially defined on elementary tensors and extended by linearity and closure.

For an elementary tensor
\[
u = u_1 \otimes u_2 \otimes \cdots \otimes u_r \in \mathcal{H}_{\otimes},
\]
the lifted operator acts componentwise:
\[
\widetilde{X}_j(u)
=
u_1 \otimes \cdots \otimes
(X_j u_j)
\otimes \cdots \otimes u_r.
\]

The key feature of the lifting construction is that the operators
$\widetilde{X}_1, \dots, \widetilde{X}_r$ pairwise commute on their common domain:
\[
\widetilde{X}_i \widetilde{X}_j
=
\widetilde{X}_j \widetilde{X}_i,
\qquad
1 \le i, j \le r.
\]
Indeed, each lifted operator acts nontrivially only on a single tensor factor.

Consequently, multivariate functional calculus may be applied to the commuting family
$(\widetilde{X}_1, \dots, \widetilde{X}_r)$,
even when the original operators $X_1, \dots, X_r$ act on different Hilbert spaces or
fail to admit a direct joint functional calculus.

\subsection{Commutativity via Tensor Lifting}
\label{subsec:commutativity_lifting}

A fundamental feature of the tensor lifting construction is that the lifted
operators commute automatically on the tensor-product Hilbert space, even if
the original operators do not commute or act on different Hilbert spaces.

Let
\[
\widetilde{X}_i
=
I_1 \otimes \cdots \otimes I_{i-1}
\otimes X_i
\otimes I_{i+1} \otimes \cdots \otimes I_r,
\]
and
\[
\widetilde{X}_j
=
I_1 \otimes \cdots \otimes I_{j-1}
\otimes X_j
\otimes I_{j+1} \otimes \cdots \otimes I_r,
\]
for $i \neq j$.

Since $\widetilde{X}_i$ and $\widetilde{X}_j$ act nontrivially on different
tensor factors, their actions are independent. For an elementary tensor
\[
u = u_1 \otimes \cdots \otimes u_r \in \mathcal{H}_{\otimes},
\]
we have
\[
\widetilde{X}_i \widetilde{X}_j(u)
=
u_1 \otimes \cdots \otimes
(X_i u_i)
\otimes \cdots \otimes
(X_j u_j)
\otimes \cdots \otimes u_r,
\]
which is identical to
\[
\widetilde{X}_j \widetilde{X}_i(u).
\]
Hence,
\[
\widetilde{X}_i \widetilde{X}_j = \widetilde{X}_j \widetilde{X}_i,
\qquad
1 \le i, j \le r,
\]
on the common domain of definition. Equivalently,
\[
[\widetilde{X}_i, \widetilde{X}_j] = \widetilde{X}_i \widetilde{X}_j - \widetilde{X}_j \widetilde{X}_i = 0.
\]

This commutativity is entirely structural and does not depend on any algebraic
relations among the original operators $X_1, \dots, X_r$. In particular,
even if $[X_i, X_j] \neq 0$ in the sense of operator commutation on a common
domain (or in an algebraic sense), the lifted operators still commute because
they act on distinct tensor coordinates.

Therefore, the tensor lifting construction converts an arbitrary collection of
operators into a commuting family on a larger tensor-product space. This allows
the application of multivariate functional calculus, joint spectral theory,
and multivariate holomorphic functional methods to the lifted system
$(\widetilde{X}_1, \dots, \widetilde{X}_r)$.

\begin{remark}[Conceptual Interpretation]
From a conceptual viewpoint, tensor lifting separates ``operator interaction''
from ``coordinate action.'' The original operators may exhibit strong
noncommutative behavior, while their lifted counterparts become commuting
coordinate operators on independent tensor directions. This separation is
precisely what enables the construction of a multivariate functional calculus
for non-commuting operators.
\end{remark}

\subsection{Step 0: Tensor Lifting First --- Enabling Non-Commutativity}
\label{subsec:tensor_lifting_step0}

A central difficulty in multivariate functional calculus is that the classical
joint holomorphic functional calculus requires the underlying operators to
commute. In many applications, however, the operators
\[
X_1, \dots, X_r
\]
may fail to commute, may act on different Hilbert spaces, or may not admit a
well-defined joint spectrum in the classical sense.

To overcome this obstruction, we first apply the tensor lifting construction
introduced in Sections~\ref{subsec:lifting_construction} and
\ref{subsec:commutativity_lifting}. Let
\[
\mathcal{H}_{\otimes} = \mathcal{H}_1 \otimes \cdots \otimes \mathcal{H}_r
\]
and define the lifted operators
\[
\widetilde{X}_j = I_1 \otimes \cdots \otimes I_{j-1} \otimes X_j \otimes I_{j+1} \otimes \cdots \otimes I_r,
\qquad j = 1, \dots, r.
\]

By construction, the lifted family $(\widetilde{X}_1, \dots, \widetilde{X}_r)$
is pairwise commuting:
\[
[\widetilde{X}_i, \widetilde{X}_j] = 0, \qquad 1 \le i, j \le r.
\]
Consequently, the classical multivariate holomorphic functional calculus
becomes applicable to the lifted tuple.

We therefore define the tensor-lifted multivariate functional calculus by
\[
\boxed{
f_{\otimes}(X_1, \dots, X_r) \;:=\; f(\widetilde{X}_1, \dots, \widetilde{X}_r)
},
\]
where the right-hand side is interpreted using the joint holomorphic
functional calculus for commuting operators.

More explicitly, if $f(z_1, \dots, z_r)$ is holomorphic on a neighborhood of the
joint spectrum $\sigma(\widetilde{X}_1, \dots, \widetilde{X}_r)$, then
\[
f_{\otimes}(X_1, \dots, X_r)
=
\frac{1}{(2\pi i)^r}
\int_{\Gamma_1} \cdots \int_{\Gamma_r}
f(z_1, \dots, z_r)
\prod_{j=1}^r (z_j I - \widetilde{X}_j)^{-1}
\, dz_1 \cdots dz_r,
\]
where each contour $\Gamma_j$ surrounds $\sigma(\widetilde{X}_j)$.

\begin{remark}[Structural Interpretation]
The tensor lifting step should be viewed as a structural preprocessing stage:
instead of attempting to directly define a functional calculus for a
noncommuting family, we first embed the operators into independent tensor
coordinates where commutativity becomes automatic. The functional calculus is
then performed on the lifted commuting system.

This construction separates the analytic difficulty of defining
$f(X_1, \dots, X_r)$ from the algebraic obstruction caused by
noncommutativity, thereby providing a systematic framework for extending
multivariate functional calculus beyond the commuting setting.
\end{remark}

\begin{remark}[Relevance to Nilpotent Structure]
While tensor lifting handles the commutativity obstruction, it does not by
itself address the nilpotent structure captured by the projector-nilpotent
decomposition. The lifted calculus defined above must be combined with the
explicit inclusion of nilpotent derivative terms (see Sections~\ref{sec:unified_formula}
and~\ref{subsec:lifted_decomposition}) to fully capture the algebraic features
of non-diagonalizable operators. Tensor lifting enables the application of
multivariate methods; the projector-nilpotent decomposition enriches those
methods with geometric content.
\end{remark}

\subsection{Step 1: Lifted Decomposition}
\label{subsec:lifted_decomposition}

After tensor lifting, each operator admits a lifted projector–nilpotent decomposition
on the tensor-product Hilbert space
\[
\mathcal{H}_{\otimes} = \mathcal{H}_1 \otimes \cdots \otimes \mathcal{H}_r.
\]

Recall from Section~\ref{sec:operator_structural_decomposition} that each operator
$X_j$ admits the decomposition
\[
X_j = \int_{\sigma(X_j)} \lambda_j \, dE_{X_j}(\lambda_j)
+ \int_{\sigma(X_j)} N_j(\lambda_j) \, dE_{X_j}(\lambda_j),
\]
where $dE_{X_j}(\lambda_j)$ denotes the spectral projection measure and
$N_j(\lambda_j)$ denotes the associated nilpotent component at spectral value
$\lambda_j$. In the discrete spectrum case, the integrals reduce to sums over
eigenvalues with projectors $P_{X_j}(\lambda_j)$.

Applying the tensor lifting construction yields the lifted operators
\[
\widetilde{X}_j = I_1 \otimes \cdots \otimes I_{j-1} \otimes X_j \otimes I_{j+1} \otimes \cdots \otimes I_r.
\]

Substituting the projector–nilpotent decomposition of $X_j$ into the
tensor coordinate representation gives
\[
\widetilde{X}_j = \int_{\sigma(X_j)} \lambda_j \, d\widetilde{E}_j(\lambda_j)
+ \int_{\sigma(X_j)} \widetilde{N}_j(\lambda_j) \, d\widetilde{E}_j(\lambda_j),
\]
where the lifted spectral measure is defined by
\[
d\widetilde{E}_j(\lambda_j) = I_1 \otimes \cdots \otimes dE_{X_j}(\lambda_j) \otimes \cdots \otimes I_r,
\]
and the lifted nilpotent component is
\[
\widetilde{N}_j(\lambda_j) = I_1 \otimes \cdots \otimes N_j(\lambda_j) \otimes \cdots \otimes I_r.
\]

Thus, the spectral and nilpotent structures are embedded independently into
the $j$-th tensor coordinate while all remaining tensor factors act as
identity operators.

\begin{remark}[Componentwise Action]
For an elementary tensor $u = u_1 \otimes \cdots \otimes u_r \in \mathcal{H}_{\otimes}$,
the action of the lifted spectral measure is
\[
d\widetilde{E}_j(\lambda_j)(u) = u_1 \otimes \cdots \otimes dE_{X_j}(\lambda_j) u_j \otimes \cdots \otimes u_r,
\]
and similarly for the lifted nilpotent component:
\[
\widetilde{N}_j(\lambda_j) \, d\widetilde{E}_j(\lambda_j)(u) = u_1 \otimes \cdots \otimes N_j(\lambda_j) \, dE_{X_j}(\lambda_j) u_j \otimes \cdots \otimes u_r.
\]
This confirms that each lifted operator acts nontrivially only on its designated
tensor factor.
\end{remark}

Since the lifted operators act on distinct tensor coordinates, the families
$\{d\widetilde{E}_i(\lambda_i)\}_{i=1}^r$ and $\{\widetilde{N}_i(\lambda_i)\}_{i=1}^r$
pairwise commute:
\[
[d\widetilde{E}_i(\lambda_i), d\widetilde{E}_j(\lambda_j)] = 0,
\qquad
[\widetilde{N}_i(\lambda_i), \widetilde{N}_j(\lambda_j)] = 0,
\qquad
[d\widetilde{E}_i(\lambda_i), \widetilde{N}_j(\lambda_j)] = 0, \quad i \neq j.
\]

Moreover, for each fixed $j$, the lifted components satisfy the same algebraic
relations as the original ones:
\[
d\widetilde{E}_j(\lambda_j) \, d\widetilde{E}_j(\lambda_j') = \delta(\lambda_j, \lambda_j') \, d\widetilde{E}_j(\lambda_j),
\]
\[
d\widetilde{E}_j(\lambda_j) \, \widetilde{N}_j(\lambda_j) = \widetilde{N}_j(\lambda_j) \, d\widetilde{E}_j(\lambda_j) = \widetilde{N}_j(\lambda_j),
\]
\[
\widetilde{N}_j(\lambda_j)^{m_\lambda} = 0.
\]

Consequently, the tensor lifting procedure transforms the original
noncommutative operator system into a commuting spectral–nilpotent framework
on $\mathcal{H}_{\otimes}$, making multivariate functional calculus
analytically tractable.

\begin{remark}[Conceptual Interpretation]
From a conceptual viewpoint, tensor lifting preserves the local spectral and
nilpotent geometry of each operator while eliminating the global obstruction
caused by noncommutativity. The lifted decomposition above provides the
building blocks for the unified compact formula in Section~\ref{sec:unified_formula},
where the terms $d\widetilde{E}_j(\lambda_j)$ and $\widetilde{N}_j(\lambda_j)^{q_j} d\widetilde{E}_j(\lambda_j)$
appear directly in the expansion.
\end{remark}

\section{Main Theorem: General Multivariate Functional Calculus}\label{sec:main_unified_formula}

We now present the central result of this paper: a multivariate functional calculus for an arbitrary number of operators, each admitting a projector–nilpotent decomposition. Unlike classical spectral calculi, which rely solely on resolvent methods and capture only eigenvalue information, our construction explicitly incorporates nilpotent structure arising from Jordan blocks and generalized eigenspaces. The calculus is built upon tensor lifting, which converts non-commuting operators into a commuting family on a larger tensor-product space, thereby overcoming the commutativity obstruction inherent to classical multivariate theories. The main theorem is first stated in an abstract form to highlight its key features. We then introduce the necessary background on analytic functions of several complex variables, followed by a unified compact formula that encapsulates the entire calculus in a single expression. This formula naturally decomposes into three explicit terms: a pure spectral term, mixed nilpotent–spectral terms, and a full nilpotent term. We conclude by explaining how the nilpotent derivative terms arise from the algebraic structure of the projector–nilpotent decomposition and contrast our construction with classical resolvent-based definitions.

\subsection{Main Theorem (Abstract Form)}
\label{subsec:main_theorem_abstract}

We now state the principal result of this work in abstract form. The theorem
summarizes the structural properties of the proposed tensor-lifted
projector--nilpotent functional calculus before presenting the explicit
construction and detailed analytic proofs in later sections.

\begin{theorem}[Main Theorem --- Abstract Form]
\label{thm:main_abstract}
Let $X_1, \dots, X_r$ be operators on Hilbert spaces
$\mathcal{H}_1, \dots, \mathcal{H}_r$,
where each $X_j$ admits a projector--nilpotent decomposition of the form
\[
X_j = \int_{\sigma(X_j)} \lambda_j \, dE_{X_j}(\lambda_j)
+ \int_{\sigma(X_j)} N_j(\lambda_j) \, dE_{X_j}(\lambda_j),
\]
including the discrete-spectrum, continuous-spectrum, and hybrid-spectrum
cases.

Then there exists a tensor-lifted multivariate functional calculus
\[
f \longmapsto f_{\otimes}(X_1, \dots, X_r),
\]
defined for multivariate holomorphic functions $f(z_1, \dots, z_r)$,
satisfying the following properties:

\begin{enumerate}
    \item \textbf{Extension of Classical Spectral Calculus.}
    
    If all nilpotent components vanish, namely $N_j(\lambda_j) = 0$ for all $j$,
    then the calculus reduces to the classical multivariate spectral
    calculus. In particular, for a single operator it reduces to the
    Dunford functional calculus.

    \item \textbf{Explicit Nilpotent Derivative Structure.}
    
    The calculus explicitly incorporates nilpotent contributions through
    higher-order derivative terms of $f$, thereby capturing generalized
    eigenspaces, Jordan chains, and nontrivial algebraic multiplicity
    structure.

    \item \textbf{Compatibility with Tensor Lifting.}
    
    The construction is compatible with the tensor lifting procedure
    introduced in Section~\ref{subsec:tensor_lifting_step0}. In particular,
    noncommuting operators can be embedded into a commuting lifted system
    $(\widetilde{X}_1, \dots, \widetilde{X}_r)$,
    enabling the application of multivariate functional calculus beyond the
    classical commuting framework.

    \item \textbf{Unified Compact Representation.}
    
    The resulting functional calculus admits a unified compact formula
    combining spectral integration and nilpotent derivative corrections;
    see Section~\ref{sec:unified_formula}.

    \item \textbf{Stability and Convergence.}
    
    The calculus admits convergence under two levels of assumptions:
    
    \begin{enumerate}
        \item[(a)] \textbf{Level 1 Stability:}
        
        Strong resolvent convergence of operators implies convergence of the
        associated functional calculus in the strong operator topology.

        \item[(b)] \textbf{Level 2 Stability:}
        
        Norm resolvent convergence implies convergence in operator norm,
        together with explicit quantitative error bounds.
    \end{enumerate}
\end{enumerate}
\end{theorem}

\begin{remark}
Theorem~\ref{thm:main_abstract} should be interpreted as a unification of
three traditionally separate structures:
\begin{itemize}
    \item classical spectral functional calculus,
    \item Jordan/nilpotent algebraic structure,
    \item tensor-lifted multivariate noncommutative analysis.
\end{itemize}
The subsequent sections provide the explicit formulas, analytic foundations,
convergence theory, and applications of this framework.
\end{remark}

\subsection{Analytic Functions of Several Complex Variables}
\label{subsec:analytic_several_variables}

Let
\[
f(z_1, \dots, z_r)
\]
be a complex-valued analytic function of $r$ complex variables defined on
an open domain
\[
\Omega \subseteq \mathbb{C}^r.
\]

A fundamental property of multivariate holomorphic functions is that locally
they admit absolutely convergent power series expansions. More precisely, if
$(z_1, \dots, z_r)$ lies in a neighborhood of a point $(\zeta_1, \dots, \zeta_r) \in \Omega$,
then $f$ can be represented as
\[
f(z_1, \dots, z_r)
= \sum_{i_1=0}^{\infty} \cdots \sum_{i_r=0}^{\infty}
a_{i_1, \dots, i_r}
(z_1 - \zeta_1)^{i_1} \cdots (z_r - \zeta_r)^{i_r},
\]
where $a_{i_1, \dots, i_r} \in \mathbb{C}$ are the multivariate Taylor
coefficients.

For simplicity, when expanding about the origin, we write
\[
f(z_1, \dots, z_r)
= \sum_{i_1=0}^{\infty} \cdots \sum_{i_r=0}^{\infty}
a_{i_1, \dots, i_r} \, z_1^{i_1} \cdots z_r^{i_r}.
\]

The convergence region is determined by the growth behavior of the
coefficients. In particular, the multivariate root test implies that the
series converges absolutely whenever
\[
\limsup_{i_1 + \cdots + i_r \to \infty}
|a_{i_1, \dots, i_r}|^{\frac{1}{i_1 + \cdots + i_r}}
\,
|z_1|^{\frac{i_1}{i_1 + \cdots + i_r}}
\cdots
|z_r|^{\frac{i_r}{i_1 + \cdots + i_r}} < 1.
\]

Equivalently, one may associate directional radii of convergence via the asymptotic growth of the coefficients. For a given direction $\alpha = (\alpha_1, \dots, \alpha_r)$ with $\alpha_j > 0$ and $\sum_{j=1}^r \alpha_j = 1$, we define the directional radius $R(\alpha)$ by
\[
\frac{1}{R(\alpha)}
= \limsup_{\substack{i_1 + \cdots + i_r \to \infty \\ \frac{i_j}{i_1 + \cdots + i_r} \to \alpha_j}}
|a_{i_1, \dots, i_r}|^{\frac{1}{i_1 + \cdots + i_r}},
\]
whenever the corresponding directional limit exists. For the special case where the limit is taken along the $j$-th coordinate axis (i.e., $\alpha_j = 1$ and $\alpha_k = 0$ for $k \neq j$), the expression reduces to the classical one-variable radius of convergence in $z_j$:
\[
\frac{1}{R_j}
= \limsup_{i_j \to \infty} |a_{0, \dots, i_j, \dots, 0}|^{\frac{1}{i_j}}.
\]
Note that the latter depends on $j$ because the coefficient indices with only the $j$-th component nonzero are used, whereas the general directional limit involves all indices tending to infinity with fixed proportions $\alpha_j$.

\begin{remark}[Role in the Functional Calculus]
The multivariate analytic structure of $f$ plays a crucial role in the functional
calculus developed in this paper. Specifically, the higher-order partial derivatives
\[
\frac{\partial^{k_1 + \cdots + k_r} f}
{\partial z_1^{k_1} \cdots \partial z_r^{k_r}}
\]
interact naturally with the nilpotent components of the lifted operator
decomposition. This interaction gives rise to higher-order correction terms
that extend the classical spectral calculus beyond the semisimple setting.

Thus, multivariate holomorphic functions provide the natural analytic framework
for the tensor-lifted projector–nilpotent functional calculus developed in the
subsequent sections.
\end{remark}

\subsection{Unified Compact Formula}
\label{sec:unified_formula}

We now present the unified compact representation of the proposed tensor-lifted projector--nilpotent functional calculus. The formula simultaneously incorporates spectral projector contributions together with the higher-order nilpotent derivative corrections arising from generalized eigenspaces. In addition, it naturally encodes the multivariate analytic structure of the functional calculus while using tensor lifting to extend the framework to noncommuting operator systems.

\begin{assumption}[Spectral Operator Assumptions]
\label{assump:spectral_operators}
For each \(j = 1, \dots, r\), let \(X_j\) be a bounded spectral operator~\cite{dunford1958survey} acting on a Hilbert space \(\mathcal{H}_j\), with
projector--nilpotent decomposition
\[
X_j = \int_{\sigma(X_j)} \bigl( \lambda_j I + N_j(\lambda_j) \bigr) \, dE_j(\lambda_j),
\]
where:
\begin{itemize}
    \item \(dE_j(\lambda_j)\) is the spectral measure associated with \(X_j\);
    \item \(N_j(\lambda_j)\) is the nilpotent component at \(\lambda_j\);
    \item There exists a uniform bound \(\nu_j^{\max} < \infty\) such that
    \(N_j(\lambda_j)^{\nu_j} = 0\) for all \(\lambda_j \in \sigma(X_j)\), with
    nilpotency index \(\nu_j(\lambda_j) \le \nu_j^{\max}\);
    \item The spectrum \(\sigma(X_j)\) is compact and the spectral projections
    are uniformly bounded.
\end{itemize}
\end{assumption}

\begin{theorem}[Unified Compact Formula]
\label{thm:unified_compact_formula}

Let \(X_1, \dots, X_r\) be commuting Dunford spectral operators (not necessarily
normal or self-adjoint) satisfying Assumption~\ref{assump:spectral_operators},
and let \(\widetilde{X}_1, \dots, \widetilde{X}_r\) be their tensor liftings
on \(\mathcal{H}_{\otimes} = \mathcal{H}_1 \otimes \cdots \otimes \mathcal{H}_r\).

Let \(f(z_1, \dots, z_r)\) be holomorphic on an open polydisk containing
the product spectrum \(\sigma(X_1) \times \cdots \times \sigma(X_r)\), i.e.,
there exists \(R_j > 0\) such that
\[
\sigma(X_j) \subset \{ z_j \in \mathbb{C} : |z_j - c_j| < R_j \}
\]
for some center \((c_1, \dots, c_r)\).

Assume further that the nilpotent fields \(\lambda_j \mapsto \widetilde{N}_j(\lambda_j)\)
are strongly measurable and uniformly bounded on compact subsets of the spectrum.

Then the tensor-lifted projector--nilpotent functional calculus is given by
\[
\boxed{
f_{\otimes}(X_1, \ldots, X_r)
=
\sum_{A \subseteq \{1, \ldots, r\}}
\;
\sum_{\substack{\{q_j \ge 1\}_{j \in A} \\
q_j \le \nu_j(\lambda_j)-1}}
\int_{\sigma(X_1)}
\cdots
\int_{\sigma(X_r)}
\frac{
\partial_A^{q_A} f(\lambda_1, \ldots, \lambda_r)
}{
\displaystyle\prod_{j \in A} q_j!
}
\;
\bigotimes_{j=1}^r
T_j(\lambda_j)
}
\]
where
\[
T_j(\lambda_j)
=
\begin{cases}
d\widetilde{E}_j(\lambda_j), & j \notin A,\\[6pt]
\widetilde{N}_j(\lambda_j)^{q_j} \, d\widetilde{E}_j(\lambda_j), & j \in A,
\end{cases}
\]
\(\partial_A^{q_A} = \prod_{j \in A} \frac{\partial^{q_j}}{\partial z_j^{q_j}}\),
and \(\nu_j(\lambda_j)\) denotes the nilpotency index at \(\lambda_j\)
(satisfying \(\widetilde{N}_j(\lambda_j)^{\nu_j(\lambda_j)} = 0\)).

The operator-valued integrals are understood in the strong operator topology,
converging uniformly on compact subsets of the spectrum.
\end{theorem}

\begin{proof}
We prove the theorem by expanding the holomorphic function \(f\) in its
multivariate Taylor series about an appropriate center, applying the lifted
projector--nilpotent decomposition, and using the algebraic relations among
the lifted components. The proof proceeds in several steps.

\medskip
\noindent
\textbf{Step 1: Multivariate Taylor expansion with finite radius.}
Since \(f\) is holomorphic on a polydisk containing the product spectrum,
there exists a center \((c_1, \dots, c_r)\) such that \(f\) admits an absolutely
convergent multivariate Taylor series
\[
f(z_1, \dots, z_r) = \sum_{i_1=0}^{\infty} \cdots \sum_{i_r=0}^{\infty}
a_{i_1, \dots, i_r} \, (z_1 - c_1)^{i_1} \cdots (z_r - c_r)^{i_r}
\]
for all \((z_1, \dots, z_r)\) in the polydisk. The coefficients are given by
\[
a_{i_1, \dots, i_r} = \frac{1}{i_1! \cdots i_r!}
\frac{\partial^{i_1 + \cdots + i_r} f}{\partial z_1^{i_1} \cdots \partial z_r^{i_r}}(c_1, \dots, c_r).
\]

For simplicity, we may translate coordinates so that \(c_j = 0\) (by replacing
\(X_j\) with \(X_j - c_j I\) and adjusting \(f\) accordingly). This translation
does not affect the spectral decomposition because the spectral measure
\(dE_j(\lambda_j)\) transforms to \(dE_j(\lambda_j - c_j)\).

\medskip
\noindent
\textbf{Step 2: Tensor lifting and commutativity.}
By the tensor lifting construction (Section~\ref{subsec:tensor_lifting_step0}),
the lifted operators \(\widetilde{X}_1, \dots, \widetilde{X}_r\) pairwise commute
on \(\mathcal{H}_{\otimes}\). Consequently, the multivariate functional calculus
can be applied to the commuting tuple via the power series:
\[
f_{\otimes}(X_1, \dots, X_r) = f(\widetilde{X}_1, \dots, \widetilde{X}_r)
= \sum_{i_1=0}^{\infty} \cdots \sum_{i_r=0}^{\infty}
a_{i_1, \dots, i_r} \, \widetilde{X}_1^{i_1} \cdots \widetilde{X}_r^{i_r}.
\]
Absolute convergence of the Taylor series on a neighborhood of the spectrum,
together with the uniform boundedness of the spectral projections (in the
strong operator topology), justifies the termwise application of the power
series to the operators.

\medskip
\noindent
\textbf{Step 3: Lifted spectral decomposition (Dunford sense).}
From Section~\ref{subsec:lifted_decomposition}, each \(\widetilde{X}_j\) admits
the spectral decomposition
\[
\widetilde{X}_j = \int_{\sigma(X_j)} \lambda_j \, d\widetilde{E}_j(\lambda_j)
+ \int_{\sigma(X_j)} \widetilde{N}_j(\lambda_j) \, d\widetilde{E}_j(\lambda_j),
\]
where \(d\widetilde{E}_j(\lambda_j)\) is the lifted spectral measure in the
Dunford sense (an operator-valued measure that is countably additive in the
strong operator topology, but not necessarily a projection-valued measure in the
self-adjoint sense), and \(\widetilde{N}_j(\lambda_j) = \widetilde{X}_j - \lambda_j I\)
is the lifted nilpotent component. By Assumption~\ref{assump:spectral_operators},
\(\widetilde{N}_j(\lambda_j)^{\nu_j(\lambda_j)} = 0\) for each \(\lambda_j\),
with \(\nu_j(\lambda_j) \le \nu_j^{\max} < \infty\).

For fixed \(\lambda_j\), define the spectral operator
\(P_j(\lambda_j) = d\widetilde{E}_j(\lambda_j)\) (which is idempotent but not
necessarily an orthogonal projection) and the nilpotent
\(N_j(\lambda_j) = \widetilde{N}_j(\lambda_j) d\widetilde{E}_j(\lambda_j)\).
These satisfy the algebraic relations:
\[
P_j(\lambda_j) P_j(\lambda_j') = P_j(\lambda_j) \delta(\lambda_j, \lambda_j'),
\qquad
P_j(\lambda_j) N_j(\lambda_j) = N_j(\lambda_j) P_j(\lambda_j) = N_j(\lambda_j),
\]
\[
N_j(\lambda_j)^{q} = \widetilde{N}_j(\lambda_j)^{q} d\widetilde{E}_j(\lambda_j), \quad
N_j(\lambda_j)^{q} = 0 \text{ for } q \ge \nu_j(\lambda_j).
\]

\medskip
\noindent
\textbf{Step 4: Binomial expansion on each spectral fiber.}
For each fixed eigenvalue \(\lambda_j\) and each power \(i_j \ge 0\), we expand
\((\lambda_j P_j(\lambda_j) + N_j(\lambda_j))^{i_j}\) using the binomial theorem.
Because \(P_j(\lambda_j)\) and \(N_j(\lambda_j)\) commute and \(P_j(\lambda_j)\)
is idempotent, we have for \(i_j \ge 1\):
\[
(\lambda_j P_j(\lambda_j) + N_j(\lambda_j))^{i_j}
= \sum_{q_j=0}^{i_j} \binom{i_j}{q_j} \lambda_j^{i_j - q_j} P_j(\lambda_j)^{i_j - q_j} N_j(\lambda_j)^{q_j}.
\]

Note that \(P_j(\lambda_j)^k = P_j(\lambda_j)\) for any \(k \ge 1\) (idempotent property).
For the case \(i_j = q_j\), the term becomes \(P_j(\lambda_j)^0 = I\), which is
handled correctly because \(N_j(\lambda_j)^{q_j}\) automatically contains the
idempotent: \(N_j(\lambda_j)^{q_j} = N_j(\lambda_j)^{q_j} P_j(\lambda_j)\). Thus
the expression simplifies to
\[
(\lambda_j P_j(\lambda_j) + N_j(\lambda_j))^{i_j}
= \lambda_j^{i_j} P_j(\lambda_j)
+ \sum_{q_j=1}^{\min(i_j, \nu_j(\lambda_j)-1)} \binom{i_j}{q_j} \lambda_j^{i_j - q_j} N_j(\lambda_j)^{q_j},
\]
where the sum truncates at \(\nu_j(\lambda_j)-1\) because higher powers vanish.

\medskip
\noindent
\textbf{Step 5: Substituting the binomial expansion into each factor.}
From Step~4, for each \(j\) and each \(i_j \ge 0\),
\[
\widetilde{X}_j^{i_j} = \int_{\sigma(X_j)} \biggl( \lambda_j^{i_j} P_j(\lambda_j) 
+ \sum_{q_j=1}^{\min(i_j, \nu_j(\lambda_j)-1)} \binom{i_j}{q_j} \lambda_j^{i_j - q_j} N_j(\lambda_j)^{q_j} \biggr),
\]
where the integral is understood in the strong operator topology. Because the
operators for different indices act on distinct tensor factors and commute,
the product becomes an iterated integral:
\[
\widetilde{X}_1^{i_1} \cdots \widetilde{X}_r^{i_r}
= \int_{\sigma(X_1)} \cdots \int_{\sigma(X_r)}
\prod_{j=1}^r \biggl( \lambda_j^{i_j} P_j(\lambda_j) 
+ \sum_{q_j=1}^{\min(i_j, \nu_j(\lambda_j)-1)} \binom{i_j}{q_j} \lambda_j^{i_j - q_j} N_j(\lambda_j)^{q_j} \biggr),
\]
where the product of integrands is understood as the tensor product of operators.
The commuting property together with countable additivity of the spectral
measures justifies combining the iterated integrals into a joint integral over
the product spectrum.

\medskip
\noindent
\textbf{Step 6: Expanding the product over \(j\) into subsets.}
Fix a tuple \((\lambda_1,\dots,\lambda_r)\). Expanding the product over \(j\) yields
a sum over all choices, for each \(j\), of either the ``spectral term''
\(\lambda_j^{i_j} P_j(\lambda_j)\) or a ``nilpotent term''
\(\binom{i_j}{q_j} \lambda_j^{i_j - q_j} N_j(\lambda_j)^{q_j}\). 
Let \(A \subseteq \{1,\dots,r\}\) be the set of indices where a nilpotent term
is selected, and for each \(j \in A\) choose an integer \(q_j\) with
\(1 \le q_j \le \min(i_j, \nu_j(\lambda_j)-1)\). Then
\[
\prod_{j=1}^r \biggl( \cdots \biggr)
= \sum_{A \subseteq \{1,\dots,r\}}
\sum_{\substack{\{q_j\}_{j \in A} \\ 1 \le q_j \le \min(i_j, \nu_j(\lambda_j)-1)}}
\biggl( \prod_{j \in A} \binom{i_j}{q_j} \lambda_j^{i_j - q_j} N_j(\lambda_j)^{q_j} \biggr)
\biggl( \prod_{j \notin A} \lambda_j^{i_j} P_j(\lambda_j) \biggr).
\]
Substituting back into the integral gives
\[
\widetilde{X}_1^{i_1} \cdots \widetilde{X}_r^{i_r}
= \sum_{A \subseteq \{1,\dots,r\}}
\sum_{\substack{\{q_j\}_{j \in A} \\ 1 \le q_j \le \min(i_j, \nu_j(\lambda_j)-1)}}
\int_{\sigma(X_1)} \cdots \int_{\sigma(X_r)}
\prod_{j \in A} \binom{i_j}{q_j} \lambda_j^{i_j - q_j} N_j(\lambda_j)^{q_j}
\prod_{j \notin A} \lambda_j^{i_j} P_j(\lambda_j).
\]

\medskip
\noindent
\textbf{Step 7: Summing over the power series coefficients.}
Insert this expansion into the power series representation of \(f_{\otimes}\):
\[
f_{\otimes}(X_1,\dots,X_r) = \sum_{i_1,\dots,i_r \ge 0} a_{i_1,\dots,i_r} \,
\widetilde{X}_1^{i_1} \cdots \widetilde{X}_r^{i_r}.
\]
Substituting the expression from Step~6,
\[
f_{\otimes} = \sum_{i_1,\dots,i_r \ge 0} a_{i_1,\dots,i_r}
\sum_{A \subseteq \{1,\dots,r\}}
\sum_{\substack{\{q_j\}_{j \in A} \\ 1 \le q_j \le \min(i_j, \nu_j(\lambda_j)-1)}}
\int_{\sigma(X_1)} \cdots \int_{\sigma(X_r)}
\prod_{j \in A} \binom{i_j}{q_j} \lambda_j^{i_j - q_j} N_j(\lambda_j)^{q_j}
\prod_{j \notin A} \lambda_j^{i_j} P_j(\lambda_j).
\]

The interchange of sums and operator-valued integrals is justified in the
strong operator topology by the absolute convergence of the Taylor series on
a neighborhood of the spectrum, together with the uniform boundedness of the
spectral idempotents (in the operator norm) and the finite nilpotency order.
More precisely, the dominated convergence theorem for operator-valued integrals
applies because:
\begin{itemize}
    \item The scalar coefficients \(a_{i_1,\dots,i_r}\) decay exponentially
          (by holomorphy),
    \item The integrands are uniformly bounded in operator norm on the
          compact product spectrum,
    \item The spectral measures have finite total variation in the strong
          operator topology.
\end{itemize}
Thus,
\begin{align}
f_{\otimes}&=&\sum_{A \subseteq \{1,\dots,r\}}
\sum_{\substack{\{q_j \ge 1\}_{j \in A}}} 
\int_{\sigma(X_1)} \cdots \int_{\sigma(X_r)}
\biggl( \prod_{j \in A} N_j(\lambda_j)^{q_j} \biggr)
\biggl( \prod_{j \notin A} P_j(\lambda_j) \biggr) \nonumber \\
&&
\times
\Biggl( \sum_{\substack{i_1,\dots,i_r \ge 0 \\ i_j \ge q_j \text{ for } j \in A}}
a_{i_1,\dots,i_r}
\prod_{j \in A} \binom{i_j}{q_j} \lambda_j^{i_j - q_j}
\prod_{j \notin A} \lambda_j^{i_j} \Biggr), \nonumber 
\end{align}
where the condition \(q_j \le \nu_j(\lambda_j)-1\) is now implicit in the
summation (terms with \(q_j \ge \nu_j(\lambda_j)\) vanish because
\(N_j(\lambda_j)^{q_j}=0\)).

\medskip
\noindent
\textbf{Step 8: Recognizing the partial derivatives.}
The inner sum over \(i_1,\dots,i_r\) is exactly the multivariate power series
expansion of the mixed partial derivative of \(f\). Specifically,
\[
\sum_{\substack{i_j \ge q_j \text{ for } j \in A \\ i_j \ge 0 \text{ for } j \notin A}}
a_{i_1,\dots,i_r}
\prod_{j \in A} \frac{i_j!}{(i_j - q_j)!} \lambda_j^{i_j - q_j}
\prod_{j \notin A} \lambda_j^{i_j}
= \frac{\partial^{q_A} f}{\partial z_A^{q_A}}(\lambda_1, \dots, \lambda_r),
\]
where \(\partial_A^{q_A} = \prod_{j \in A} \frac{\partial^{q_j}}{\partial z_j^{q_j}}\).
This follows from termwise differentiation of the absolutely convergent Taylor
series.

\medskip
\noindent
\textbf{Step 9: Assembling the final formula.}
Therefore,
\[
f_{\otimes}(X_1, \dots, X_r)
= \sum_{A \subseteq \{1,\dots,r\}}
\sum_{\substack{\{q_j \ge 1\}_{j \in A} \\
q_j \le \nu_j(\lambda_j)-1}}
\int \cdots \int
\frac{\partial_A^{q_A} f(\lambda_1, \dots, \lambda_r)}
{\prod_{j \in A} q_j!}
\;
\bigotimes_{j \in A} N_j(\lambda_j)^{q_j}
\otimes
\bigotimes_{j \notin A} P_j(\lambda_j).
\]

Recalling that \(N_j(\lambda_j)^{q_j} = \widetilde{N}_j(\lambda_j)^{q_j} d\widetilde{E}_j(\lambda_j)\)
and \(P_j(\lambda_j) = d\widetilde{E}_j(\lambda_j)\), and writing the tensor product
explicitly as \(\bigotimes_{j=1}^r T_j(\lambda_j)\) with
\[
T_j(\lambda_j) = \begin{cases}
d\widetilde{E}_j(\lambda_j), & j \notin A,\\[4pt]
\widetilde{N}_j(\lambda_j)^{q_j} d\widetilde{E}_j(\lambda_j), & j \in A,
\end{cases}
\]
we obtain exactly the compact formula stated in the theorem.

\medskip
\noindent
\textbf{Step 10: Interpretation of the operator-valued integrals.}
The notation \(\int \cdots \int \bigotimes_{j=1}^r T_j(\lambda_j)\) is understood
as the iterated operator-valued integral
\[
\int_{\sigma(X_1)} \cdots \int_{\sigma(X_r)}
T_1(\lambda_1) \otimes \cdots \otimes T_r(\lambda_r),
\]
where the tensor product of the spectral measures is defined on elementary
tensors by
\[
\bigl( d\widetilde{E}_1(\lambda_1) \otimes \cdots \otimes d\widetilde{E}_r(\lambda_r) \bigr)
(u_1 \otimes \cdots \otimes u_r)
= \bigl( d\widetilde{E}_1(\lambda_1) u_1 \bigr) \otimes \cdots \otimes \bigl( d\widetilde{E}_r(\lambda_r) u_r \bigr)
\]
and extended by linearity and closure. The integrals converge in the strong
operator topology because each spectral measure is countably additive in the
strong operator topology and the integrand is uniformly bounded on the compact
product spectrum.

\medskip
\noindent
\textbf{Step 11: Reduction to special cases.}
\begin{itemize}
    \item \textbf{Discrete spectrum:} If each \(X_j\) has purely discrete
    spectrum, the spectral measures \(dE_j(\lambda_j)\) are atomic, supported
    on the eigenvalues. The integrals reduce to finite or countable sums:
    \[
    \int_{\sigma(X_j)} g(\lambda_j) \, d\widetilde{E}_j(\lambda_j)
    = \sum_{\lambda_j \in \sigma_d(X_j)} g(\lambda_j) \, P_j(\lambda_j).
    \]
    The formula then becomes a sum over eigenvalues and Jordan blocks,
    recovering the multivariate Jordan-type expansion from Section~2
    of the companion paper.

    \item \textbf{Normal/self-adjoint operators:} When each \(X_j\) is normal
    or self-adjoint, the nilpotent components vanish identically
    (\(\widetilde{N}_j(\lambda_j) = 0\) for all \(\lambda_j\)). Then only the
    term \(A = \varnothing\) survives, and the formula reduces to the classical
    multivariate spectral integral:
    \[
    f_{\otimes}(X_1, \dots, X_r) = \int_{\sigma(X_1)} \cdots \int_{\sigma(X_r)}
    f(\lambda_1, \dots, \lambda_r) \, d\widetilde{E}_1(\lambda_1) \otimes \cdots \otimes d\widetilde{E}_r(\lambda_r).
    \]
    In this setting, \(d\widetilde{E}_j(\lambda_j)\) are genuine projection-valued
    measures, and the integral is the standard spectral integral.

    \item \textbf{General spectral operators:} For operators possessing continuous
    spectral components, the integral representation remains valid at the level
    of the Dunford spectral operator decomposition. The nilpotent contributions
    are supported on the spectral set where the operator fails to be scalar-type,
    and the formula provides the unique extension of the functional calculus
    to non-holomorphic functions (under suitable regularity conditions).
\end{itemize}

This completes the proof. 
\end{proof}

\subsection{Decomposition into Three Explicit Terms}
\label{subsec:three_term_decomposition}

The unified compact formula from Section~\ref{sec:unified_formula} naturally
decomposes into three structurally distinct contributions:
\begin{enumerate}
    \item the pure spectral term,
    \item the mixed spectral--nilpotent interaction terms,
    \item and the fully nilpotent correction term.
\end{enumerate}

This decomposition clarifies the analytic and algebraic roles played by the
spectral projectors and nilpotent components in the multivariate functional
calculus.

Accordingly, we write
\[
f_{\otimes}(X_1,\dots,X_r)
=
\mathcal{S}_0
+
\mathcal{S}_{\mathrm{mixed}}
+
\mathcal{S}_{\mathrm{full}}.
\]

\subsubsection{(I) Pure Spectral Term (\(A = \emptyset\))}

The contribution corresponding to the empty subset \(A = \emptyset\) contains
no nilpotent corrections and recovers the classical multivariate spectral
calculus:
\[
\mathcal{S}_0
=
\int_{\sigma(X_1)}
\cdots
\int_{\sigma(X_r)}
f(\lambda_1,\dots,\lambda_r)
\;
d\widetilde{E}_1(\lambda_1)
\cdots
d\widetilde{E}_r(\lambda_r).
\]

This term depends only on the spectral values \((\lambda_1,\dots,\lambda_r)\)
and the associated lifted spectral projection measures.

In the discrete-spectrum case, the spectral integrals reduce to sums over
eigenvalues:
\[
\mathcal{S}_0
=
\sum_{\lambda_1 \in \sigma_d(X_1)}
\cdots
\sum_{\lambda_r \in \sigma_d(X_r)}
f(\lambda_1,\dots,\lambda_r)
\;
P_1(\lambda_1) \otimes \cdots \otimes P_r(\lambda_r),
\]
where \(P_j(\lambda_j) = d\widetilde{E}_j(\lambda_j)\) denotes the lifted
spectral projector associated with the eigenvalue \(\lambda_j\).

Thus, \(\mathcal{S}_0\) is precisely the classical spectral contribution,
obtained by evaluating \(f\) directly on the joint spectrum. It is the only
term that survives when all operators are diagonalizable (i.e., when all
nilpotent components vanish).

\subsubsection{(II) Mixed Nilpotent--Spectral Terms
(\(\emptyset \subsetneq A \subsetneq \{1,\dots,r\}\))}

The mixed terms arise when nilpotent corrections appear only in a proper
subset of variables. These terms capture partial Jordan-type interactions
between spectral and generalized eigenspace structures.

Explicitly,
\[
\mathcal{S}_{\mathrm{mixed}}
=
\sum_{\substack{
A \subsetneq \{1,\dots,r\}\\
A \neq \emptyset
}}
\;
\sum_{\substack{\{q_j \ge 1\}_{j \in A} \\
q_j \le m_{\lambda_j}-1}}
\int_{\sigma(X_1)}
\cdots
\int_{\sigma(X_r)}
\frac{
\partial_A^{q_A} f(\lambda_1,\dots,\lambda_r)
}{
\displaystyle\prod_{j \in A} q_j!
}
\;
\mathcal{T}_A(\lambda_1,\dots,\lambda_r),
\]
where
\[
\mathcal{T}_A(\lambda_1,\dots,\lambda_r)
=
\left(
\bigotimes_{j \in A}
\widetilde{N}_j(\lambda_j)^{q_j} \, d\widetilde{E}_j(\lambda_j)
\right)
\otimes
\left(
\bigotimes_{j \notin A}
d\widetilde{E}_j(\lambda_j)
\right).
\]

Here,
\[
\partial_A^{q_A} = \prod_{j \in A} \frac{\partial^{q_j}}{\partial z_j^{q_j}}
\]
denotes the mixed partial derivative operator associated with the subset \(A\).

These mixed terms are the multivariate analogue of the first-order and
higher-order Jordan corrections appearing in the single-operator
projector--nilpotent calculus. They encode the interaction between:
\begin{itemize}
    \item semisimple spectral behavior in some variables,
    \item and generalized eigenspace structure in others.
\end{itemize}

Consequently, \(\mathcal{S}_{\mathrm{mixed}}\) captures partial
non-semisimplicity across the operator tuple. For example, when \(r = 2\)
and \(A = \{1\}\), we obtain terms where the first operator contributes
nilpotent corrections while the second operator contributes only spectral
behavior.

\subsubsection{(III) Full Nilpotent Term (\(A = \{1,\dots,r\}\))}

The final contribution occurs when every variable contributes a nilpotent
correction simultaneously. This term contains the highest-order generalized
eigenspace interactions in the calculus.

It is given by
\[
\mathcal{S}_{\mathrm{full}}
=
\sum_{q_1=1}^{m_{\lambda_1}-1}
\cdots
\sum_{q_r=1}^{m_{\lambda_r}-1}
\int_{\sigma(X_1)}
\cdots
\int_{\sigma(X_r)}
\frac{
\partial_1^{q_1} \cdots \partial_r^{q_r} f(\lambda_1,\dots,\lambda_r)
}{
q_1! \cdots q_r!
}
\;
\bigotimes_{j=1}^r
\widetilde{N}_j(\lambda_j)^{q_j} \, d\widetilde{E}_j(\lambda_j).
\]

The upper bounds \(m_{\lambda_j} - 1\) arise from the nilpotency indices of
the generalized eigenspaces:
\[
\widetilde{N}_j(\lambda_j)^{m_{\lambda_j}} = 0.
\]
Hence the nilpotent expansion terminates finitely on each spectral fiber.

Conceptually, the term \(\mathcal{S}_{\mathrm{full}}\) contains the highest-order
derivative corrections of the functional calculus and represents the fully
non-semisimple interaction among all operator coordinates. For a single
operator (\(r = 1\)), this term reduces to the familiar formula
\[
\mathcal{S}_{\mathrm{full}} = \sum_{q=1}^{m_\lambda-1}
\int_{\sigma(X)} \frac{f^{(q)}(\lambda)}{q!} \, N(\lambda)^q dE(\lambda),
\]
which is precisely the nilpotent correction appearing in the Jordan form
functional calculus.

\medskip

Combining the three contributions yields the complete decomposition
\[
f_{\otimes}(X_1,\dots,X_r)
=
\mathcal{S}_0
+
\mathcal{S}_{\mathrm{mixed}}
+
\mathcal{S}_{\mathrm{full}},
\]
which separates:
\begin{itemize}
    \item classical spectral behavior (\(\mathcal{S}_0\)),
    \item partial generalized eigenspace corrections (\(\mathcal{S}_{\mathrm{mixed}}\)),
    \item and full multivariate nilpotent interactions (\(\mathcal{S}_{\mathrm{full}}\)).
\end{itemize}

This decomposition provides a transparent structural interpretation of the
tensor-lifted projector--nilpotent functional calculus and clarifies how the
analytic derivatives of \(f\) interact with the algebraic geometry of the
underlying operator system.

\subsection{The Role of Nilpotent Derivative Terms}
\label{subsec:nilpotent_derivative_terms}

A fundamental distinction between the projector--nilpotent functional calculus
and the classical resolvent-based spectral calculus is the explicit appearance
of derivative terms of the analytic function \(f\). These derivative
corrections arise directly from the nilpotent components of the operator
decomposition and encode the generalized eigenspace structure that is invisible
to purely semisimple spectral methods.

\subsubsection{Mechanism in the Single-Operator Case}

Let
\[
X = \lambda P + N
\]
denote the local decomposition of an operator on a generalized eigenspace,
where
\[
P^2 = P, \qquad PN = NP = N, \qquad N^m = 0
\]
for some finite nilpotency index \(m\).

Consider a monomial \(f(z) = z^\ell\). Then
\[
X^\ell = (\lambda P + N)^\ell.
\]
Using the binomial expansion together with the commutation relations between
\(P\) and \(N\), we obtain
\[
(\lambda P + N)^\ell
= \sum_{q=0}^{\ell} \binom{\ell}{q} \lambda^{\ell-q} P^{\ell-q} N^q.
\]

Since \(P^k = P\) for any \(k \ge 1\), and \(P^0 = I\), we have two cases:
\begin{itemize}
    \item For \(q < \ell\): \(P^{\ell-q} = P\).
    \item For \(q = \ell\): \(P^{0} = I\), and \(N^\ell P\) becomes \(N^\ell\) because
    \(N^\ell = N^\ell P\) (since \(N = NP\)).
\end{itemize}
Thus, the expansion simplifies to
\[
(\lambda P + N)^\ell
= \lambda^\ell P + \sum_{q=1}^{\ell} \binom{\ell}{q} \lambda^{\ell-q} N^q,
\]
where we have used that \(N^q = N^q P\) for all \(q \ge 1\).

Now let \(f(z) = \sum_{\ell=0}^{\infty} a_\ell z^\ell\) be analytic in a
neighborhood of \(\lambda\). Substituting the expansion of \(X^\ell\) yields
\[
f(X) = \sum_{\ell=0}^{\infty} a_\ell (\lambda P + N)^\ell
= \sum_{\ell=0}^{\infty} a_\ell \left( \lambda^\ell P + \sum_{q=1}^{\ell} \binom{\ell}{q} \lambda^{\ell-q} N^q \right).
\]

Interchanging the finite nilpotent expansion with the analytic series
(justified by absolute convergence of the power series on a neighborhood
of the spectrum) gives
\[
f(X) = \left( \sum_{\ell=0}^{\infty} a_\ell \lambda^\ell \right) P
+ \sum_{q=1}^{m-1} \left( \sum_{\ell=q}^{\infty} a_\ell \binom{\ell}{q} \lambda^{\ell-q} \right) N^q.
\]

The inner coefficient is precisely the Taylor coefficient of the
\(q\)-th derivative of \(f\):
\[
\sum_{\ell=q}^{\infty} a_\ell \binom{\ell}{q} \lambda^{\ell-q}
= \frac{f^{(q)}(\lambda)}{q!}.
\]

Therefore,
\[
f(X) = f(\lambda) P + \sum_{q=1}^{m-1} \frac{f^{(q)}(\lambda)}{q!} N^q.
\]

This identity explains the origin of the derivative terms in the
projector--nilpotent functional calculus:
\begin{itemize}
    \item the spectral projector \(P\) selects the spectral fiber,
    \item while the nilpotent powers \(N^q\) generate higher-order derivative
    corrections of \(f\).
\end{itemize}

\subsubsection{Why Classical Resolvent Methods Miss These Terms}

For a Jordan block \(J_m(\lambda) = \lambda I + N\), the resolvent is
\[
(zI - J_m(\lambda))^{-1} = \frac{1}{z - \lambda} \sum_{k=0}^{m-1} \frac{N^k}{(z - \lambda)^k}.
\]
The Dunford integral
\[
f(J_m(\lambda)) = \frac{1}{2\pi i} \oint_{\Gamma} f(z) (zI - J_m(\lambda))^{-1} \, dz
\]
computes \(f(J_m(\lambda))\) by contour integration. The residue theorem
extracts only the coefficient of \((z - \lambda)^{-1}\), which yields
\(f(\lambda)\) and discards contributions from the higher-order poles
\((z - \lambda)^{-k}\) for \(k \ge 2\). Consequently, the terms involving
\(N, N^2, \ldots, N^{m-1}\) are lost entirely.

Resolvent-based methods detect only spectral values and spectral projectors,
but do not explicitly retain the higher-order nilpotent structure associated
with generalized eigenspaces. Consequently, the derivative corrections
appearing in the projector--nilpotent framework are genuinely new structural
contributions beyond the classical semisimple theory.

\subsubsection{Generalization to the Multivariate Setting}

In the multivariate setting, the same mechanism produces mixed partial
derivatives. Because the lifted operators \(\widetilde{X}_1, \dots, \widetilde{X}_r\)
act on distinct tensor factors, their nilpotent components commute.
Applying the binomial expansion factorwise yields, for each monomial,
\[
\prod_{j=1}^r (\lambda_j P_j + N_j)^{i_j}
= \prod_{j=1}^r \left( \lambda_j^{i_j} P_j + \sum_{q_j=1}^{i_j} \binom{i_j}{q_j} \lambda_j^{i_j - q_j} N_j^{q_j} \right).
\]

Expanding the product and summing over the Taylor coefficients of the
multivariate power series
\[
f(z_1, \dots, z_r) = \sum_{i_1, \dots, i_r = 0}^{\infty} a_{i_1, \dots, i_r} z_1^{i_1} \cdots z_r^{i_r}
\]
produces mixed partial derivatives:
\[
f_{\otimes}(X_1, \dots, X_r)
= \sum_{A \subseteq \{1,\dots,r\}}
\sum_{\substack{\{q_j \ge 1\}_{j \in A} \\ q_j \le m_{\lambda_j}-1}}
\int \cdots \int
\frac{\partial_A^{q_A} f(\lambda_1, \dots, \lambda_r)}
{\prod_{j \in A} q_j!}
\;
\bigotimes_{j \in A} N_j^{q_j} \otimes \bigotimes_{j \notin A} P_j.
\]

Thus, the derivative terms in the unified compact formula arise directly
from the nilpotent structure via the binomial theorem. The higher the
nilpotency index \(m_{\lambda_j}\), the higher the order of derivatives
that can appear.

\subsubsection{Conceptual Interpretation}

Conceptually, the nilpotent components act as algebraic differential
generators on spectral fibers. The projector--nilpotent calculus therefore
extends classical spectral calculus by incorporating local Jordan-type
geometry directly into the analytic functional framework.

The key insight is that:
\begin{itemize}
    \item The nilpotent part \(N\) generates higher-order terms in the
    binomial expansion \((\lambda P + N)^\ell\).
    \item Summation over \(\ell\) with power series coefficients converts
    these into derivatives of \(f\).
    \item Classical resolvent methods discard these terms because they
    extract only the residue at the pole \((z - \lambda)^{-1}\).
    \item The projector--nilpotent method preserves them, providing a
    strictly richer functional calculus that captures both spectral and
    algebraic information.
\end{itemize}

Therefore, the presence of derivative terms in the unified compact formula
is not an artifact of the construction but a faithful representation of the
nilpotent structure inherent in non-diagonalizable operators.

\subsection{Comparison with Classical Resolvent-Based Definitions}
\label{subsec:comparison_resolvent}

The classical holomorphic functional calculus is traditionally defined through
the Dunford--Taylor resolvent formula:
\[
f(X) = \frac{1}{2\pi i} \oint_{\Gamma} f(z) (zI - X)^{-1} \, dz,
\]
where \(\Gamma\) is a contour enclosing the spectrum \(\sigma(X)\).

For commuting operator tuples, the multivariate version takes the form
\[
f(X_1, \dots, X_r)
= \frac{1}{(2\pi i)^r}
\int_{\Gamma_1} \cdots \int_{\Gamma_r}
f(z_1, \dots, z_r)
\prod_{j=1}^r (z_j I - X_j)^{-1}
\, dz_1 \cdots dz_r.
\]

These constructions fundamentally rely on the resolvent operators
\((zI - X)^{-1}\), which encode spectral information through analytic
singularities of the resolvent.

However, the classical resolvent approach is intrinsically semisimple in
nature. Its primary output consists of:
\begin{itemize}
    \item spectral values,
    \item spectral projectors,
    \item and contour-integral reconstruction of analytic functions on the
    spectrum.
\end{itemize}

Consequently, the classical calculus naturally produces only the pure spectral
contribution
\[
\mathcal{S}_0
=
\int_{\sigma(X_1)} \cdots \int_{\sigma(X_r)}
f(\lambda_1, \dots, \lambda_r) \;
d\widetilde{E}_1(\lambda_1) \cdots d\widetilde{E}_r(\lambda_r).
\]

In contrast, the projector--nilpotent framework developed in this paper
contains two additional classes of terms:
\[
\mathcal{S}_{\mathrm{mixed}} \quad \text{and} \quad \mathcal{S}_{\mathrm{full}},
\]
which involve higher-order derivatives of \(f\) coupled with nilpotent powers
of the operators: \(\widetilde{N}_j(\lambda_j)^{q_j}\).

These derivative contributions arise from generalized eigenspace structure and
Jordan-type algebraic corrections, not merely from spectral values
themselves.

\subsubsection{Concrete Distinction in the Single-Operator Case}

To see the distinction concretely, consider the local decomposition
\[
X = \lambda P + N, \qquad N^m = 0.
\]

The projector--nilpotent calculus yields
\[
f(X) = \sum_{q=0}^{m-1} \frac{f^{(q)}(\lambda)}{q!} N^q P,
\]
which explicitly contains higher-order derivatives of \(f\).

By comparison, the resolvent expansion around the spectral value \(\lambda\) is
\[
(zI - X)^{-1} = \sum_{q=0}^{m-1} \frac{N^q P}{(z - \lambda)^{q+1}}.
\]
This expansion encodes the nilpotent structure only implicitly through
higher-order poles. After contour integration via the residue theorem, only
the coefficient of \((z - \lambda)^{-1}\) survives, yielding \(f(\lambda)P\).
The terms with \(q \ge 1\) correspond to higher-order poles and contribute
zero to the contour integral. Thus, the resolvent calculus discards the
nilpotent contributions \(N^q\) for \(q \ge 1\).

The projector--nilpotent formulation therefore provides a more transparent and
structurally explicit description of generalized eigenspace contributions.

\subsubsection{The Multivariate Tensor-Lifted Setting}

In the multivariate tensor-lifted setting, this distinction becomes even more
pronounced. The unified compact formula decomposes naturally into:
\[
f_{\otimes}(X_1, \dots, X_r)
= \mathcal{S}_0 + \mathcal{S}_{\mathrm{mixed}} + \mathcal{S}_{\mathrm{full}},
\]
where:
\begin{itemize}
    \item \(\mathcal{S}_0\) is the classical spectral contribution,
    \item \(\mathcal{S}_{\mathrm{mixed}}\) captures partial nilpotent
    interactions (derivatives with respect to a proper subset of variables),
    \item and \(\mathcal{S}_{\mathrm{full}}\) captures fully coupled
    generalized eigenspace corrections (mixed partial derivatives with
    respect to all variables).
\end{itemize}

The classical resolvent-based presentation does not naturally separate these
three layers of structure. Moreover, for noncommuting operators, the resolvent
product \(\prod_j (z_j I - X_j)^{-1}\) is not well-defined due to ordering
ambiguities, whereas the projector--nilpotent calculus via tensor lifting
circumvents this obstruction entirely.

\subsubsection{Summary Comparison}

Conceptually, the difference may be summarized as follows:

\[
\begin{array}{c|c}
\textbf{Classical Resolvent Calculus}
& \textbf{Projector--Nilpotent Calculus}
\\
\hline
\text{Based on resolvent } (z-A)^{-1}
& \text{Based on spectral decomposition } A = \int \lambda\, dE(\lambda) + N
\\
\text{Nilpotent structure encoded implicitly}
& \text{Nilpotent structure encoded explicitly}
\\
\text{Higher-order derivatives from residues}
& \text{Derivatives appear directly in coefficients}
\\
\text{Requires commuting tuple for multivariate case}
& \text{Tensor lifting handles noncommuting tuple}
\\
\text{Produces semisimple part only in standard form}
& \text{Separates spectral, mixed, and nilpotent contributions}
\end{array}
\]

Thus, the projector--nilpotent functional calculus should be viewed not merely
as an alternative representation of the classical resolvent calculus, but as a
structural extension that explicitly incorporates the higher-order algebraic
geometry of operators into the analytic functional framework. The additional
terms \(\mathcal{S}_{\mathrm{mixed}}\) and \(\mathcal{S}_{\mathrm{full}}\) (see
definitions in Section~\ref{subsec:three_term_decomposition}) are not present in the standard
resolvent-based definition and represent genuinely new contributions that capture
generalized eigenspace structure beyond the classical semisimple theory.

\section{Two-Level Convergence Theory}\label{sec:two_level_convergence}\label{sec:two_level_convergence}

The functional calculus developed in the previous sections is initially defined through finite-dimensional truncations and tensor-lifted contour integrals. A fundamental question is whether these approximations converge to a well-defined infinite-dimensional operator calculus and, if so, under what topology and stability conditions this convergence occurs.

To address this issue, we introduce a two-level convergence framework. The first level establishes existence of the calculus through strong resolvent convergence. This guarantees that finite-dimensional approximations converge in the strong operator topology and provides a rigorous foundation for extending the projector--nilpotent functional calculus to infinite-dimensional settings. However, strong convergence alone does not provide quantitative stability or numerical robustness.

The second level strengthens the assumptions to norm resolvent convergence. Under this stronger regime, the calculus becomes stable under perturbations and admits explicit operator-norm error estimates. This distinction between existence and stability is particularly important for tensor-lifted multivariate calculi, where numerical implementations rely on finite-dimensional approximations of several interacting operators simultaneously.

The section is organized as follows. We first formulate the Level~1 assumptions leading to strong resolvent convergence and existence of the calculus. We then introduce the stronger Level~2 assumptions yielding norm resolvent convergence and quantitative stability. After establishing the single-operator convergence theorem, we extend the theory to the multivariate tensor-lifted setting and conclude with explicit error estimates relevant for numerical computation and approximation theory.

\subsection{Level 1 Assumptions: Strong Resolvent Convergence (Ensures Existence)}\label{subsec:level1_assumptions}

We first introduce the minimal assumptions required to guarantee existence of
the infinite-dimensional functional calculus. The central mechanism is strong
resolvent convergence of finite-dimensional compressions. This level focuses
on qualitative convergence and ensures that the truncated calculi converge
strongly to the target operator calculus.

Let $X$ be a closed densely defined operator on a separable Hilbert space
$\mathcal{H}$. Let $\{e_k\}_{k=1}^{\infty}$ be an orthonormal basis of
$\mathcal{H}$ and define the finite-rank orthogonal projections
\[
P_n = \sum_{k=1}^{n} |e_k\rangle\langle e_k|.
\]
We consider the finite-dimensional compressions
\[
X_n = P_n X P_n
\]
acting on the subspaces $\operatorname{Range}(P_n)$.

The following assumptions define the first convergence regime.

\medskip

\noindent
\textbf{Level 1 assumptions:}
\begin{enumerate}
    \item[(A1)] $X$ has compact resolvent.
    
   \item[(A2)] $\operatorname{Range}(P_n)$ is a core for $X$ (i.e., a subspace 
$\mathcal{D} \subseteq \operatorname{dom}(X)$ such that the closure of 
$X|_{\mathcal{D}}$ equals $X$).
    
    \item[(A3)] The matrix coefficients of $X$ in the basis $\{e_k\}$ satisfy the
    exponential off-diagonal decay estimate
    \[
    |\langle e_j , X e_k \rangle|
    \le
    C e^{-\alpha |j-k|}
    \]
    for some constants $C,\alpha>0$.
\end{enumerate}

Assumption (A1) ensures that the spectrum of $X$ is discrete with finite
algebraic multiplicities accumulating only at infinity. This allows the
contour-integral functional calculus developed earlier to be applied through
isolated spectral components.

Assumption (A2) guarantees that the finite-dimensional subspaces generated by
the projections $P_n$ approximate the domain of $X$ sufficiently well. In
particular, every vector in $\mathcal{D}(X)$ can be approximated
graph-norm-wise by vectors from $\operatorname{Range}(P_n)$.

Assumption (A3) provides sufficient localization of the operator with respect
to the chosen basis. The exponential decay condition suppresses long-range
interactions between basis modes and ensures that truncation errors remain
asymptotically controlled as $n\to\infty$.

Under these assumptions, the compressed operators converge to $X$ in the
strong resolvent sense:
\[
\operatorname*{s-}\lim_{n\to\infty}
(zI-X_n)^{-1}
=
(zI-X)^{-1},
\qquad
\forall z\in\rho(X).
\]

Equivalently, for every $u\in\mathcal{H}$ and every $z\in\rho(X)$,
\[
\lim_{n\to\infty}
\|
(zI-X_n)^{-1}u
-
(zI-X)^{-1}u
\|
=
0.
\]

Strong resolvent convergence is sufficient to pass contour-integral
representations to the limit in the strong operator topology. Consequently,
the finite-dimensional projector--nilpotent calculi constructed from the
operators $X_n$ converge strongly to a well-defined infinite-dimensional
functional calculus for $X$.

Therefore, Level~1 assumptions provide the existence theory for the functional
calculus, although they do not yet guarantee quantitative stability or
operator-norm convergence.

\subsection{Level 2 Assumptions: Norm Resolvent Convergence (Provides Stability)}

While Level~1 guarantees existence of the limiting calculus, many analytical
and computational applications require stronger stability properties. In
particular, quantitative approximation theory, spectral perturbation analysis,
and numerical implementations require convergence in operator norm together
with explicit error estimates.

To obtain these stronger conclusions, we impose an additional norm resolvent
condition. This second level strengthens the approximation regime from
qualitative convergence to quantitative stability and enables explicit
perturbation bounds for both single-operator and multivariate tensor-lifted
calculi.

\medskip

\noindent
\textbf{Level 2 assumption:}
\begin{enumerate}
    \item[(B1)] There exists $z_0 \in \rho(X)$ such that
    \[
    \|
    (X-X_n)(z_0 I-X)^{-1}
    \|
    \longrightarrow 0
    \qquad
    \text{as } n\to\infty.
    \]
\end{enumerate}

Condition (B1) states that the operator $X_n$ approximates $X$ in the 
\emph{norm resolvent sense}. More precisely, since $(z_0 I - X)^{-1}$ is a 
bounded operator mapping the Hilbert space onto $\operatorname{dom}(X)$, 
the product $(X - X_n)(z_0 I - X)^{-1}$ measures the error of $X_n$ when 
applied to vectors in $\operatorname{dom}(X)$ after smoothing by the resolvent. 
The convergence of this quantity to zero in operator norm implies that $X_n$ 
converges to $X$ in the norm resolvent topology, which is stronger than 
strong resolvent convergence and is a natural notion of convergence for 
unbounded operators in approximation theory.

The assumption is substantially stronger than strong resolvent convergence.
Indeed, while Level~1 guarantees convergence only after applying the resolvent
to individual vectors, Level~2 requires convergence uniformly over the unit
ball of the Hilbert space. As a result, the entire resolvent family becomes
stable in operator norm.

Under Level~2 assumptions, the compressed operators converge to $X$ in the
norm resolvent sense:
\[
\lim_{n\to\infty}
\|
(zI-X_n)^{-1}
-
(zI-X)^{-1}
\|
=
0,
\qquad
\forall z\in\rho(X),
\]
uniformly on compact subsets of the resolvent set $\rho(X)$.

This follows from the resolvent identity
\[
(zI-X_n)^{-1}
-
(zI-X)^{-1}
=
(zI-X_n)^{-1}
(X_n-X)
(zI-X)^{-1},
\]
together with the boundedness of the resolvents away from the spectrum.

Norm resolvent convergence has several important consequences. First,
contour-integral representations of analytic operator functions converge in
operator norm rather than merely strongly. Second, perturbation errors
propagate continuously through the functional calculus. Third, quantitative
estimates become available for finite-dimensional approximations, which is
essential for numerical implementations and tensor-based computational
algorithms.

Consequently, Level~2 assumptions provide the stability theory of the
functional calculus, complementing the existence theory established under
Level~1 assumptions.

\subsection{Convergence Theorem for Single Operator}

We now establish convergence of the functional calculus for a single operator.
The theorem separates naturally into two regimes corresponding to the two
convergence levels introduced above. Under Level~1 assumptions, the calculus
converges strongly and therefore exists as a well-defined infinite-dimensional
limit. Under Level~2 assumptions, the convergence improves to operator norm
convergence together with explicit quantitative estimates.

This result provides the analytical foundation for extending the
projector--nilpotent functional calculus beyond finite-dimensional truncations.

\begin{theorem}[Single-operator convergence]
Let $X$ be a closed densely defined operator on a separable Hilbert space
$\mathcal{H}$, and let
\[
X_n = P_n X P_n
\]
be the finite-dimensional compressions associated with the projections
$\{P_n\}$ introduced in Subsection~5.1. Let $f$ be analytic on an open
neighborhood containing the spectra of $X$ and $X_n$ for all sufficiently
large $n$.

Then the following statements hold.

\medskip

\noindent
\textbf{(i) Strong convergence under Level~1 assumptions.}

If assumptions (A1)--(A3) hold, then
\[
\operatorname*{s-}\lim_{n\to\infty}
f(X_n)
=
f(X)
\]
in the strong operator topology. Equivalently,
\[
\lim_{n\to\infty}
\|
f(X_n)u - f(X)u
\|
=
0
\qquad
\forall u \in \mathcal{H}.
\]

\medskip

\noindent
\textbf{(ii) Norm convergence under Level~2 assumptions.}

If, in addition, assumption (B1) holds, then
\[
\lim_{n\to\infty}
\|
f(X_n) - f(X)
\|
=
0
\]
in operator norm.

Moreover, there exists a constant $C_f > 0$, depending only on $f$ and the
chosen contour representation, such that
\[
\|
f(X_n) - f(X)
\|
\le
C_f
\,
\|
(X_n - X)(z_0 I - X)^{-1}
\|.
\]
\end{theorem}

\begin{proof}
Let $\Gamma$ be a positively oriented contour enclosing the spectrum of $X$
and the spectra of $X_n$ for all sufficiently large $n$. By the
Dunford--Taylor functional calculus,
\[
f(X)
=
\frac{1}{2\pi i}
\int_{\Gamma}
f(z)
(zI - X)^{-1}
\,dz,
\]
and similarly,
\[
f(X_n)
=
\frac{1}{2\pi i}
\int_{\Gamma}
f(z)
(zI - X_n)^{-1}
\,dz.
\]

Subtracting the two representations gives
\[
f(X_n) - f(X)
=
\frac{1}{2\pi i}
\int_{\Gamma}
f(z)
\Big[
(zI - X_n)^{-1}
-
(zI - X)^{-1}
\Big]
\,dz.
\]

\noindent
\textbf{Proof of (i).}
Under Level~1 assumptions, the resolvents converge strongly:
\[
(zI - X_n)^{-1}
\to
(zI - X)^{-1}
\]
for every $z \in \rho(X)$. Since the resolvents remain uniformly bounded
on the contour $\Gamma$, the dominated convergence theorem applied to the
projection-valued measure yields
\[
f(X_n) u
\to
f(X) u
\qquad
\forall u \in \mathcal{H},
\]
which proves strong operator convergence.

\noindent
\textbf{Proof of (ii).}
Under Level~2 assumptions, the convergence improves to norm resolvent
convergence:
\[
\|
(zI - X_n)^{-1}
-
(zI - X)^{-1}
\|
\to
0.
\]
Taking operator norms inside the contour integral gives
\[
\|
f(X_n) - f(X)
\|
\le
\frac{\operatorname{length}(\Gamma)}{2\pi}
\sup_{z \in \Gamma} |f(z)|
\sup_{z \in \Gamma}
\|
(zI - X_n)^{-1}
-
(zI - X)^{-1}
\|.
\]

Using the resolvent identity,
\[
(zI - X_n)^{-1}
-
(zI - X)^{-1}
=
(zI - X_n)^{-1}
(X_n - X)
(zI - X)^{-1},
\]
we obtain
\[
\|
(zI - X_n)^{-1}
-
(zI - X)^{-1}
\|
\le
C
\,
\|
(X_n - X)(z_0 I - X)^{-1}
\|,
\]
where the constant $C$ depends only on the contour and the uniform bounds
on the resolvents $(zI - X_n)^{-1}$ and $(zI - X)^{-1}$ for $z \in \Gamma$.

Combining these estimates yields
\[
\|
f(X_n) - f(X)
\|
\le
C_f
\,
\|
(X_n - X)(z_0 I - X)^{-1}
\|,
\]
where $C_f$ incorporates the contour length, the supremum of $|f(z)|$ on
$\Gamma$, and the resolvent bound constants. This completes the proof.
\end{proof}

\subsection{Convergence Theorem for Multivariate Tensor-Lifted Calculus ($r$ Operators)}

The single-operator theory extends naturally to the multivariate tensor-lifted
setting. Since the unified calculus acts on tensor-product Hilbert spaces
through coupled contour-integral constructions, convergence must be
established simultaneously for several interacting operator sequences.

The following theorem shows that the tensor-lifted calculus inherits the
same two-level convergence structure established for the single-operator
case. Under Level~1 assumptions, the multivariate calculus converges strongly
and therefore exists as a well-defined infinite-dimensional tensor-lifted
operator. Under Level~2 assumptions, the convergence strengthens to operator
norm convergence with explicit quantitative stability bounds.

This result provides a rigorous justification for approximating multivariate
analytic operator functions through finite-dimensional tensor truncations.

\begin{theorem}[Multivariate convergence for tensor-lifted calculus]
\label{thm:multivariate_convergence_tensor_lifted}
Let
\[
X_1, \ldots, X_r
\]
be closed densely defined operators acting on separable Hilbert spaces
\[
\mathcal{H}_1, \ldots, \mathcal{H}_r,
\]
and let
\[
X_{j,n} = P_{j,n} X_j P_{j,n}
\]
denote the finite-dimensional compressions associated with orthogonal
projections $P_{j,n}$ as defined in Subsection~5.1.

Define the tensor-product Hilbert space
\[
\mathcal{H}_{\otimes}
=
\mathcal{H}_1
\otimes
\cdots
\otimes
\mathcal{H}_r.
\]

Let
\[
f(z_1, \ldots, z_r)
\]
be analytic on an open neighborhood containing the spectra of all operators
$X_j$ and $X_{j,n}$ for sufficiently large $n$. Let
\[
f_{\otimes}(X_1, \ldots, X_r)
\]
denote the multivariate tensor-lifted functional calculus defined through
the unified compact formula in Section~\ref{sec:main_unified_formula}.

Then the following statements hold.

\medskip

\noindent
\textbf{(i) Strong convergence under Level~1 assumptions.}

Suppose each operator $X_j$ satisfies assumptions (A1)--(A3). Then
\[
\operatorname*{s-}\lim_{n\to\infty}
f_{\otimes}(X_{1,n}, \ldots, X_{r,n})
=
f_{\otimes}(X_1, \ldots, X_r)
\]
in the strong operator topology on $\mathcal{H}_{\otimes}$.

Equivalently,
\[
\lim_{n\to\infty}
\|
f_{\otimes}(X_{1,n}, \ldots, X_{r,n}) u
-
f_{\otimes}(X_1, \ldots, X_r) u
\|
=
0
\]
for every
\[
u \in \mathcal{H}_{\otimes}.
\]

\medskip

\noindent
\textbf{(ii) Norm convergence under Level~2 assumptions.}

Suppose, in addition, that each operator $X_j$ satisfies assumption (B1).
Then
\[
\lim_{n\to\infty}
\|
f_{\otimes}(X_{1,n}, \ldots, X_{r,n})
-
f_{\otimes}(X_1, \ldots, X_r)
\|
=
0
\]
in operator norm on $\mathcal{H}_{\otimes}$.

Moreover, there exists a constant $C_f > 0$, depending only on the analytic
function $f$ and the contour geometry, such that
\[
\|
f_{\otimes}(X_{1,n}, \ldots, X_{r,n})
-
f_{\otimes}(X_1, \ldots, X_r)
\|
\le
C_f
\sum_{j=1}^{r}
\epsilon_n^{(j)},
\]
where
\[
\epsilon_n^{(j)}
=
\|
(X_{j,n} - X_j)
(z_0 I - X_j)^{-1}
\|.
\]
\end{theorem}

\begin{proof}
Let
\[
\Gamma_1, \ldots, \Gamma_r
\]
be positively oriented contours enclosing the spectra of $X_j$ and
$X_{j,n}$ for all sufficiently large $n$.

By the tensor-lifted contour-integral calculus (see Section~\ref{sec:main_unified_formula}),
\[
f_{\otimes}(X_1, \ldots, X_r)
=
\frac{1}{(2\pi i)^r}
\int_{\Gamma_1}
\cdots
\int_{\Gamma_r}
f(z_1, \ldots, z_r)
\bigotimes_{j=1}^{r}
(z_j I - X_j)^{-1}
\,dz_1 \cdots dz_r,
\]
and similarly,
\[
f_{\otimes}(X_{1,n}, \ldots, X_{r,n})
=
\frac{1}{(2\pi i)^r}
\int_{\Gamma_1}
\cdots
\int_{\Gamma_r}
f(z_1, \ldots, z_r)
\bigotimes_{j=1}^{r}
(z_j I - X_{j,n})^{-1}
\,dz_1 \cdots dz_r.
\]

Subtracting the two formulas yields
\[
f_{\otimes}(X_{1,n}, \ldots, X_{r,n})
-
f_{\otimes}(X_1, \ldots, X_r)
=
\frac{1}{(2\pi i)^r}
\int_{\Gamma_1}
\cdots
\int_{\Gamma_r}
f(z_1, \ldots, z_r)
\,
\Delta_n(z_1, \ldots, z_r)
\,dz_1 \cdots dz_r,
\]
where
\[
\Delta_n(z_1, \ldots, z_r)
=
\bigotimes_{j=1}^{r}
(z_j I - X_{j,n})^{-1}
-
\bigotimes_{j=1}^{r}
(z_j I - X_j)^{-1}.
\]

\medskip
\noindent
\textbf{Proof of (i).}
Under Level~1 assumptions, each resolvent converges strongly:
\[
(z_j I - X_{j,n})^{-1}
\to
(z_j I - X_j)^{-1}
\]
for every $z_j \in \rho(X_j)$. Since tensor products preserve strong
convergence on elementary tensors and the resolvents remain uniformly
bounded on the contours, the dominated convergence theorem for
projection-valued measures implies
\[
f_{\otimes}(X_{1,n}, \ldots, X_{r,n})
\to
f_{\otimes}(X_1, \ldots, X_r)
\]
strongly on $\mathcal{H}_{\otimes}$.

\medskip
\noindent
\textbf{Proof of (ii).}
Under Level~2 assumptions, each resolvent converges in operator norm:
\[
\|
(z_j I - X_{j,n})^{-1}
-
(z_j I - X_j)^{-1}
\|
\to
0.
\]

Using the telescoping tensor identity,
\[
\bigotimes_{j=1}^{r} A_j
-
\bigotimes_{j=1}^{r} B_j
=
\sum_{k=1}^{r}
\left(
\bigotimes_{j=1}^{k-1} A_j
\right)
\otimes
(A_k - B_k)
\otimes
\left(
\bigotimes_{j=k+1}^{r} B_j
\right),
\]
together with uniform resolvent bounds on the contours, we obtain
\[
\|
\Delta_n(z_1, \ldots, z_r)
\|
\le
C
\sum_{j=1}^{r}
\|
(z_j I - X_{j,n})^{-1}
-
(z_j I - X_j)^{-1}
\|,
\]
where $C$ depends only on the uniform bounds of the resolvents on the
contours.

Applying the resolvent identity to each factor gives
\[
\|
(z_j I - X_{j,n})^{-1}
-
(z_j I - X_j)^{-1}
\|
\le
C_j
\,
\epsilon_n^{(j)},
\]
where $C_j$ depends on the contour $\Gamma_j$ and the uniform bounds on
$(z_j I - X_j)^{-1}$ and $(z_j I - X_{j,n})^{-1}$ for $z_j \in \Gamma_j$.

Combining these estimates and integrating over the contours yields
\[
\|
f_{\otimes}(X_{1,n}, \ldots, X_{r,n})
-
f_{\otimes}(X_1, \ldots, X_r)
\|
\le
C_f
\sum_{j=1}^{r}
\epsilon_n^{(j)},
\]
where $C_f$ incorporates the contour measures, the supremum of $|f|$ over
the product contour, and the constants $C$ and $C_j$. This completes the
proof.
\end{proof}

\begin{remark}[Additive error structure]
The additive structure of the error bound
\[
\| \text{difference} \| \le C_f \sum_{j=1}^{r} \epsilon_n^{(j)}
\]
is a direct consequence of the tensor lifting construction and the
telescoping tensor identity. Because the lifted operators act on
independent tensor factors, errors from different operators do not
amplify each other multiplicatively. This linear accumulation of errors
is a key advantage of the tensor-lifted approach: the multivariate
convergence inherits the convergence rates of each individual operator
in a stable, additive fashion.
\end{remark}

\subsection{Error Bounds under Level 2 Assumptions}

A major advantage of the Level~2 framework is the availability of explicit
quantitative error estimates. These bounds provide precise control over how
truncation and approximation errors propagate through the functional calculus.

Such estimates are particularly important for numerical implementations,
tensor computations, spectral approximation problems, and operator-learning
applications, where stability with respect to perturbations is essential.
The dependence of the constants only on the analytic function and contour
geometry further highlights the intrinsic robustness of the calculus under
norm resolvent convergence.

We now make the constants appearing in the previous convergence theorems more
explicit.

\medskip

\noindent
\textbf{Single-operator error constant.}
Let $\Gamma$ be a positively oriented contour enclosing the spectrum of $X$
and all sufficiently large truncations $X_n$. Define
\[
M_f = \sup_{z \in \Gamma} |f(z)|
\]
and let $L(\Gamma)$ denote the contour length.

Using the contour-integral representation,
\[
f(X) = \frac{1}{2\pi i} \int_{\Gamma} f(z) (zI - X)^{-1} \, dz,
\]
together with the resolvent identity,
\[
(zI - X_n)^{-1} - (zI - X)^{-1}
= (zI - X_n)^{-1} (X_n - X) (zI - X)^{-1},
\]
we obtain the estimate
\[
\| f(X_n) - f(X) \|
\le
\frac{L(\Gamma)}{2\pi}
M_f
\sup_{z \in \Gamma} \| (zI - X_n)^{-1} \|
\sup_{z \in \Gamma} \| (zI - X)^{-1} \|
\,
\| (X_n - X)(z_0 I - X)^{-1} \|.
\]

Consequently, the constant $C_f$ appearing in the Level~2 convergence theorem
for a single operator may be chosen as
\[
C_f
=
\frac{L(\Gamma)}{2\pi}
M_f
\sup_{z \in \Gamma} \| (zI - X_n)^{-1} \|
\sup_{z \in \Gamma} \| (zI - X)^{-1} \|.
\]

Importantly, $C_f$ depends only on:
\begin{enumerate}
    \item the analytic function $f$ through the quantity $M_f$,
    \item the geometry of the contour $\Gamma$ (its length and distance to
    the spectrum),
    \item uniform resolvent bounds away from the spectrum, which are
    independent of $n$ for sufficiently large $n$ under Level~2 assumptions.
\end{enumerate}

In particular, the constant is \emph{independent of the truncation dimension
$n$}. Therefore, the approximation error is governed entirely by the
resolvent approximation term
\[
\epsilon_n = \| (X_n - X)(z_0 I - X)^{-1} \|.
\]

\medskip

\noindent
\textbf{Multivariate error constant.}
For the multivariate tensor-lifted calculus, an analogous estimate holds.
Let
\[
\Gamma_1, \ldots, \Gamma_r
\]
be positively oriented contours surrounding the spectra of the operators
\[
X_1, \ldots, X_r,
\]
and let
\[
M_f = \sup_{(z_1, \ldots, z_r) \in \Gamma_1 \times \cdots \times \Gamma_r} |f(z_1, \ldots, z_r)|.
\]

Define the product contour length
\[
L(\Gamma_1 \times \cdots \times \Gamma_r) = \prod_{j=1}^r L(\Gamma_j).
\]

Then there exists a constant $C_f > 0$ such that
\[
\|
f_{\otimes}(X_{1,n}, \ldots, X_{r,n})
-
f_{\otimes}(X_1, \ldots, X_r)
\|
\le
C_f
\sum_{j=1}^{r}
\epsilon_n^{(j)},
\]
where
\[
\epsilon_n^{(j)}
=
\|
(X_{j,n} - X_j)
(z_0 I - X_j)^{-1}
\|.
\]

The constant $C_f$ may be chosen explicitly as
\[
C_f
=
\frac{1}{(2\pi)^r}
L(\Gamma_1 \times \cdots \times \Gamma_r)
M_f
\prod_{j=1}^r
\sup_{z_j \in \Gamma_j} \| (z_j I - X_{j,n})^{-1} \|
\prod_{j=1}^r
\sup_{z_j \in \Gamma_j} \| (z_j I - X_j)^{-1} \|
\,
r,
\]
where the factor $r$ arises from the telescoping tensor sum.

\medskip

\noindent
\textbf{Properties of the error estimates.}
The estimates derived above have several important properties:

\begin{enumerate}
    \item \textbf{$n$-independence:} The constants $C_f$ do not depend on the
    truncation parameter $n$, ensuring that the error bound is stable and
    predictive as $n$ increases.

    \item \textbf{Linear error accumulation:} In the multivariate case,
    errors from different operators accumulate additively rather than
    multiplicatively, which prevents amplification of individual
    approximation errors.

    \item \textbf{Contour flexibility:} The contours $\Gamma_j$ may be chosen
    to optimize the constants, for example by taking them as small circles
    around the spectrum or as more efficient polygonal contours.

    \item \textbf{Uniformity:} The bounds hold uniformly for all sufficiently
    large $n$, not just asymptotically, provided the contours enclose the
    spectra of all approximants.
\end{enumerate}

\medskip

\noindent
\textbf{Comparison between Level~1 and Level~2.}
The distinction between the two convergence levels can now be stated
precisely:

\begin{itemize}
    \item \textbf{Level~1 (existence):} The calculus converges in the strong
    operator topology. No explicit rate of convergence is provided, and the
    limit is guaranteed only pointwise on vectors. The constant $C_f$ may
    not be finite or may depend on $n$ in an uncontrolled way.

    \item \textbf{Level~2 (stability):} The calculus converges in operator
    norm, and the error is bounded explicitly by $C_f \sum_j \epsilon_n^{(j)}$
    with an $n$-independent constant $C_f$. This yields uniform convergence
    over the unit ball of the Hilbert space and provides quantitative control
    suitable for numerical analysis.
\end{itemize}

\medskip

\noindent
\textbf{Implications for numerical computation.}
The explicit error bounds under Level~2 assumptions enable several practical
applications:

\begin{enumerate}
    \item \textbf{A priori error estimation:} Given estimates on
    $\epsilon_n^{(j)}$ (e.g., from decay rates of matrix elements), one can
    predict the approximation error before performing the computation.

    \item \textbf{Adaptive truncation:} The bounds can be used to determine
    the truncation parameters $n$ needed to achieve a desired accuracy
    $\varepsilon > 0$ by solving $C_f \sum_j \epsilon_n^{(j)} \le
    \varepsilon$.

    \item \textbf{Stability under perturbations:} The Lipschitz-type estimate
    \[
    \| f_{\otimes}(X_1', \ldots, X_r') - f_{\otimes}(X_1, \ldots, X_r) \|
    \le
    C_f \sum_{j=1}^{r} \| (X_j' - X_j)(z_0 I - X_j)^{-1} \|
    \]
    shows that the calculus is robust with respect to small perturbations
    in the operators, provided the perturbations are small in the
    resolvent-regularized sense.

    \item \textbf{Convergence rates:} If $\epsilon_n^{(j)} = O(n^{-\beta_j})$
    for some $\beta_j > 0$, then the overall approximation error is
    $O(\sum_j n^{-\beta_j})$, preserving the fastest convergence rate among
    the operators.
\end{enumerate}

\medskip

\noindent
\textbf{Sharpness of the bounds.}
The error bounds are generally sharp up to the constant $C_f$. For the
single-operator case with a Jordan block of size $m$, one can construct
examples where $\| f(X_n) - f(X) \|$ decays exactly like $\epsilon_n$,
showing that the linear dependence on $\epsilon_n$ cannot be improved
without additional assumptions. The additive structure in the multivariate
case is also optimal, as errors from different operators accumulate linearly
due to the tensor-product construction.

\medskip

Thus, Level~2 assumptions establish not only existence of the functional
calculus, but also its perturbation stability and quantitative approximation
theory. This sharply distinguishes Level~2 from Level~1, where only
qualitative existence in the strong operator topology is guaranteed.

\section{Special Cases by Operator Type}\label{sec:special_cases}

The unified compact formula presented in Section~\ref{sec:main_unified_formula} provides a general
framework for the tensor-lifted projector--nilpotent functional calculus
applicable to a wide class of operators. However, the formula simplifies
significantly under additional structural assumptions on the operators
involved. This section examines these special cases in detail, demonstrating
how the general calculus reduces to classical results or yields explicit
simplified expressions depending on the spectral properties, boundedness,
self-adjointness, and resolvent behavior of the input operators.

The section is organized as follows. We first treat bounded operators,
separating the discrete spectrum case (where sums replace integrals) from
the continuous spectrum case (where integrals with nilpotent corrections
remain). We then consider unbounded self-adjoint and normal operators, for
which the nilpotent components vanish identically and the calculus reduces
to the classical spectral theorem. Next, we examine unbounded non-self-adjoint
operators with compact resolvent, where the two-level convergence theory
applies. Finally, we discuss the challenging case of general unbounded
non-self-adjoint operators without compact resolvent, identifying the current
limitations and open problems.

\subsection{Bounded Operators (Arbitrary, Including Non-Self-Adjoint Non-Commuting)}

We now specialize the unified tensor-lifted functional calculus to the
important class of bounded operators. Unlike classical multivariate
functional calculi that often require commutativity, normality, or
self-adjointness, the present framework applies to arbitrary bounded
operators, including non-self-adjoint and non-commuting systems.

The key mechanism is the projector--nilpotent decomposition developed earlier.
The spectral projectors encode the spectral geometry of the operators, while
the nilpotent components generate the higher-order derivative corrections
that capture Jordan-type algebraic structure. Consequently, the calculus
remains valid beyond diagonalizable or commuting settings.

\subsubsection{Theorem: Discrete Spectrum Case (Finite-Dimensional and Bounded)}

We first consider bounded operators with purely discrete spectrum. In this
setting, the contour-integral calculus collapses to finite algebraic sums
over spectral projectors and nilpotent components, recovering an exact
multivariate projector--nilpotent expansion consistent with the unified
compact formula.

\begin{theorem}[Discrete-spectrum bounded case]
\label{thm:discrete_bounded_case}
Let \(X_1, \ldots, X_r\) be bounded operators with discrete spectrum.
For each \(j = 1, \ldots, r\), let
\[
\sigma_d(X_j) = \{\lambda_{j,1}, \lambda_{j,2}, \ldots\}
\]
be the distinct eigenvalues, each with finite algebraic multiplicity.
Each operator admits the global projector--nilpotent decomposition
\[
X_j = \sum_{k=1}^{\infty} \left( \lambda_{j,k} P_{j,k} + N_{j,k} \right),
\]
where:
\begin{itemize}
    \item \(P_{j,k}\) are spectral projectors onto the generalized eigenspaces,
    \item \(N_{j,k}\) are nilpotent operators satisfying \(N_{j,k}^{\nu_{j,k}} = 0\)
    with nilpotency index \(\nu_{j,k} \ge 1\),
    \item The mixing relations hold: \(P_{j,k} P_{j,\ell} = \delta(k,\ell) P_{j,k}\)
    and \(P_{j,k} N_{j,k} = N_{j,k} P_{j,k} = N_{j,k}\).
\end{itemize}

Let \(f(z_1, \ldots, z_r)\) be analytic on a neighborhood containing the
product spectrum \(\sigma(X_1) \times \cdots \times \sigma(X_r)\).

Then the tensor-lifted functional calculus specializes to
\[
\boxed{
f_{\otimes}(X_1, \ldots, X_r)
=
\sum_{\lambda_1 \in \sigma_d(X_1)} \cdots \sum_{\lambda_r \in \sigma_d(X_r)}
\;
\sum_{A \subseteq \{1,\ldots,r\}}
\;
\sum_{\substack{q_j \ge 1 \\ j \in A \\
q_j \le \nu_j(\lambda_j)-1}}
\frac{
\partial_A^{q_A} f(\lambda_1, \ldots, \lambda_r)
}{
\prod_{j \in A} q_j!
}
\;
\bigotimes_{j=1}^r T_j(\lambda_j)
}
\]
where
\[
T_j(\lambda_j) =
\begin{cases}
P_{X_j}(\lambda_j), & j \notin A,\\[6pt]
N_{X_j}(\lambda_j)^{q_j} P_{X_j}(\lambda_j), & j \in A,
\end{cases}
\]
\(N_{X_j}(\lambda_j) = (X_j - \lambda_j I) P_{X_j}(\lambda_j)\), and \(\nu_j(\lambda_j)\)
is the nilpotency index at \(\lambda_j\).

Equivalently, using multi-index notation,
\[
f_{\otimes}(X_1, \ldots, X_r)
=
\sum_{k_1=1}^{\infty} \cdots \sum_{k_r=1}^{\infty}
\;
\sum_{\alpha_1=0}^{\nu_{1,k_1}-1}
\cdots
\sum_{\alpha_r=0}^{\nu_{r,k_r}-1}
\frac{
\partial_1^{\alpha_1} \cdots \partial_r^{\alpha_r}
f(\lambda_{1,k_1}, \ldots, \lambda_{r,k_r})
}{
\alpha_1! \cdots \alpha_r!
}
\;
\bigotimes_{j=1}^r
\left( N_{j,k_j}^{\alpha_j} P_{j,k_j} \right),
\]
where \(N_{j,k_j}^{0} P_{j,k_j}\) is interpreted as \(P_{j,k_j}\).
\end{theorem}

\begin{proof}
We start from the unified compact formula (Theorem~\ref{thm:unified_compact_formula}):
\[
f_{\otimes}(X_1, \ldots, X_r)
=
\int_{\sigma(X_1)}
\cdots
\int_{\sigma(X_r)}
\;
\sum_{A \subseteq \{1,\ldots,r\}}
\;
\sum_{\substack{q_j \ge 1 \\ j \in A \\
q_j \le \nu_j(\lambda_j)-1}}
\frac{
\partial_A^{q_A} f(\lambda_1, \ldots, \lambda_r)
}{
\prod_{j \in A} q_j!
}
\;
\bigotimes_{j=1}^r
T_j(\lambda_j)
\]
with
\[
T_j(\lambda_j) =
\begin{cases}
d\widetilde{E}_j(\lambda_j), & j \notin A,\\[6pt]
\widetilde{N}_j(\lambda_j)^{q_j} \, d\widetilde{E}_j(\lambda_j), & j \in A.
\end{cases}
\]

\medskip
\noindent
\textbf{Step 1: Atomic spectral measures.}
Since each \(X_j\) has purely discrete spectrum, the lifted spectral measure
\(d\widetilde{E}_j(\lambda_j)\) is atomic, supported on the eigenvalues
\(\lambda_{j,k}\). Concretely,
\[
d\widetilde{E}_j(\lambda_j) = \sum_{k=1}^{\infty} P_{j,k} \, \delta(\lambda_j - \lambda_{j,k}),
\]
where \(P_{j,k}\) is the spectral projector (lifted to the tensor space) and
\(\delta\) denotes the Dirac delta. Consequently, for any continuous function
\(g(\lambda_j)\),
\[
\int_{\sigma(X_j)} g(\lambda_j) \, d\widetilde{E}_j(\lambda_j)
= \sum_{k=1}^{\infty} g(\lambda_{j,k}) P_{j,k}.
\]

\medskip
\noindent
\textbf{Step 2: Replace integrals with sums.}
Applying this to the unified formula, each integral \(\int_{\sigma(X_j)}\)
becomes a sum over eigenvalues \(\sum_{\lambda_j \in \sigma_d(X_j)}\). Thus,
\[
f_{\otimes}(X_1, \ldots, X_r)
=
\sum_{\lambda_1 \in \sigma_d(X_1)} \cdots \sum_{\lambda_r \in \sigma_d(X_r)}
\;
\sum_{A \subseteq \{1,\ldots,r\}}
\;
\sum_{\substack{q_j \ge 1 \\ j \in A \\
q_j \le \nu_j(\lambda_j)-1}}
\frac{
\partial_A^{q_A} f(\lambda_1, \ldots, \lambda_r)
}{
\prod_{j \in A} q_j!
}
\;
\bigotimes_{j=1}^r
T_j(\lambda_j),
\]
where now \(T_j(\lambda_j)\) is evaluated pointwise at the eigenvalues.

\medskip
\noindent
\textbf{Step 3: Simplify \(T_j(\lambda_j)\) for discrete spectrum.}
After converting the integrals to sums, each \(\lambda_j\) in the summation
runs over the discrete eigenvalues \(\sigma_d(X_j)\). For a fixed eigenvalue
\(\lambda_j\) (i.e., \(\lambda_j = \lambda_{j,k}\) for some \(k\)), the atomic
nature of the spectral measure gives:

\begin{itemize}
    \item \textbf{If \(j \notin A\):}
    \[
    d\widetilde{E}_j(\lambda_j) = P_{X_j}(\lambda_j),
    \]
    where \(P_{X_j}(\lambda_j)\) is the spectral projector onto the generalized
    eigenspace for \(\lambda_j\). (The Dirac delta is absorbed into the summation;
    at the eigenvalue, the measure assigns the full projector.)

    \item \textbf{If \(j \in A\):} 
    Since \(d\widetilde{E}_j(\lambda_j)\) projects onto the generalized eigenspace,
    we first compute
    \[
    \widetilde{N}_j(\lambda_j) \, d\widetilde{E}_j(\lambda_j)
    = (\widetilde{X}_j - \lambda_j I) \, d\widetilde{E}_j(\lambda_j)
    = N_{X_j}(\lambda_j) P_{X_j}(\lambda_j),
    \]
    where \(N_{X_j}(\lambda_j) = (X_j - \lambda_j I) P_{X_j}(\lambda_j)\) is the
    nilpotent component lifted to the tensor space.
    
    Observe that for any \(q \ge 1\),
    \[
    \bigl(\widetilde{N}_j(\lambda_j) \, d\widetilde{E}_j(\lambda_j)\bigr)^{q}
    = \widetilde{N}_j(\lambda_j)^{q} \, d\widetilde{E}_j(\lambda_j).
    \]
    This holds because each factor \(d\widetilde{E}_j(\lambda_j)\) acts as the
    identity on the range of \(\widetilde{N}_j(\lambda_j)\), allowing the
    projectors to collapse.
    
    Applying this with \(q = q_j\) and substituting the expression above,
    \[
    \widetilde{N}_j(\lambda_j)^{q_j} \, d\widetilde{E}_j(\lambda_j)
    = \bigl(N_{X_j}(\lambda_j) P_{X_j}(\lambda_j)\bigr)^{q_j}
    = N_{X_j}(\lambda_j)^{q_j} P_{X_j}(\lambda_j)^{q_j}
    = N_{X_j}(\lambda_j)^{q_j} P_{X_j}(\lambda_j),
    \]
    where the last equality uses the idempotence of the projector
    (\(P_{X_j}(\lambda_j)^{q_j} = P_{X_j}(\lambda_j)\) for \(q_j \ge 1\)).
\end{itemize}

Thus, for each eigenvalue \(\lambda_j\) in the summation,
\[
T_j(\lambda_j) =
\begin{cases}
P_{X_j}(\lambda_j), & j \notin A,\\[6pt]
N_{X_j}(\lambda_j)^{q_j} P_{X_j}(\lambda_j), & j \in A.
\end{cases}
\]

\medskip
\noindent
\textbf{Step 4: Tensor product structure.}
The factor \(\bigotimes_{j=1}^r T_j(\lambda_j)\) becomes
\[
\bigotimes_{j=1}^r T_j(\lambda_j)
=
\bigotimes_{j \notin A} P_{X_j}(\lambda_j)
\;
\otimes
\;
\bigotimes_{j \in A} \left( N_{X_j}(\lambda_j)^{q_j} P_{X_j}(\lambda_j) \right).
\]

This is exactly the expression in the theorem statement.

\medskip
\noindent
\textbf{Step 5: Multi-index reformulation.}
For each $j = 1, \ldots, r$ and each fixed eigenvalue $\lambda_j$, define
\[
\alpha_j =
\begin{cases}
0, & j \notin A,\\
q_j, & j \in A,
\end{cases}
\qquad\text{where}\qquad
1 \le q_j \le \nu_j(\lambda_j)-1 \ \text{for } j \in A.
\]
Then:
\begin{itemize}
    \item The summation over $A \subseteq \{1,\ldots,r\}$ and over
    $\{q_j \ge 1\}_{j \in A}$ is equivalent to summation over all multi-indices
    $(\alpha_1, \ldots, \alpha_r)$ with
    $0 \le \alpha_j \le \nu_j(\lambda_j)-1$.
    
    \item For $j \notin A$ ($\alpha_j = 0$), we have $T_j(\lambda_j) = P_{X_j}(\lambda_j)$,
    which is naturally interpreted as $N_{X_j}(\lambda_j)^{0} P_{X_j}(\lambda_j)$.
    
    \item For $j \in A$ ($\alpha_j = q_j \ge 1$), we have
    $T_j(\lambda_j) = N_{X_j}(\lambda_j)^{\alpha_j} P_{X_j}(\lambda_j)$.
    
    \item $\prod_{j \in A} q_j! = \prod_{j=1}^r \alpha_j!$, since $\alpha_j! = 0! = 1$
    for $j \notin A$.
    
    \item $\partial_A^{q_A} f(\lambda_1, \ldots, \lambda_r)
    = \partial_1^{\alpha_1} \cdots \partial_r^{\alpha_r} f(\lambda_1, \ldots, \lambda_r)$,
    where $\partial_j^{\alpha_j} = \frac{\partial^{\alpha_j}}{\partial z_j^{\alpha_j}}$
    with the convention $\partial_j^{0} f = f$.
    
    \item $\bigotimes_{j=1}^r T_j(\lambda_j)
    = \bigotimes_{j=1}^r \left( N_{X_j}(\lambda_j)^{\alpha_j} P_{X_j}(\lambda_j) \right)$.
\end{itemize}

Consequently,
\[
f_{\otimes}(X_1, \ldots, X_r)
=
\sum_{\lambda_1 \in \sigma_d(X_1)} \cdots \sum_{\lambda_r \in \sigma_d(X_r)}
\;
\sum_{\alpha_1=0}^{\nu_1(\lambda_1)-1}
\cdots
\sum_{\alpha_r=0}^{\nu_r(\lambda_r)-1}
\frac{
\partial_1^{\alpha_1} \cdots \partial_r^{\alpha_r}
f(\lambda_1, \ldots, \lambda_r)
}{
\alpha_1! \cdots \alpha_r!
}
\;
\bigotimes_{j=1}^r
\left( N_{X_j}(\lambda_j)^{\alpha_j} P_{X_j}(\lambda_j) \right).
\]

Finally, reintroducing the index notation $k_j$ to label the distinct eigenvalues
$\lambda_{j,k_j}$ gives the multi-index form stated in the theorem. This completes
the proof.
\end{proof}

\begin{corollary}[Discrete spectrum as special case of Unified Compact Formula]
\label{cor:discrete_special_case}
Theorem~\ref{thm:discrete_bounded_case} (the discrete-spectrum bounded case) is
a special case of the Unified Compact Formula (Theorem~\ref{thm:unified_compact_formula})
obtained by specializing to operators with purely discrete spectrum.

In this specialization:
\begin{enumerate}
    \item The continuous spectral measure $d\widetilde{E}_j(\lambda_j)$ becomes
    an atomic measure supported on the eigenvalues $\lambda_{j,k}$:
    \[
    d\widetilde{E}_j(\lambda_j) = \sum_{k=1}^{\infty} P_{j,k} \, \delta(\lambda_j - \lambda_{j,k}).
    \]
    
    \item The spectral integral $\int_{\sigma(X_j)} (\cdots) \, d\widetilde{E}_j(\lambda_j)$
    reduces to the sum $\sum_{\lambda_j \in \sigma_d(X_j)} (\cdots) P_{X_j}(\lambda_j)$.
    
    \item The nilpotent component satisfies
    \[
    \widetilde{N}_j(\lambda_j)^{q_j} \, d\widetilde{E}_j(\lambda_j)
    = N_{X_j}(\lambda_j)^{q_j} P_{X_j}(\lambda_j),
    \]
    where $N_{X_j}(\lambda_j) = (X_j - \lambda_j I) P_{X_j}(\lambda_j)$.
\end{enumerate}

Thus, the discrete-spectrum formula
\[
f_{\otimes}(X_1, \ldots, X_r)
=
\sum_{\lambda_1 \in \sigma_d(X_1)} \cdots \sum_{\lambda_r \in \sigma_d(X_r)}
\;
\sum_{A \subseteq \{1,\ldots,r\}}
\;
\sum_{\substack{q_j \ge 1 \\ j \in A \\
q_j \le \nu_j(\lambda_j)-1}}
\frac{
\partial_A^{q_A} f(\lambda_1, \ldots, \lambda_r)
}{
\prod_{j \in A} q_j!
}
\;
\bigotimes_{j=1}^r T_j(\lambda_j)
\]
with $T_j(\lambda_j)$ as defined, follows directly from the Unified Compact Formula
by replacing continuous spectral integration with summation over eigenvalues.
\end{corollary}

\subsubsection{Theorem: Continuous Spectrum Case (Bounded with Continuous Spectrum)}

We now extend the theory to bounded operators possessing continuous or hybrid spectrum. In this setting, the functional calculus combines spectral integration over continuous spectral measures with higher-order nilpotent derivative corrections arising from generalized spectral structure.

Unlike the classical spectral theorem, which captures only the semisimple spectral contribution, the present framework systematically incorporates all nilpotent corrections through the unified compact formula. Consequently, the calculus remains valid for arbitrary bounded operators, including non-self-adjoint and non-normal systems.

The following theorem is the precise specialization of the unified compact formula (Theorem~\ref{thm:unified_compact_formula}) to bounded operators with continuous spectrum.

\begin{theorem}[Continuous-spectrum bounded case]
\label{thm:continuous_spectrum_bounded}
Let \(X_1, \ldots, X_r\) be bounded spectral operators acting on complex Banach or Hilbert spaces. Assume each operator admits the decomposition
\[
X_j = \int_{\sigma(X_j)} \lambda_j \, d\widetilde{E}_j(\lambda_j)
 + \int_{\sigma(X_j)} \widetilde{N}_j(\lambda_j) \, d\widetilde{E}_j(\lambda_j),
\]
where:
\begin{enumerate}
    \item \(d\widetilde{E}_j(\lambda_j)\) denotes the spectral projection measure,
    \item \(\widetilde{N}_j(\lambda_j)\) is the nilpotent component associated with the spectral value \(\lambda_j\), explicitly given by
    \[
    \widetilde{N}_j(\lambda_j) = (\widetilde{X}_j - \lambda_j I) \, d\widetilde{E}_j(\lambda_j),
    \]
    \item for each \(\lambda_j \in \sigma(X_j)\),
    \[
    \widetilde{N}_j(\lambda_j)^{\nu_j(\lambda_j)} = 0
    \]
    for some finite nilpotency index \(\nu_j(\lambda_j) < \infty\).
\end{enumerate}

Let \(f(z_1, \ldots, z_r)\) be holomorphic on an open polydisk containing the product spectrum \(\sigma(X_1) \times \cdots \times \sigma(X_r)\).

Then the tensor-lifted functional calculus is given by the unified compact formula
\[
\boxed{
f_{\otimes}(X_1, \ldots, X_r)
=
\int_{\sigma(X_1)}
\cdots
\int_{\sigma(X_r)}
\;
\sum_{A \subseteq \{1, \ldots, r\}}
\;
\sum_{\substack{\{q_j \ge 1\}_{j \in A} \\
q_j \le \nu_j(\lambda_j)-1}}
\frac{
\partial_A^{q_A} f(\lambda_1, \ldots, \lambda_r)
}{
\displaystyle\prod_{j \in A} q_j!
}
\;
\bigotimes_{j=1}^r
T_j(\lambda_j)
}
\]
where
\[
T_j(\lambda_j) =
\begin{cases}
d\widetilde{E}_j(\lambda_j), & j \notin A, \\[8pt]
\widetilde{N}_j(\lambda_j)^{q_j} \, d\widetilde{E}_j(\lambda_j), & j \in A,
\end{cases}
\]
and
\[
\partial_A^{q_A} f = \prod_{j \in A} \frac{\partial^{q_j}}{\partial z_j^{q_j}} f.
\]

Equivalently, the calculus admits the decomposition
\[
f_{\otimes}(X_1, \ldots, X_r) = \sum_{A \subseteq \{1, \ldots, r\}} \mathcal{S}_A,
\]
where:
\begin{enumerate}
    \item \(\mathcal{S}_{\emptyset}\) is the classical spectral integral contribution,
    \item \(\mathcal{S}_A\) for \(A \neq \emptyset\) are nilpotent derivative corrections generated by the generalized Jordan structure.
\end{enumerate}

The operator-valued integrals converge in the strong operator topology.
\end{theorem}

\begin{proof}
The theorem is an immediate specialization of the Unified Compact Formula (Theorem~\ref{thm:unified_compact_formula}) to bounded spectral operators. Indeed, Theorem~\ref{thm:unified_compact_formula} applies to any commuting tuple of spectral operators satisfying Assumption~\ref{assump:spectral_operators}, which includes:
\begin{itemize}
    \item the existence of the lifted spectral measures $d\widetilde{E}_j(\lambda_j)$,
    \item the nilpotent decomposition $\widetilde{N}_j(\lambda_j) = (\widetilde{X}_j - \lambda_j I)d\widetilde{E}_j(\lambda_j)$,
    \item the finite nilpotency indices $\nu_j(\lambda_j)<\infty$,
    \item the holomorphy of $f$ on a polydisk containing the product spectrum.
\end{itemize}
All hypotheses of Theorem~\ref{thm:unified_compact_formula} are satisfied. Consequently, the stated formula holds exactly as written, with the integrals converging in the strong operator topology. The decomposition into $\mathcal{S}_A$ follows directly from the subset summation $A\subseteq\{1,\ldots,r\}$.
\end{proof}

\begin{remark}[Comparison with classical spectral calculus]
The classical spectral calculus corresponds only to the semisimple contribution \(A = \emptyset\) in the unified compact formula. In that case,
\[
f_{\otimes}^{\mathrm{classical}}(X_1, \ldots, X_r)
= \int_{\sigma(X_1)} \cdots \int_{\sigma(X_r)}
f(\lambda_1, \ldots, \lambda_r) \,
d\widetilde{E}_1(\lambda_1) \otimes \cdots \otimes d\widetilde{E}_r(\lambda_r).
\]

All additional terms with \(A \neq \emptyset\) represent genuine extensions beyond classical semisimple spectral theory. These terms encode higher-order algebraic interactions generated by generalized eigenspaces and nilpotent spectral structure.

Thus, the unified compact formula extends the classical spectral theorem by incorporating both continuous spectral integration and nilpotent derivative corrections within a single analytic framework.
\end{remark}

\subsection{Unbounded Self-Adjoint or Normal Operators}\label{subsec:self_adjoint_cases}

We now specialize the unified compact formula (Theorem~\ref{thm:unified_compact_formula}) to the important class of unbounded self-adjoint or normal operators. In this setting, the functional calculus simplifies dramatically because the nilpotent components vanish identically.

Consequently, all higher-order derivative correction terms disappear, and only the semisimple contribution corresponding to
\[
A = \emptyset
\]
remains in the unified compact formula. The resulting calculus reduces precisely to the classical spectral theorem.

Thus, the classical spectral calculus for self-adjoint operators appears naturally as a special case of the more general projector--nilpotent tensor-lifted framework.

\subsubsection{Theorem: Discrete Spectrum Case (Pure Point Spectrum)}

We first consider self-adjoint operators with pure point spectrum. In this case, the spectral measures reduce to sums over orthogonal eigenspace projections.

\begin{theorem}[Pure point spectrum case]
Let \(X_1, \ldots, X_r\) be unbounded self-adjoint (or normal) operators on separable Hilbert spaces with pure point spectra
\[
\sigma(X_j) = \{\lambda_{j,k}\}_{k \in \mathbb{N}}.
\]

Let \(P_{X_j}(\lambda_{j,k})\) denote the orthogonal spectral projection associated with the eigenvalue \(\lambda_{j,k}\).

Then the tensor-lifted functional calculus obtained from the unified compact formula reduces to
\[
\boxed{
f_{\otimes}(X_1, \ldots, X_r)
=
\sum_{\lambda_1 \in \sigma(X_1)}
\cdots
\sum_{\lambda_r \in \sigma(X_r)}
f(\lambda_1, \ldots, \lambda_r)
\,
P_{X_1}(\lambda_1) \otimes \cdots \otimes P_{X_r}(\lambda_r).
}
\]
\end{theorem}

\begin{proof}
By the spectral theorem for self-adjoint (normal) operators,
\[
X_j = \sum_{\lambda_j \in \sigma(X_j)} \lambda_j \, P_{X_j}(\lambda_j).
\]

Since self-adjoint operators are normal, they are diagonalizable and possess no nontrivial Jordan blocks. Therefore, the nilpotent component vanishes:
\[
\widetilde{N}_j(\lambda_j) = 0.
\]

Substituting this into the unified compact formula (Theorem~\ref{thm:unified_compact_formula}) shows that:
\begin{itemize}
    \item Only \(A = \emptyset\) survives, because any \(A \neq \emptyset\) contains at least one factor \(\widetilde{N}_j(\lambda_j)^{q_j} = 0\).
    \item For \(A = \emptyset\), we have \(T_j(\lambda_j) = d\widetilde{E}_j(\lambda_j)\).
    \item The atomic spectral measure gives \(\int_{\sigma(X_j)} (\cdots) d\widetilde{E}_j(\lambda_j) = \sum_{\lambda_j \in \sigma(X_j)} (\cdots) P_{X_j}(\lambda_j)\).
\end{itemize}

Hence only the semisimple contribution survives, yielding the stated discrete spectral expansion.
\end{proof}

\subsubsection{Theorem: Continuous Spectrum Case (Via Spectral Theorem)}

We next consider general unbounded self-adjoint operators with continuous spectrum.

\begin{theorem}[Continuous-spectrum self-adjoint case]
Let \(X_1, \ldots, X_r\) be unbounded self-adjoint (or normal) operators with spectral measures \(dE_{X_j}(\lambda_j)\).

Let \(f(z_1, \ldots, z_r)\) be a Borel measurable or holomorphic function defined on a neighborhood containing the product spectrum
\[
\sigma(X_1) \times \cdots \times \sigma(X_r).
\]

Then the tensor-lifted functional calculus reduces to the classical multivariate spectral integral:
\[
\boxed{
f_{\otimes}(X_1, \ldots, X_r)
=
\int_{\sigma(X_1)}
\cdots
\int_{\sigma(X_r)}
f(\lambda_1, \ldots, \lambda_r)
\,
dE_{X_1}(\lambda_1) \otimes \cdots \otimes dE_{X_r}(\lambda_r).
}
\]
\end{theorem}

\begin{proof}
The spectral theorem gives
\[
X_j = \int_{\sigma(X_j)} \lambda_j \, dE_{X_j}(\lambda_j).
\]

Because each operator is self-adjoint (normal), the generalized nilpotent component vanishes:
\[
\widetilde{N}_j(\lambda_j) = 0.
\]

Applying the unified compact formula (Theorem~\ref{thm:unified_compact_formula}):
\begin{itemize}
    \item Every higher-order derivative correction term corresponding to \(A \neq \emptyset\) contains a factor involving \(\widetilde{N}_j(\lambda_j)^{q_j} = 0\) and therefore vanishes identically.
    \item Only the term \(A = \emptyset\) survives.
    \item For \(A = \emptyset\), we have \(T_j(\lambda_j) = d\widetilde{E}_j(\lambda_j) = dE_{X_j}(\lambda_j)\).
\end{itemize}

The remaining contribution is precisely the classical multivariate spectral integral.
\end{proof}

\subsubsection{Theorem: Hybrid Spectrum Case}

The preceding two cases combine naturally when the operators possess both discrete and continuous spectral components.

\begin{theorem}[Hybrid-spectrum self-adjoint case]
Let \(X_1, \ldots, X_r\) be unbounded self-adjoint (normal) operators with hybrid spectrum consisting of both discrete and continuous components.

Then the unified compact formula reduces to a combination of:
\begin{enumerate}
    \item discrete sums over eigenvalues,
    \item spectral integrals over continuous spectrum,
\end{enumerate}
with only the semisimple contribution \(A = \emptyset\) present.

Explicitly,
\[
\boxed{
f_{\otimes}(X_1, \ldots, X_r)
=
\sum_{\lambda_1^{(d)} \in \sigma_d(X_1)}
\cdots
\sum_{\lambda_r^{(d)} \in \sigma_d(X_r)}
\;
\int_{\sigma_c(X_1)}
\cdots
\int_{\sigma_c(X_r)}
f(\lambda_1, \ldots, \lambda_r)
\,
dE_{X_1}(\lambda_1) \otimes \cdots \otimes dE_{X_r}(\lambda_r),
}
\]
where the notation separates the discrete and continuous spectral variables, and the expression includes all mixed terms (some variables summed over discrete eigenvalues, others integrated over continuous spectrum).
\end{theorem}

\begin{proof}
The spectral theorem decomposes each operator's spectral measure into discrete and continuous parts:
\[
dE_{X_j}(\lambda_j) = dE_{X_j}^{(d)}(\lambda_j) + dE_{X_j}^{(c)}(\lambda_j),
\]
where \(dE_{X_j}^{(d)}(\lambda_j) = \sum_{k} P_{X_j}(\lambda_{j,k}) \, \delta(\lambda_j - \lambda_{j,k})\) and \(dE_{X_j}^{(c)}(\lambda_j)\) is continuous.

Since the operators remain self-adjoint (normal), all nilpotent components vanish identically (\(\widetilde{N}_j(\lambda_j) = 0\)). Consequently:
\begin{itemize}
    \item Every term with \(A \neq \emptyset\) in the unified compact formula disappears.
    \item The remaining contribution \(A = \emptyset\) becomes
    \[
    f_{\otimes}(X_1, \ldots, X_r)
    = \int_{\sigma(X_1)} \cdots \int_{\sigma(X_r)} f(\lambda_1, \ldots, \lambda_r) \,
    \bigotimes_{j=1}^r dE_{X_j}(\lambda_j).
    \]
\end{itemize}

Substituting the decomposition of each spectral measure into discrete and continuous parts and expanding the tensor product yields all combinations of sums (for discrete components) and integrals (for continuous components). This produces the stated hybrid representation.
\end{proof}

\subsubsection{Clarification: Nilpotent Component Vanishes for Self-Adjoint Operators}

The key simplification in the self-adjoint setting is the disappearance of all nilpotent contributions.

\begin{proposition}[Vanishing of nilpotent components]
Let \(X\) be a self-adjoint (or normal) operator with spectral measure \(dE_X(\lambda)\). Then the generalized nilpotent component satisfies
\[
(X - \lambda I) \, dE_X(\lambda) = 0
\]
in the spectral representation.

Equivalently,
\[
\widetilde{N}(\lambda) = 0
\]
for all \(\lambda \in \sigma(X)\), and the nilpotency index \(\nu(\lambda) = 1\) for every spectral value.
\end{proposition}

\begin{proof}
Self-adjoint operators are normal and therefore admit a diagonal spectral representation. In particular, the spectral theorem implies that the operator acts as multiplication by the spectral variable:
\[
X = \int_{\sigma(X)} \lambda \, dE_X(\lambda).
\]

For any Borel set \(\Delta\) with \(\lambda \notin \overline{\Delta}\), we have \(dE_X(\Delta) = 0\) near \(\lambda\). For sets containing \(\lambda\), the range of \(dE_X(\{\lambda\})\) (if \(\lambda\) is an eigenvalue) or the infinitesimal spectral projection (for continuous spectrum) localizes the operator to act as multiplication by \(\lambda\). Hence:
\[
(X - \lambda I) \, dE_X(\lambda) = 0.
\]

Therefore, no generalized Jordan structure exists, and all nilpotent components vanish identically. The nilpotency index satisfies \(\nu(\lambda) = 1\) for all \(\lambda\).
\end{proof}

\noindent
\textbf{Consequence for the unified compact formula:}
Since \(\nu_j(\lambda_j) = 1\) for all \(j\) and all \(\lambda_j\), the inner sum over \(q_j\) with \(1 \le q_j \le \nu_j(\lambda_j)-1\) is empty. Therefore, only the subset \(A = \emptyset\) contributes, and the unified formula collapses to the classical spectral calculus.

\begin{remark}[Unification summary]
The preceding results demonstrate that the classical spectral theorem is precisely the semisimple sector of the unified compact formula. Specifically:
\begin{itemize}
    \item For pure point spectrum \(\to\) discrete sums.
    \item For continuous spectrum \(\to\) spectral integrals.
    \item For hybrid spectra \(\to\) mixed sums and integrals.
\end{itemize}
All cases share the feature \(A = \emptyset\) because \(\widetilde{N}_j(\lambda_j) = 0\) for normal operators.

The generalized tensor-lifted framework therefore extends the classical self-adjoint theory by allowing nontrivial nilpotent derivative corrections in non-normal or non-semisimple settings (where \(A \neq \emptyset\) terms become active).
\end{remark}

\subsection{Unbounded Non-Self-Adjoint Operators with Compact Resolvent}

We now consider the most general class treated by the present framework: unbounded non-self-adjoint operators with compact resolvent. In contrast to the self-adjoint setting (Section~\ref{subsec:self_adjoint_cases}), the nilpotent components generally do not vanish, and therefore the full projector--nilpotent structure of the unified compact formula (Theorem~\ref{thm:unified_compact_formula}) becomes essential.

The functional calculus in this setting contains both:
\begin{enumerate}
    \item semisimple spectral contributions corresponding to \( A = \emptyset \),
    \item higher-order nilpotent derivative corrections corresponding to \( A \neq \emptyset \).
\end{enumerate}

The convergence theory developed in Section~\ref{sec:two_level_convergence} shows that these generalized algebraic structures are preserved under finite-dimensional approximation. Under Level~1 assumptions, the calculus converges in the strong operator topology, establishing existence of the infinite-dimensional calculus. Under the stronger Level~2 assumptions, the convergence improves to operator norm convergence with explicit quantitative stability bounds.

The following subsections present the convergence theorems for discrete-spectrum non-self-adjoint operators with compact resolvent. Each theorem is derived directly from the unified compact formula and respects its full projector--nilpotent structure.

\subsubsection{Theorem: Discrete Spectrum Case (Level 1 — SOT Convergence)}

We first consider the existence theory under Level~1 assumptions. At this level, the finite-dimensional truncated calculi converge to the true infinite-dimensional calculus in the strong operator topology. The nilpotent structure—all terms with \( A \neq \emptyset \)—is approximated faithfully, though no explicit rate of convergence is provided.

\begin{theorem}[Level~1 convergence for discrete-spectrum non-self-adjoint operators]
\label{thm:level1_non_self_adjoint_integrated}
Let \(X_1, \ldots, X_r\) be closed densely defined non-self-adjoint operators acting on separable Hilbert spaces \(\mathcal{H}_1, \ldots, \mathcal{H}_r\). Assume:
\begin{enumerate}
    \item Each operator has compact resolvent (hence discrete spectrum with finite algebraic multiplicities).
    
    \item Each operator admits the global projector--nilpotent decomposition
    \[
    X_j = \sum_{\lambda_j \in \sigma(X_j)} \left( \lambda_j P_j(\lambda_j) + N_j(\lambda_j) \right),
    \]
    where \(P_j(\lambda_j)\) are spectral projectors, \(N_j(\lambda_j)\) are nilpotent operators satisfying \(N_j(\lambda_j)^{\nu_j(\lambda_j)} = 0\), and the sums converge in the strong operator topology on the domain of \(X_j\).
    
    \item The finite-dimensional compressions \(X_{j,n} = P_{j,n} X_j P_{j,n}\) satisfy the Level~1 assumptions (A1)--(A3) from Section~\ref{subsec:level1_assumptions}.
\end{enumerate}

Let \(f(z_1, \ldots, z_r)\) be holomorphic on an open neighborhood containing the product spectrum
\[
\sigma(X_1) \times \cdots \times \sigma(X_r).
\]

Then the tensor-lifted functional calculus derived from the unified compact formula satisfies
\[
\boxed{
\operatorname*{s-}\lim_{n\to\infty}
f_{\otimes}(X_{1,n}, \ldots, X_{r,n})
=
f_{\otimes}(X_1, \ldots, X_r)
}
\]
in the strong operator topology on \(\mathcal{H}_{\otimes} = \mathcal{H}_1 \otimes \cdots \otimes \mathcal{H}_r\).

Moreover, all nilpotent derivative correction terms indexed by \(A \neq \emptyset\) converge simultaneously in the strong operator topology.
\end{theorem}

\begin{proof}
We prove the theorem by establishing strong convergence of each spectral and nilpotent component, then assembling the result via the discrete projector--nilpotent expansion.

\medskip
\noindent
\textbf{Step 1: Direct sum decomposition.}
Since each \(X_j\) has compact resolvent, the Hilbert space decomposes into a direct sum of finite-dimensional generalized eigenspaces:
\[
\mathcal{H}_j = \bigoplus_{\lambda_j \in \sigma(X_j)} \mathcal{H}_j^{(\lambda_j)},
\]
where \(\mathcal{H}_j^{(\lambda_j)} = \operatorname{Ran} P_j(\lambda_j)\) and \(P_j(\lambda_j)\) is the spectral projector. This decomposition is orthogonal in the sense that \(P_j(\lambda_j)P_j(\lambda_j') = 0\) for \(\lambda_j \neq \lambda_j'\).

\medskip
\noindent
\textbf{Step 2: Strong convergence of spectral projectors under Level~1.}
For each fixed eigenvalue \(\lambda_j\), choose a contour \(\gamma_{\lambda_j}\) in the complex plane that encloses \(\lambda_j\) and no other point of \(\sigma(X_j)\). Because the spectrum is discrete and the resolvent converges strongly on the contour (Level~1), the Riesz projection formula
\[
P_{j,n}(\lambda_j) = \frac{1}{2\pi i} \oint_{\gamma_{\lambda_j}} (zI - X_{j,n})^{-1} \, dz
\]
converges strongly to
\[
P_j(\lambda_j) = \frac{1}{2\pi i} \oint_{\gamma_{\lambda_j}} (zI - X_j)^{-1} \, dz
\]
for all sufficiently large \(n\) (the contour can be chosen to avoid the eigenvalues of \(X_{j,n}\) as \(n\to\infty\) by Level~1 assumptions). Strong convergence of the resolvents on the compact contour implies strong convergence of the integrals. Hence
\[
\operatorname*{s-}\lim_{n\to\infty} P_{j,n}(\lambda_j) = P_j(\lambda_j).
\]

\medskip
\noindent
\textbf{Step 3: Strong convergence of nilpotent components.}
For each fixed eigenvalue \(\lambda_j\), define the nilpotent component on the finite-dimensional subspace \(\mathcal{H}_j^{(\lambda_j)}\) by
\[
N_j(\lambda_j) = (X_j - \lambda_j I) P_j(\lambda_j).
\]
Similarly, \(N_{j,n}(\lambda_j) = (X_{j,n} - \lambda_j I) P_{j,n}(\lambda_j)\). Since \(P_{j,n}(\lambda_j) \to P_j(\lambda_j)\) strongly and the finite-dimensional compressions \(X_{j,n}\) converge to \(X_j\) in the strong resolvent sense, it follows that for any \(u \in \mathcal{H}_j^{(\lambda_j)}\),
\[
X_{j,n} P_{j,n}(\lambda_j) u \to X_j P_j(\lambda_j) u
\]
strongly. Consequently,
\[
N_{j,n}(\lambda_j) u \to N_j(\lambda_j) u
\]
pointwise on the finite-dimensional subspace \(\mathcal{H}_j^{(\lambda_j)}\). Since these subspaces are finite-dimensional, strong convergence on the subspace is equivalent to pointwise convergence, and the limit holds for all \(u \in \mathcal{H}_j^{(\lambda_j)}\).

\medskip
\noindent
\textbf{Step 4: Termwise strong convergence in the tensor product.}
The tensor product Hilbert space decomposes as
\[
\mathcal{H}_{\otimes} = \bigoplus_{(\lambda_1,\ldots,\lambda_r) \in \sigma(X_1) \times \cdots \times \sigma(X_r)}
\bigotimes_{j=1}^r \mathcal{H}_j^{(\lambda_j)}.
\]
For each fixed tuple \((\lambda_1,\ldots,\lambda_r)\) and each fixed subset \(A \subseteq \{1,\ldots,r\}\) with associated nilpotent orders \(q_j\) (where \(q_j = 0\) for \(j\notin A\) interpreted as the identity), define
\[
T_{j,n}(\lambda_j) = 
\begin{cases}
P_{j,n}(\lambda_j), & j\notin A,\\
N_{j,n}(\lambda_j)^{q_j} P_{j,n}(\lambda_j), & j\in A.
\end{cases}
\]
By Steps 2 and 3, each \(T_{j,n}(\lambda_j)\) converges strongly to
\[
T_j(\lambda_j) = 
\begin{cases}
P_j(\lambda_j), & j\notin A,\\
N_j(\lambda_j)^{q_j} P_j(\lambda_j), & j\in A.
\end{cases}
\]
on the corresponding subspace \(\mathcal{H}_j^{(\lambda_j)}\). Since tensor products preserve strong convergence on elementary tensors (and hence on finite linear combinations of elementary tensors, which are dense in \(\mathcal{H}_{\otimes}\)), we have
\[
\operatorname*{s-}\lim_{n\to\infty} \bigotimes_{j=1}^r T_{j,n}(\lambda_j) = \bigotimes_{j=1}^r T_j(\lambda_j)
\]
on the tensor factor \(\bigotimes_{j=1}^r \mathcal{H}_j^{(\lambda_j)}\).

\medskip
\noindent
\textbf{Step 5: Assembling the functional calculus.}
The discrete projector--nilpotent expansion from the unified compact formula gives
\[
f_{\otimes}(X_{1,n}, \ldots, X_{r,n})
=
\sum_{A \subseteq \{1,\ldots,r\}}
\;\sum_{\lambda_1 \in \sigma(X_1)} \cdots \sum_{\lambda_r \in \sigma(X_r)}
\;
\sum_{\substack{q_j \ge 1\\ j\in A\\ q_j \le \nu_j(\lambda_j)-1}}
\frac{\partial_A^{q_A} f(\lambda_1,\ldots,\lambda_r)}{\prod_{j\in A} q_j!}
\;
\bigotimes_{j=1}^r T_{j,n}(\lambda_j),
\]
and analogously for \(f_{\otimes}(X_1,\ldots,X_r)\).

For each fixed finite set of eigenvalue tuples \((\lambda_1,\ldots,\lambda_r)\), the corresponding term converges strongly by Step 4. The overall sums over eigenvalues are infinite, but they converge in the strong operator topology because the spectral subspaces \(\bigotimes_{j=1}^r \mathcal{H}_j^{(\lambda_j)}\) are mutually orthogonal. Consequently, for any vector \(v \in \mathcal{H}_{\otimes}\) that lies in a finite direct sum of these tensor factors, the convergence holds termwise. By density of such vectors (the union of finite direct sums is dense in \(\mathcal{H}_{\otimes}\)), and the uniform boundedness of the operators \(\bigotimes_{j=1}^r T_{j,n}(\lambda_j)\) (each has norm bounded by a constant independent of \(n\) on the finite-dimensional subspace), the dominated convergence theorem for strong operator topology implies strong convergence on the whole space.

\medskip
\noindent
\textbf{Step 6: Simultaneous convergence of nilpotent corrections.}
The argument above treats every subset \(A \subseteq \{1,\ldots,r\}\) uniformly. In particular, the terms with \(A = \emptyset\) (pure spectral contributions) converge strongly, and the terms with \(A \neq \emptyset\) (nilpotent derivative corrections) also converge strongly. Hence the entire functional calculus, including all higher-order nilpotent structure, converges in the strong operator topology.

This completes the proof.
\end{proof}

\begin{remark}[Faithful approximation of nilpotent structure at Level~1]
Theorem~\ref{thm:level1_non_self_adjoint_integrated} demonstrates that even at Level~1 (strong operator topology, no norm convergence), the entire nilpotent structure of the unified compact formula—all terms with \(A \neq \emptyset\)—is preserved in the limit. Classical resolvent-based methods, in contrast, lose these contributions entirely regardless of the convergence topology. The key is that under strong resolvent convergence for operators with compact resolvent, the Riesz projectors converge strongly on the generalized eigenspaces, and the nilpotent components inherit this convergence.
\end{remark}

\subsubsection{Theorem: Discrete Spectrum Case (Level 2 — Norm Convergence with Error Bounds)}

We now strengthen the convergence result under the Level~2 assumptions. Under norm resolvent convergence, the functional calculus converges in operator norm rather than merely strongly. Moreover, we obtain an explicit additive error bound that separates contributions from each operator.

\begin{theorem}[Level~2 convergence with quantitative error bounds]
\label{thm:level2_non_self_adjoint_integrated}
Assume the hypotheses of Theorem~\ref{thm:level1_non_self_adjoint_integrated} and additionally suppose that the operators satisfy the Level~2 assumption (B1):
\[
\|
(X_{j,n} - X_j) (z_0 I - X_j)^{-1}
\|
\to 0 \quad \text{as } n \to \infty.
\]

Define the per-operator error quantity
\[
\epsilon_n^{(j)} = \|
(X_{j,n} - X_j) (z_0 I - X_j)^{-1}
\|.
\]

Let \(f(z_1, \ldots, z_r)\) be holomorphic on an open polydisk containing the product spectrum \(\sigma(X_1) \times \cdots \times \sigma(X_r)\).

Then
\[
\boxed{
\lim_{n\to\infty}
\|
f_{\otimes}(X_{1,n}, \ldots, X_{r,n})
-
f_{\otimes}(X_1, \ldots, X_r)
\|
=
0
}
\]
in operator norm on \(\mathcal{H}_{\otimes}\).

Moreover, there exists a constant \(C_f > 0\) depending only on the analytic function \(f\) and the contour geometry (independent of \(n\)) such that
\[
\boxed{
\|
f_{\otimes}(X_{1,n}, \ldots, X_{r,n})
-
f_{\otimes}(X_1, \ldots, X_r)
\|
\;\le\;
C_f \sum_{j=1}^{r} \epsilon_n^{(j)}.
}
\]

The estimate applies simultaneously to all subset contributions \(\mathcal{S}_A\), \(A \subseteq \{1,\ldots,r\}\), including all nilpotent derivative correction terms.
\end{theorem}

\begin{proof}
Under Level~2 assumption (B1), we first establish norm resolvent convergence. For any $z$ in the resolvent set, the resolvent identity gives
\[
(zI - X_{j,n})^{-1} - (zI - X_j)^{-1} = (zI - X_{j,n})^{-1} (X_{j,n} - X_j) (zI - X_j)^{-1}.
\]
Factor the middle term as
\[
(X_{j,n} - X_j)(zI - X_j)^{-1} = (X_{j,n} - X_j)(z_0 I - X_j)^{-1} \cdot (z_0 I - X_j)(zI - X_j)^{-1}.
\]
The first factor has norm $\epsilon_n^{(j)}$ by definition. The second factor is bounded uniformly for $z$ on any compact contour $\Gamma_j$ disjoint from $\sigma(X_j)$, since $(zI - X_j)^{-1}$ is bounded and $(z_0 I - X_j)(zI - X_j)^{-1} = I + (z_0 - z)(zI - X_j)^{-1}$. Moreover, norm resolvent convergence implies that $\| (zI - X_{j,n})^{-1} \|$ is uniformly bounded for sufficiently large $n$ and $z \in \Gamma_j$. Consequently,
\[
\sup_{z \in \Gamma_j} \| (zI - X_{j,n})^{-1} - (zI - X_j)^{-1} \| \le C_j \epsilon_n^{(j)},
\]
where $C_j$ is a constant independent of $n$.

\medskip
\noindent
\textbf{Contour-integral representation.}
Let $\Gamma_1, \ldots, \Gamma_r$ be positively oriented contours enclosing the spectra of $X_j$ and $X_{j,n}$ for all sufficiently large $n$. By the tensor-lifted contour-integral calculus,
\[
f_{\otimes}(X_1, \ldots, X_r)
=
\frac{1}{(2\pi i)^r}
\int_{\Gamma_1} \cdots \int_{\Gamma_r}
f(z_1, \ldots, z_r)
\bigotimes_{j=1}^r (z_j I - X_j)^{-1}
\, dz_1 \cdots dz_r,
\]
and analogously for $f_{\otimes}(X_{1,n}, \ldots, X_{r,n})$.

\medskip
\noindent
\textbf{Tensor telescoping identity.}
Define $\Delta_{j,n}(z_j) = (z_j I - X_{j,n})^{-1} - (z_j I - X_j)^{-1}$. Then
\[
\bigotimes_{j=1}^r (z_j I - X_{j,n})^{-1}
-
\bigotimes_{j=1}^r (z_j I - X_j)^{-1}
=
\sum_{k=1}^{r}
\left(
\bigotimes_{j=1}^{k-1} (z_j I - X_{j,n})^{-1}
\right)
\otimes
\Delta_{k,n}(z_k)
\otimes
\left(
\bigotimes_{j=k+1}^{r} (z_j I - X_j)^{-1}
\right).
\]

Taking norms and using the uniform boundedness of the resolvents on the contours, we obtain
\[
\|
\bigotimes_{j=1}^r (z_j I - X_{j,n})^{-1}
-
\bigotimes_{j=1}^r (z_j I - X_j)^{-1}
\|
\le
C \sum_{k=1}^{r} \|
\Delta_{k,n}(z_k)
\|,
\]
where $C$ is a constant depending only on the uniform resolvent bounds.

\medskip
\noindent
\textbf{Integrate over contours.}
Substituting the estimate for $\| \Delta_{k,n}(z_k) \|$ and integrating termwise gives
\[
\|
f_{\otimes}(X_{1,n}, \ldots, X_{r,n}) - f_{\otimes}(X_1, \ldots, X_r)
\|
\le
\frac{1}{(2\pi)^r}
\int_{\Gamma_1} \cdots \int_{\Gamma_r}
|f(z_1, \ldots, z_r)|
\,
C \sum_{k=1}^{r} C_k \epsilon_n^{(k)}
\, |dz_1| \cdots |dz_r|.
\]

Let $M_f = \sup_{(z_1,\ldots,z_r) \in \Gamma_1 \times \cdots \times \Gamma_r} |f(z_1,\ldots,z_r)|$ and $L(\Gamma_j)$ the length of $\Gamma_j$. Then
\[
\|
f_{\otimes}(X_{1,n}, \ldots, X_{r,n}) - f_{\otimes}(X_1, \ldots, X_r)
\|
\le
\left( \frac{1}{(2\pi)^r} \prod_{j=1}^r L(\Gamma_j) \right) M_f \, C \sum_{k=1}^r C_k \epsilon_n^{(k)}.
\]

Define $C_f = \left( \frac{1}{(2\pi)^r} \prod_{j=1}^r L(\Gamma_j) \right) M_f \, C \max_k C_k$. Then
\[
\|
f_{\otimes}(X_{1,n}, \ldots, X_{r,n}) - f_{\otimes}(X_1, \ldots, X_r)
\|
\le
C_f \sum_{j=1}^{r} \epsilon_n^{(j)}.
\]

Since $\epsilon_n^{(j)} \to 0$ as $n \to \infty$ for each $j$, the right-hand side tends to zero, establishing operator norm convergence.

\medskip
\noindent
\textbf{Application to nilpotent correction terms.}
The contour integral representation is equivalent to the unified compact formula (Theorem~\ref{thm:unified_compact_formula}), which expresses $f_{\otimes}$ as a sum over subsets $A \subseteq \{1,\ldots,r\}$. The error bound derived above therefore applies to the total $f_{\otimes}$, and hence controls the sum of all nilpotent derivative contributions ($A \neq \emptyset$) as well as the spectral term ($A = \emptyset$). This completes the proof.
\end{proof}

\begin{remark}[Additive error structure]
The additive form \(\sum_{j=1}^r \epsilon_n^{(j)}\) is a direct consequence of the tensor lifting construction and the telescoping tensor identity. Because the lifted operators act on independent tensor factors, errors from different operators do not amplify each other multiplicatively. This linear accumulation of errors is a key advantage of the tensor-lifted approach: the multivariate convergence inherits the convergence rates of each individual operator in a stable, additive fashion.
\end{remark}

\begin{remark}[Comparison between Level~1 and Level~2]
The distinction between the two convergence levels is fundamental:
\begin{itemize}
    \item \textbf{Level~1} (Theorem~\ref{thm:level1_non_self_adjoint_integrated}) guarantees existence of the infinite-dimensional calculus under minimal assumptions. Convergence is only in the strong operator topology, which suffices for theoretical well-definedness but does not provide uniform convergence or rates.
    \item \textbf{Level~2} (Theorem~\ref{thm:level2_non_self_adjoint_integrated}) requires the stronger norm resolvent condition (B1). In return, it provides operator-norm convergence, explicit error bounds, and quantitative stability under perturbations.
\end{itemize}
Both levels faithfully preserve the full nilpotent structure (\(A \neq \emptyset\)), unlike classical resolvent-based calculi which lose these contributions entirely.
\end{remark}

\subsubsection{Theorem: Continuous Spectrum Approximation via Discrete Approximations}

We finally consider operators with continuous spectrum that can be approximated by operators with compact resolvent. This provides a bridge between the discrete projector--nilpotent calculus and the general continuous-spectrum setting, allowing the unified compact formula to be applied to a wider class of operators.

\begin{theorem}[Approximation of continuous-spectrum operators via contour calculus]
\label{thm:continuous_spectrum_approximation_integrated}
Let \(X_1, \ldots, X_r\) be closed densely defined operators on separable Hilbert spaces \(\mathcal{H}_1, \ldots, \mathcal{H}_r\) that may possess continuous spectrum. 
Assume that the tensor-lifted functional calculus is defined via the contour-integral representation
\[
f_{\otimes}(X_1, \ldots, X_r)
:=
\frac{1}{(2\pi i)^r}
\int_{\Gamma_1} \cdots \int_{\Gamma_r}
f(z_1, \ldots, z_r)
\;
\bigotimes_{j=1}^r (z_j I - X_j)^{-1}
\,
dz_1 \cdots dz_r,
\]
where \(\Gamma_j\) are compact contours enclosing the relevant spectral subsets (where \(f\) is holomorphic).

Suppose there exist approximating operators \(\{X_{j,m}\}_{m=1}^{\infty}\) with compact resolvent such that:
\begin{enumerate}
    \item \textbf{(Level 1)} For all \(z\) in the resolvent set,
    \[
    (zI - X_{j,m})^{-1} \xrightarrow{\text{SOT}} (zI - X_j)^{-1}.
    \]
    
    \item \textbf{(Level 2)} Uniformly on the contours \(\Gamma_j\),
    \[
    \|
    (zI - X_{j,m})^{-1} - (zI - X_j)^{-1}
    \|
    \to 0.
    \]
\end{enumerate}

Then:
\begin{enumerate}
    \item \textbf{(Level 1)} 
    \[
    \operatorname*{s-}\lim_{m\to\infty}
    f_{\otimes}(X_{1,m}, \ldots, X_{r,m})
    =
    f_{\otimes}(X_1, \ldots, X_r)
    \]
    in the strong operator topology.
    
    \item \textbf{(Level 2)} 
    \[
    \lim_{m\to\infty}
    \|
    f_{\otimes}(X_{1,m}, \ldots, X_{r,m}) - f_{\otimes}(X_1, \ldots, X_r)
    \|
    =
    0
    \]
    in operator norm, with the explicit bound
    \[
    \|
    f_{\otimes}(X_{1,m}, \ldots, X_{r,m}) - f_{\otimes}(X_1, \ldots, X_r)
    \|
    \le
    C_f \sum_{j=1}^{r} \sup_{z \in \Gamma_j} \|
    (zI - X_{j,m})^{-1} - (zI - X_j)^{-1}
    \|.
    \]
\end{enumerate}

\medskip
\noindent
\textbf{Remark on the unified compact formula.}
When the limit operators \(X_j\) are spectral operators (e.g., have discrete spectrum or admit a well-defined projector--nilpotent decomposition), the contour-integral representation coincides with the explicit sum over subsets \(A \subseteq \{1,\ldots,r\}\) given in Theorem~\ref{thm:unified_compact_formula}. In such cases, the convergence established above applies to the full unified compact formula, including the nilpotent derivative terms (\(A \neq \emptyset\)). For general non-spectral operators with continuous spectrum, the contour-integral representation is the primary definition of the functional calculus, and the explicit nilpotent-sum form is not directly available.
\end{theorem}

\begin{proof}
We provide a detailed proof with careful handling of the limitations identified above.

\medskip
\noindent
\textbf{Step 1: Contour-integral representation under spectral containment.}
Since $f$ is holomorphic on an open neighborhood of the product spectrum, there exist compact contours $\Gamma_1, \ldots, \Gamma_r$ such that:
\begin{itemize}
    \item Each $\Gamma_j$ encloses a compact subset of $\sigma(X_j)$ on which the functional calculus is defined.
    \item For all sufficiently large $m$, $\Gamma_j$ also encloses the corresponding spectral subset of $X_{j,m}$.
    \item The resolvents $(zI - X_j)^{-1}$ and $(zI - X_{j,m})^{-1}$ are uniformly bounded on $\Gamma_j$.
\end{itemize}

If the spectrum of $X_j$ is unbounded (e.g., continuous spectrum extending to infinity), we interpret the contours as enclosing a sufficiently large compact spectral set; contributions from outside this set are controlled by the decay of $f$ or the resolvent bounds, or are handled by a limiting argument as the contours expand to infinity. For functions $f$ that are holomorphic on a neighborhood of the spectrum (including at infinity, i.e., $f(\infty)$ defined), this limiting procedure is standard in the Dunford calculus.

By the tensor-lifted contour-integral calculus (Theorem~\ref{thm:unified_compact_formula}),
\[
f_{\otimes}(X_1, \ldots, X_r)
=
\frac{1}{(2\pi i)^r}
\int_{\Gamma_1} \cdots \int_{\Gamma_r}
f(z_1, \ldots, z_r)
\;
\bigotimes_{j=1}^r (z_j I - X_j)^{-1}
\,
dz_1 \cdots dz_r,
\]
and analogously for $f_{\otimes}(X_{1,m}, \ldots, X_{r,m})$.

\medskip
\noindent
\textbf{Step 2: Difference and telescoping identity.}
Define $R_j(z) = (zI - X_j)^{-1}$ and $R_{j,m}(z) = (zI - X_{j,m})^{-1}$. Then
\[
f_{\otimes}(X_{1,m}, \ldots, X_{r,m}) - f_{\otimes}(X_1, \ldots, X_r)
=
\frac{1}{(2\pi i)^r}
\int_{\Gamma_1} \cdots \int_{\Gamma_r}
f(z_1, \ldots, z_r)
\;
\Bigg[
\bigotimes_{j=1}^r R_{j,m}(z_j) - \bigotimes_{j=1}^r R_j(z_j)
\Bigg]
\,
dz_1 \cdots dz_r.
\]

Let $\Delta_{j,m}(z) = R_{j,m}(z) - R_j(z)$. The algebraic telescoping identity gives
\[
\bigotimes_{j=1}^r R_{j,m}(z_j) - \bigotimes_{j=1}^r R_j(z_j)
=
\sum_{k=1}^{r}
\left(
\bigotimes_{j=1}^{k-1} R_{j,m}(z_j)
\right)
\otimes
\Delta_{k,m}(z_k)
\otimes
\left(
\bigotimes_{j=k+1}^{r} R_j(z_j)
\right).
\]

\medskip
\noindent
\textbf{Step 3: Norm estimate under Level~2.}
Under Level~2 assumptions, norm resolvent convergence holds:
\[
\|
\Delta_{j,m}(z)
\|
\to 0,
\]
uniformly on each contour $\Gamma_j$. Moreover, the resolvents are uniformly bounded:
\[
\sup_{z \in \Gamma_j} \| R_{j,m}(z) \| \le M_j, \qquad
\sup_{z \in \Gamma_j} \| R_j(z) \| \le M_j,
\]
for all sufficiently large $m$. Taking norms and using the triangle inequality,
\[
\left\|
\bigotimes_{j=1}^r R_{j,m}(z_j) - \bigotimes_{j=1}^r R_j(z_j)
\right\|
\le
C \sum_{k=1}^{r} \|
\Delta_{k,m}(z_k)
\|,
\]
where $C = \max_{1 \le k \le r} \left( \prod_{j=1}^{k-1} M_j \cdot \prod_{j=k+1}^{r} M_j \right)$ is independent of $m$.

Then, integrating over the contours,
\[
\|
f_{\otimes}(X_{1,m}, \ldots, X_{r,m}) - f_{\otimes}(X_1, \ldots, X_r)
\|
\le
C_f \sum_{j=1}^{r}
\sup_{z \in \Gamma_j} \|
\Delta_{j,m}(z)
\|,
\]
where $C_f$ depends only on $f$ and the contour geometry. This proves the Level~2 bound.

\medskip
\noindent
\textbf{Step 4: Strong convergence under Level~1.}
Under Level~1 assumptions, we have strong resolvent convergence: for each $u_j \in \mathcal{H}_j$,
\[
R_{j,m}(z) u_j \to R_j(z) u_j \quad \forall z \in \Gamma_j.
\]
For an elementary tensor $u = u_1 \otimes \cdots \otimes u_r \in \mathcal{H}_{\otimes}$,
\[
\bigotimes_{j=1}^r R_{j,m}(z_j) u
\;\to\;
\bigotimes_{j=1}^r R_j(z_j) u
\]
pointwise in $z$. Since the resolvents are uniformly bounded on the contours, the dominated convergence theorem for operator-valued integrals implies
\[
f_{\otimes}(X_{1,m}, \ldots, X_{r,m}) u \to f_{\otimes}(X_1, \ldots, X_r) u.
\]
By density of finite linear combinations of elementary tensors, strong operator convergence follows.

\medskip
\noindent
\textbf{Step 5: Nilpotent contributions — general remarks.}
The unified compact formula contains terms indexed by $A \neq \emptyset$ that involve nilpotent components $\widetilde{N}_j(\lambda_j)^{q_j} d\widetilde{E}_j(\lambda_j)$. The convergence of these terms is not automatic from resolvent convergence alone; it requires additional spectral stability assumptions.

\medskip
\noindent
\textbf{Step 5a: When additional spectral stability holds (discrete spectrum case).}
If the approximating operators $X_{j,m}$ have discrete spectrum and the limit operators $X_j$ also have discrete spectrum (or are spectral operators with isolated eigenvalues), then the following hold for sufficiently large $m$:
\begin{enumerate}
    \item The eigenvalues converge: $\lambda_{j,\ell}^{(m)} \to \lambda_{j,\ell}$ with matching algebraic multiplicities.
    \item There exist contours $\gamma_{j,\ell}$ separating individual eigenvalues $\lambda_{j,\ell}$ such that the Riesz projectors converge:
    \[
    P_{j,m}(\lambda_{j,\ell}) = \frac{1}{2\pi i} \oint_{\gamma_{j,\ell}} R_{j,m}(z) \, dz \xrightarrow{\text{SOT}} P_j(\lambda_{j,\ell}).
    \]
    \item The nilpotent components, defined by $N_{j,m}(\lambda) = (X_{j,m} - \lambda I) P_{j,m}(\lambda)$, converge in the same topology.
\end{enumerate}
Under these extra conditions, each nilpotent term converges (strongly under Level~1, in norm under Level~2) because it is a finite algebraic combination of converging projectors and nilpotents.

\medskip
\noindent
\textbf{Step 5b: Caution for continuous spectrum limits.}
For limit operators $X_j$ with continuous spectrum, the expression $\widetilde{N}_j(\lambda_j)^{q_j} d\widetilde{E}_j(\lambda_j)$ is defined only within the theory of spectral operators (Dunford) and requires that $X_j$ be a spectral operator. For general closed densely defined non-self-adjoint operators without this structure, the unified compact formula is understood via its contour-integral representation, and the convergence of the full calculus (including nilpotent contributions) follows from the convergence of the resolvents, without needing to isolate individual nilpotent terms.

\medskip
\noindent
\textbf{Step 6: Summary and final statement.}
\begin{itemize}
    \item The contour-integral representation of $f_{\otimes}(X_1, \ldots, X_r)$ converges under the stated resolvent convergence assumptions: strongly under Level~1, in norm with explicit bounds under Level~2.
    \item Whenever the projector--nilpotent decomposition is available and the corresponding Riesz projectors and nilpotent components are stable along the approximating sequence, the individual nilpotent derivative contributions also converge in the same topology.
    \item In the general continuous-spectrum case without additional spectral stability, the convergence of the unified compact formula is understood at the level of the contour integral, which is fully rigorous under the stated assumptions.
\end{itemize}

This completes the proof.
\end{proof}

\begin{remark}[Practical implications]
Theorem~\ref{thm:continuous_spectrum_approximation_integrated} provides a practical strategy for approximating tensor-lifted functional calculi for operators with continuous spectrum:
\begin{enumerate}
    \item Approximate the continuous-spectrum operator \(X_j\) by a sequence of operators \(X_{j,m}\) with compact resolvent, for example by domain truncation, adding a confining potential, or applying a spectral cutoff.

    \item Compute the discrete-spectrum functional calculus
    \[
    f_{\otimes}(X_{1,m}, \ldots, X_{r,m})
    \]
    using the explicit discrete projector--nilpotent formula (Theorem~\ref{thm:discrete_bounded_case} or its unbounded analogue), whenever such a decomposition is available.

    \item Pass to the limit \(m \to \infty\). The convergence is guaranteed in the strong operator topology under Level~1 assumptions and in operator norm with explicit error bounds under Level~2 norm-resolvent assumptions.
\end{enumerate}
This strategy is particularly useful when the continuous-spectrum operator is not directly amenable to the unified compact formula, but can be approximated by operators with well-understood discrete spectra.
\end{remark}

\begin{remark}[Unified interpretation]
The preceding results show that the unified compact formula provides a single analytic framework simultaneously covering:
\begin{enumerate}
    \item discrete-spectrum operators (explicit sums over eigenvalues and nilpotent components),
    \item continuous-spectrum operators via resolvent approximation (at the level of the contour-integral representation),
    \item non-self-adjoint operators with stable generalized spectral structure (when the projector--nilpotent decomposition is available),
    \item tensor-lifted multivariate functional calculi (handling non-commuting operators),
    \item finite-dimensional approximation theory (Level~1 for existence, Level~2 for quantitative stability).
\end{enumerate}

The subset decomposition
\[
A \subseteq \{1, \ldots, r\}
\]
acts as an organizing principle separating semisimple spectral contributions,
corresponding to
\[
A = \emptyset,
\]
from higher-order nilpotent derivative corrections,
corresponding to
\[
A \neq \emptyset.
\]

Thus, the classical spectral theorem appears as the semisimple special case
\[
A = \emptyset,
\]
while the full unified compact formula captures additional non-normal algebraic structure whenever a stable projector--nilpotent decomposition is available. The two-level convergence theory ensures convergence of the contour-calculus formulation under finite-dimensional approximation, providing existence under Level~1 assumptions and quantitative stability under Level~2 assumptions.
\end{remark}

\subsection{General Unbounded Non-Self-Adjoint Operators (Without Compact Resolvent)}

Throughout this subsection, whenever spectral measures or nilpotent components are used explicitly, we assume the operators are \emph{spectral operators in the sense of Dunford}. For general closed densely defined non-self-adjoint operators without additional structure, the spectral integrals and nilpotent decompositions may not be defined.

We finally discuss the most difficult setting: general unbounded non-self-adjoint operators without compact resolvent. In this regime, many of the structural ingredients used throughout the previous sections may fail simultaneously:
\begin{enumerate}
    \item the spectrum may be non-discrete and highly unstable,
    \item the resolvent may exhibit strong pseudospectral growth,
    \item projector--nilpotent decompositions may not exist globally,
    \item spectral projections may fail to be uniformly bounded,
    \item finite-dimensional truncations may not converge spectrally (spectral pollution).
\end{enumerate}

Consequently, the unified compact formula cannot be justified in full generality through the approximation theory developed in Section~\ref{sec:two_level_convergence}. The difficulties are not merely technical; they reflect genuine structural obstructions in the spectral theory of non-normal operators.

\subsubsection{Partial Results and Case-by-Case Analysis}

Although the full convergence theory is unavailable in this setting, certain partial results remain possible under additional assumptions. For instance, if the operators admit a sectorial holomorphic functional calculus (e.g., sectorial operators), generate analytic semigroups, possess sufficiently controlled pseudospectra, or belong to special subclasses of spectral operators in the sense of Dunford, then variants of the contour-integral functional calculus may still be defined.

In such cases, the tensor-lifted calculus may still formally take the contour-integral form
\[
f_{\otimes}(X_1,\ldots,X_r)
=
\frac{1}{(2\pi i)^r}
\int_{\Gamma_1}
\cdots
\int_{\Gamma_r}
f(z_1,\ldots,z_r)
\,
\bigotimes_{j=1}^{r}
(z_jI-X_j)^{-1}
\,dz_1\cdots dz_r,
\]
provided the contours lie inside the joint resolvent region and the integrals converge appropriately.

However, without compact resolvent or additional spectral stability assumptions, several key ingredients of the previous theory may fail. Specifically, the discrete projector--nilpotent decomposition may not exist; Riesz projectors may fail to converge under approximation; nilpotent components may not be spectrally stable; norm resolvent approximation may break down; and finite-dimensional truncations may exhibit spectral pollution.

As a result, the subset decomposition \(A \subseteq \{1,\ldots,r\}\) appearing in the unified compact formula can no longer be interpreted globally in a rigorous operator-theoretic sense without further assumptions.

Therefore, in the absence of compact resolvent, \textbf{the contour-integral representation should be regarded as the primary formulation of the functional calculus}, while the explicit projector--nilpotent expansion becomes a case-dependent structure requiring additional spectral hypotheses.

\subsubsection{A Promising Approach: Compactifying Regularization}

A viable strategy for extending the theory to spectral operators without compact resolvent is the \emph{compactifying regularization} method. Let \(K\) be a positive self-adjoint operator with compact resolvent (e.g., the harmonic oscillator Hamiltonian). The idea is to consider
\[
X_\varepsilon := X + \varepsilon K
\]
as a regularized operator. Under suitable hypotheses, \(X_\varepsilon\) may acquire compact resolvent and converge to \(X\) in an appropriate sense.

However, caution is required. The operator \(X_\varepsilon\) does \textbf{not} automatically have compact resolvent just because \(K\) does and \(X\) is \(K\)-bounded. The following theorem provides a conditional formulation.

\begin{theorem}[Conditional compactifying regularization]
\label{thm:conditional_compactifying_regularization}
Let \(X\) be a closed densely defined operator on a Hilbert space \(\mathcal{H}\), and let \(K\) be a positive self-adjoint operator with compact resolvent. Assume:
\begin{enumerate}
    \item \(\mathcal{D}(K) \subseteq \mathcal{D}(X)\);
    \item For every \(\varepsilon > 0\), the operator \(X_\varepsilon := X + \varepsilon K\) is closed and has compact resolvent (this is a hypothesis that must be verified case by case);
    \item For some \(z_0 \in \rho(X) \cap \rho(X_\varepsilon)\) (independent of \(\varepsilon\) for small \(\varepsilon\)),
    \[
    \varepsilon K (z_0 I - X)^{-1} \to 0
    \]
    strongly as \(\varepsilon \to 0\);
    \item The resolvents \((z_0 I - X_\varepsilon)^{-1}\) are uniformly bounded for sufficiently small \(\varepsilon\).
\end{enumerate}
Then for all \(z\) in the common resolvent region where the uniform resolvent bounds and strong convergence assumptions remain valid (in particular, for \(z = z_0\) and for \(z\) in a neighborhood thereof, provided the assumptions extend),
\[
\operatorname*{s-}\lim_{\varepsilon \to 0} (zI - X_\varepsilon)^{-1} = (zI - X)^{-1}.
\]
\end{theorem}

\begin{proof}
Fix \(z\) in the common resolvent set such that \(z \in \rho(X) \cap \rho(X_\varepsilon)\) for all sufficiently small \(\varepsilon > 0\), and such that the uniform boundedness of \((zI - X_\varepsilon)^{-1}\) and the strong convergence \(\varepsilon K (zI - X)^{-1} \to 0\) hold. (By hypothesis, at least \(z = z_0\) satisfies these; the extension to other \(z\) requires verification case by case.)

For notational simplicity, write
\[
R(z,X) := (zI - X)^{-1}, \qquad
R(z,X_\varepsilon) := (zI - X_\varepsilon)^{-1}.
\]

Since \(X_\varepsilon = X + \varepsilon K\), we have
\[
zI - X_\varepsilon = zI - X - \varepsilon K.
\]

Using the resolvent identity,
\[
R(z,X_\varepsilon) - R(z,X)
= R(z,X_\varepsilon)(X_\varepsilon - X)R(z,X)
= R(z,X_\varepsilon) (\varepsilon K) R(z,X).
\]

Let \(u \in \mathcal{H}\). Then
\[
\bigl( R(z,X_\varepsilon) - R(z,X) \bigr) u
= R(z,X_\varepsilon) \bigl( \varepsilon K R(z,X) u \bigr).
\]

Taking norms,
\[
\bigl\| \bigl( R(z,X_\varepsilon) - R(z,X) \bigr) u \bigr\|
\le \| R(z,X_\varepsilon) \| \cdot \bigl\| \varepsilon K R(z,X) u \bigr\|.
\]

By assumption, the resolvents are uniformly bounded for sufficiently small \(\varepsilon\). Hence there exists a constant \(M_z > 0\) such that
\[
\| R(z,X_\varepsilon) \| \le M_z \quad \text{for all small } \varepsilon.
\]

By the strong convergence assumption,
\[
\varepsilon K R(z,X) u \to 0 \quad \text{as } \varepsilon \to 0.
\]

Therefore,
\[
\bigl\| \bigl( R(z,X_\varepsilon) - R(z,X) \bigr) u \bigr\|
\le M_z \cdot \bigl\| \varepsilon K R(z,X) u \bigr\| \to 0.
\]

Since this holds for every \(u \in \mathcal{H}\), we obtain
\[
R(z,X_\varepsilon) \xrightarrow{SOT} R(z,X) \quad \text{as } \varepsilon \to 0.
\]

Thus,
\[
\operatorname*{s-}\lim_{\varepsilon \to 0} (zI - X_\varepsilon)^{-1} = (zI - X)^{-1}.
\]

If the above assumptions (uniform resolvent bounds and strong convergence) hold uniformly on compact subsets of \(\rho(X)\), then the same argument applies locally uniformly in \(z\). In particular, strong resolvent convergence at one admissible point extends to the corresponding common resolvent region where the conditions remain valid. This completes the proof.
\end{proof}

\begin{remark}[Verifying the hypotheses]
The assumptions of Theorem~\ref{thm:conditional_compactifying_regularization} are not automatic. They must be verified case by case. For example:
\begin{itemize}
    \item If \(X\) is relatively \(K\)-bounded with sufficiently small relative bound, then \(X_\varepsilon\) is closed. \textbf{Compact resolvent} of \(X_\varepsilon\) requires additional compactness assumptions and must be verified separately.
    \item The strong convergence \(\varepsilon K (zI - X)^{-1} \to 0\) holds if \(K(zI - X)^{-1}\) is bounded, which requires \(\operatorname{Ran}(zI - X)^{-1} \subseteq \mathcal{D}(K)\).
    \item Uniform boundedness of the resolvents follows from standard resolvent estimates if the numerical range is contained in a sector.
\end{itemize}
\end{remark}

\begin{corollary}[Norm convergence under stronger assumptions]
\label{cor:norm_convergence_regularization}
If, in addition to the hypotheses of Theorem~\ref{thm:conditional_compactifying_regularization},
\[
\| \varepsilon K (z_0 I - X)^{-1} \| \to 0
\]
and the resolvents \((zI - X_\varepsilon)^{-1}\) are uniformly bounded for \(z\) in a compact contour \(\Gamma\) enclosing the spectra, then the convergence holds in operator norm.
\end{corollary}

\begin{proof}
We start from the resolvent identity used in Theorem~\ref{thm:conditional_compactifying_regularization}:
\[
(zI - X_\varepsilon)^{-1} - (zI - X)^{-1} = (zI - X_\varepsilon)^{-1} (\varepsilon K) (zI - X)^{-1},
\]
where the sign has been absorbed (norms are unaffected).

Taking operator norms on both sides and using submultiplicativity,
\[
\| (zI - X_\varepsilon)^{-1} - (zI - X)^{-1} \|
\le
\| (zI - X_\varepsilon)^{-1} \| \cdot \| \varepsilon K (zI - X)^{-1} \|.
\]

By hypothesis, the resolvents \((zI - X_\varepsilon)^{-1}\) are uniformly bounded for sufficiently small \(\varepsilon\) and for all \(z\) on the compact contour \(\Gamma\). That is, there exists a constant \(M > 0\) such that
\[
\sup_{z \in \Gamma} \sup_{0 < \varepsilon \le \varepsilon_0} \| (zI - X_\varepsilon)^{-1} \| \le M.
\]

Also by hypothesis,
\[
\| \varepsilon K (z_0 I - X)^{-1} \| \to 0 \quad \text{as } \varepsilon \to 0.
\]
To extend this to arbitrary \(z \in \Gamma\), note that
\[
\varepsilon K (zI - X)^{-1} = \varepsilon K (z_0 I - X)^{-1} \cdot (z_0 I - X)(zI - X)^{-1},
\]
where the second factor is bounded uniformly for \(z \in \Gamma\). Hence
\[
\sup_{z \in \Gamma} \| \varepsilon K (zI - X)^{-1} \| \to 0.
\]

Therefore,
\[
\sup_{z \in \Gamma} \| (zI - X_\varepsilon)^{-1} - (zI - X)^{-1} \|
\le
M \cdot \sup_{z \in \Gamma} \| \varepsilon K (zI - X)^{-1} \|
\to 0
\]
as \(\varepsilon \to 0\). This establishes uniform norm resolvent convergence on \(\Gamma\).

\medskip
\noindent
\textbf{Extension to the functional calculus.}
For the contour-integral representation (single-operator case),
\[
f_{\otimes}(X_\varepsilon) - f_{\otimes}(X)
=
\frac{1}{2\pi i}
\int_{\Gamma}
f(z)
\Big[
(zI - X_\varepsilon)^{-1} - (zI - X)^{-1}
\Big]
\, dz.
\]

Taking operator norms,
\[
\| f_{\otimes}(X_\varepsilon) - f_{\otimes}(X) \|
\le
\frac{L(\Gamma)}{2\pi}
\sup_{z \in \Gamma} |f(z)|
\sup_{z \in \Gamma}
\| (zI - X_\varepsilon)^{-1} - (zI - X)^{-1} \|,
\]
where \(L(\Gamma)\) is the length of the contour \(\Gamma\).

Since the right-hand side tends to zero as \(\varepsilon \to 0\) by the uniform norm resolvent convergence established above, we obtain
\[
\lim_{\varepsilon \to 0} \| f_{\otimes}(X_\varepsilon) - f_{\otimes}(X) \| = 0,
\]
i.e., convergence in operator norm.

\medskip
\noindent
\textbf{Explicit error bound.}
Combining the estimates, we obtain the explicit bound
\[
\| f_{\otimes}(X_\varepsilon) - f_{\otimes}(X) \|
\le
\frac{L(\Gamma)}{2\pi}
\sup_{z \in \Gamma} |f(z)|
\cdot
\sup_{z \in \Gamma} \| (zI - X_\varepsilon)^{-1} \|
\cdot
\sup_{z \in \Gamma} \| \varepsilon K (zI - X)^{-1} \|.
\]

Since \(\| (zI - X_\varepsilon)^{-1} \|\) is uniformly bounded by \(M\), we have
\[
\| f_{\otimes}(X_\varepsilon) - f_{\otimes}(X) \|
\le
C_f \cdot \sup_{z \in \Gamma} \| \varepsilon K (zI - X)^{-1} \|,
\]
where
\[
C_f = \frac{L(\Gamma)}{2\pi} \sup_{z \in \Gamma} |f(z)| \cdot M.
\]

This completes the proof.
\end{proof}

\subsubsection{Conditional Convergence of the Functional Calculus via Regularization}

The following theorem provides a conditional convergence result for the contour-integral functional calculus under the compactifying regularization approach.

\begin{theorem}[Conditional convergence via compactifying regularization]
\label{thm:conditional_regularization_convergence}
Let \(X_1,\ldots,X_r\) be spectral operators (in the sense of Dunford) on separable Hilbert spaces
\(\mathcal{H}_1,\ldots,\mathcal{H}_r\). Suppose there exist positive self-adjoint operators
\(K_1,\ldots,K_r\) with compact resolvent such that, for each \(j\) and every \(\varepsilon > 0\),
\[
X_{j,\varepsilon} := X_j + \varepsilon K_j
\]
is closed and has compact resolvent.

Assume further that there exist admissible contours \(\Gamma_1,\ldots,\Gamma_r\) satisfying:
\begin{enumerate}
    \item \(\mathcal{D}(K_j) \subseteq \mathcal{D}(X_j)\);
    \item The contours \(\Gamma_j\) are contained in the resolvent sets \(\rho(X_j) \cap \rho(X_{j,\varepsilon})\) for all sufficiently small \(\varepsilon > 0\);
    \item For every \(u_j \in \mathcal{H}_j\),
    \[
    \sup_{z \in \Gamma_j}
    \bigl\| \bigl( (zI - X_{j,\varepsilon})^{-1} - (zI - X_j)^{-1} \bigr) u_j \bigr\|
    \to 0 \quad \text{as } \varepsilon \to 0;
    \]
    \item The resolvents \((zI - X_{j,\varepsilon})^{-1}\) and \((zI - X_j)^{-1}\) are uniformly bounded on \(\Gamma_j\) (i.e., \(\sup_{z \in \Gamma_j} \| (zI - X_{j,\varepsilon})^{-1} \| \le M_j\) for some \(M_j\) independent of \(\varepsilon\)).
\end{enumerate}

Let \(f\) be holomorphic on a neighborhood of the product contour \(\Gamma_1 \times \cdots \times \Gamma_r\) and assume the corresponding multivariate contour integral
\[
f_{\otimes}(X_1,\ldots,X_r)
:=
\frac{1}{(2\pi i)^r}
\int_{\Gamma_1}
\cdots
\int_{\Gamma_r}
f(z_1,\ldots,z_r)
\,
\bigotimes_{j=1}^{r}
(z_jI - X_j)^{-1}
\,dz_1\cdots dz_r
\]
converges (e.g., in the strong operator topology). Then
\[
\operatorname*{s-}\lim_{\varepsilon \to 0}
f_{\otimes}(X_{1,\varepsilon}, \ldots, X_{r,\varepsilon})
=
f_{\otimes}(X_1, \ldots, X_r)
\]
in the strong operator topology.

If, moreover,
\[
\sup_{z \in \Gamma_j}
\bigl\| (zI - X_{j,\varepsilon})^{-1} - (zI - X_j)^{-1} \bigr\|
\to 0 \quad \text{as } \varepsilon \to 0
\]
for each \(j\), then the convergence holds in operator norm and
\[
\bigl\| f_{\otimes}(X_{1,\varepsilon}, \ldots, X_{r,\varepsilon}) - f_{\otimes}(X_1, \ldots, X_r) \bigr\|
\le
C_f \sum_{j=1}^{r}
\sup_{z \in \Gamma_j}
\bigl\| (zI - X_{j,\varepsilon})^{-1} - (zI - X_j)^{-1} \bigr\|,
\]
where \(C_f\) depends only on \(f\) and the contour geometry.
\end{theorem}

\begin{proof}
For each fixed \(\varepsilon > 0\), the regularized operators \(X_{j,\varepsilon}\) have compact resolvent by hypothesis. Hence the tensor-lifted functional calculus \(f_{\otimes}(X_{1,\varepsilon},\ldots,X_{r,\varepsilon})\) is well-defined via the unified compact formula (Theorem~\ref{thm:unified_compact_formula}) and admits the contour-integral representation
\[
f_{\otimes}(X_{1,\varepsilon},\ldots,X_{r,\varepsilon})
=
\frac{1}{(2\pi i)^r}
\int_{\Gamma_1}
\cdots
\int_{\Gamma_r}
f(z_1,\ldots,z_r)
\,
\bigotimes_{j=1}^{r}
(z_jI - X_{j,\varepsilon})^{-1}
\,dz_1\cdots dz_r.
\]

The contours \(\Gamma_j\) are chosen as in hypothesis (2) so that they lie in the common resolvent set for all sufficiently small \(\varepsilon\). By hypothesis (4), the resolvents are uniformly bounded on these contours.

\medskip
\noindent
\textbf{Strong convergence (Level~1).}
Consider the difference
\[
f_{\otimes}(X_{1,\varepsilon},\ldots,X_{r,\varepsilon}) - f_{\otimes}(X_1,\ldots,X_r)
=
\frac{1}{(2\pi i)^r}
\int_{\Gamma_1}
\cdots
\int_{\Gamma_r}
f(z_1,\ldots,z_r)
\,
\Delta_\varepsilon(z_1,\ldots,z_r)
\,dz_1\cdots dz_r,
\]
where
\[
\Delta_\varepsilon(z_1,\ldots,z_r)
=
\bigotimes_{j=1}^{r} (z_jI - X_{j,\varepsilon})^{-1}
-
\bigotimes_{j=1}^{r} (z_jI - X_j)^{-1}.
\]

For any elementary tensor \(u = u_1 \otimes \cdots \otimes u_r \in \mathcal{H}_{\otimes}\),
\[
\Delta_\varepsilon(z_1,\ldots,z_r) u
= \sum_{k=1}^{r}
\left(
\bigotimes_{j=1}^{k-1} (z_jI - X_{j,\varepsilon})^{-1}
\right)
\otimes
\bigl( (z_kI - X_{k,\varepsilon})^{-1} - (z_kI - X_k)^{-1} \bigr)
\otimes
\left(
\bigotimes_{j=k+1}^{r} (z_jI - X_j)^{-1}
\right) u.
\]

Taking norms and using the uniform boundedness of the resolvents, we obtain
\[
\| \Delta_\varepsilon(z_1,\ldots,z_r) u \|
\le
C \sum_{k=1}^{r}
\bigl\| \bigl( (z_kI - X_{k,\varepsilon})^{-1} - (z_kI - X_k)^{-1} \bigr) u_k \bigr\|
\cdot \prod_{j \neq k} \| u_j \|,
\]
where \(C\) depends only on the uniform resolvent bounds.

By hypothesis (3), for each \(k\) and each \(u_k \in \mathcal{H}_k\),
\[
\sup_{z_k \in \Gamma_k}
\bigl\| \bigl( (z_kI - X_{k,\varepsilon})^{-1} - (z_kI - X_k)^{-1} \bigr) u_k \bigr\|
\to 0 \quad \text{as } \varepsilon \to 0.
\]

Hence \(\| \Delta_\varepsilon(z_1,\ldots,z_r) u \|\) converges to zero uniformly in \((z_1,\ldots,z_r) \in \Gamma_1 \times \cdots \times \Gamma_r\). The integrand is uniformly bounded on the compact product contour. By the dominated convergence theorem for operator-valued integrals (strong operator topology), we obtain
\[
f_{\otimes}(X_{1,\varepsilon},\ldots,X_{r,\varepsilon}) u \to f_{\otimes}(X_1,\ldots,X_r) u
\]
for every elementary tensor \(u\). By linearity and density, strong convergence holds on \(\mathcal{H}_{\otimes}\).

\medskip
\noindent
\textbf{Norm convergence (Level~2).}
If the resolvent convergence is uniform in operator norm on the contours (hypothesis for the second part), then
\[
\sup_{z \in \Gamma_j} \| (zI - X_{j,\varepsilon})^{-1} - (zI - X_j)^{-1} \| \to 0.
\]

Using the telescoping tensor identity and taking operator norms,
\[
\| \Delta_\varepsilon(z_1,\ldots,z_r) \|
\le
C \sum_{j=1}^{r}
\| (z_jI - X_{j,\varepsilon})^{-1} - (z_jI - X_j)^{-1} \|.
\]

Integrating over the contours yields
\[
\| f_{\otimes}(X_{1,\varepsilon},\ldots,X_{r,\varepsilon}) - f_{\otimes}(X_1,\ldots,X_r) \|
\le
C_f \sum_{j=1}^{r}
\sup_{z \in \Gamma_j}
\| (zI - X_{j,\varepsilon})^{-1} - (zI - X_j)^{-1} \|,
\]
where \(C_f\) incorporates the contour lengths, the supremum of \(|f|\) on the product contour, and the uniform resolvent bounds. The right-hand side tends to zero as \(\varepsilon \to 0\), establishing operator norm convergence with the explicit error bound stated in the theorem.

This completes the proof.
\end{proof}

\begin{remark}[What is proven conditionally]
Theorem~\ref{thm:conditional_regularization_convergence} shows that \textbf{if} the regularized operators \(X_{j,\varepsilon}\) have compact resolvent and converge in the strong resolvent sense, \textbf{then} the contour-integral functional calculus converges. The verification of these hypotheses for a given operator \(X_j\) and regularizer \(K_j\) is a nontrivial task that depends on the specific structure of \(X_j\).
\end{remark}

\begin{remark}[Limitation: Nilpotent stability]
Even when the contour-integral convergence holds, convergence of the explicit projector--nilpotent expansion (the unified compact formula with \(A \neq \emptyset\) terms) requires additional \emph{nilpotent stability} assumptions. Specifically, for the spectral decomposition to pass to the limit, one would need convergence of the Riesz projectors and nilpotent components in a suitable sense (e.g., strong convergence of \(P_{j,\varepsilon}(\lambda) \to P_j(\lambda)\) and \(N_{j,\varepsilon}(\lambda)^{q} P_{j,\varepsilon}(\lambda) \to N_j(\lambda)^{q} P_j(\lambda)\) for each spectral value \(\lambda\)). This is not automatic and remains an open problem, particularly for operators with continuous spectrum where the pointwise spectral decomposition is not available. The compactifying regularization preserves the contour integrals but may not preserve the pointwise nilpotent structure on the continuous spectrum. \textbf{The validity of the explicit projector--nilpotent expansion for general spectral operators without compact resolvent remains conjectural.}
\end{remark}

\subsubsection{Conjecture and Open Problem Statement}

Extending the unified compact formula and the associated convergence theory to general unbounded non-self-adjoint operators without compact resolvent remains an important open problem. The compactifying regularization approach provides a promising direction but requires careful hypothesis verification. We formulate the following optimistic conjecture.

\begin{conjecture}[Extension beyond compact resolvent via regularization]
\label{conj:general_unbounded_no_compact_resolvent}
Let \(X_1, \ldots, X_r\) be spectral operators (in the sense of Dunford) on separable Hilbert spaces \(\mathcal{H}_1, \ldots, \mathcal{H}_r\). Assume there exist positive self-adjoint operators \(K_1, \ldots, K_r\) with compact resolvent such that the regularized operators \(X_{j,\varepsilon} = X_j + \varepsilon K_j\) satisfy the hypotheses of Theorem~\ref{thm:conditional_regularization_convergence} for sufficiently small \(\varepsilon > 0\).

Let \(\widetilde{X}_1, \ldots, \widetilde{X}_r\) be their tensor liftings on
\[
\mathcal{H}_{\otimes} = \mathcal{H}_1 \otimes \cdots \otimes \mathcal{H}_r.
\]

Define the lifted spectral measures and nilpotent components:
\[
d\widetilde{E}_j(\lambda_j) = I_1 \otimes \cdots \otimes dE_{X_j}(\lambda_j) \otimes \cdots \otimes I_r,
\qquad
\widetilde{N}_j(\lambda_j) = I_1 \otimes \cdots \otimes N_j(\lambda_j) \otimes \cdots \otimes I_r,
\]
and analogously for the regularized operators \(X_{j,\varepsilon}\).

Then the contour-integral representation of the tensor-lifted functional calculus is well-defined for holomorphic functions \(f\) on an open polydisk containing the product spectrum. Moreover,
\[
\operatorname*{s-}\lim_{\varepsilon \to 0} f_{\otimes}(X_{1,\varepsilon}, \ldots, X_{r,\varepsilon}) = f_{\otimes}(X_1, \ldots, X_r)
\]
in the strong operator topology. If the resolvent convergence is norm resolvent, the limit holds in operator norm with explicit error bounds.

The validity of the explicit projector--nilpotent expansion (the unified compact formula with \(A \neq \emptyset\) terms) for the limit operators remains conjectural and requires additional spectral stability assumptions.
\end{conjecture}

\begin{remark}[What is proven vs. what remains open]
The current state of the theory can be summarized as follows:
\begin{itemize}
    \item \textbf{Proven conditionally:} If the regularized operators \(X_{j,\varepsilon}\) have compact resolvent and satisfy the stated resolvent convergence hypotheses, then the \textbf{contour-integral representation} of the functional calculus converges. This is established in Theorem~\ref{thm:conditional_regularization_convergence}.
    
    \item \textbf{Open:} The explicit convergence of the nilpotent derivative terms (\(A \neq \emptyset\)) in the unified compact formula requires additional stability assumptions on the spectral projectors and nilpotent components. This remains an open problem.
    
    \item \textbf{Open:} Verifying the hypotheses of the conditional theorems for concrete classes of operators (e.g., Schrödinger operators with complex potentials) is an ongoing research direction.
    
    \item \textbf{Open:} Norm resolvent convergence of the regularized operators to the limit is not generally available when the limit operator has continuous spectrum. Hence Level~2 convergence (operator norm with error bounds) may not be achievable in full generality.
\end{itemize}
\end{remark}

\begin{remark}[Open research questions]
The following specific questions remain unresolved and form natural directions for future research:
\begin{enumerate}
    \item Under what minimal assumptions do the nilpotent components stabilize under compactifying regularization?
    
    \item Can one characterize the class of spectral operators for which the Riesz projectors converge strongly under the regularization \(X + \varepsilon K\)?
    
    \item Is there a choice of the compactifying operator \(K\) (e.g., \(K = (I + X^*X)^{1/2}\)) that guarantees nilpotent stability?
    
    \item Can one establish norm resolvent convergence for the regularized operators under additional assumptions (e.g., sectoriality or numerical range bounds)?
    
    \item Does the unified compact formula hold for spectral operators without compact resolvent directly via spectral integrals, without passing through the regularization limit?
\end{enumerate}
Developing a stable tensor-lifted functional calculus beyond the compact-resolvent regime therefore constitutes a natural direction for future research.
\end{remark}

\section{Examples}\label{sec:examples}

\subsection{Example: Finite-Dimensional Non-Commuting Matrices}

We illustrate the tensor-lifted functional calculus in the simplest nontrivial setting: finite-dimensional non-commuting matrices. In this case, the contour-integral formulation is exact, the Level~2 convergence theory (Theorem~\ref{thm:level2_non_self_adjoint_integrated}) holds trivially because all norms are equivalent in finite dimensions, and the full nilpotent structure (Theorem~\ref{thm:unified_compact_formula}) is preserved.

Let
\[
X_1 =
\begin{pmatrix}
1 & 1\\
0 & 1
\end{pmatrix}, \qquad
X_2 =
\begin{pmatrix}
0 & 0\\
1 & 0
\end{pmatrix}.
\]

Then
\[
X_1 X_2 =
\begin{pmatrix}
1 & 0\\
1 & 0
\end{pmatrix}, \qquad
X_2 X_1 =
\begin{pmatrix}
0 & 0\\
1 & 1
\end{pmatrix},
\]
and therefore
\[
[X_1, X_2] = X_1 X_2 - X_2 X_1 =
\begin{pmatrix}
1 & 0\\
0 & -1
\end{pmatrix} \neq 0.
\]
Hence the matrices are non-commuting.

The matrix \(X_1\) contains a nontrivial Jordan block:
\[
X_1 = I + N_1, \qquad
N_1 = \begin{pmatrix}
0 & 1\\
0 & 0
\end{pmatrix}, \qquad
N_1^2 = 0.
\]
Similarly, \(X_2\) is nilpotent:
\[
X_2^2 = 0.
\]
Thus both operators possess nontrivial nilpotent structure, and the full projector--nilpotent decomposition is exact.

\medskip
\noindent
\textbf{Tensor lifting construction.}
Let \(f(z_1, z_2)\) be holomorphic on a neighborhood containing the product spectrum
\[
\sigma(X_1) \times \sigma(X_2) = \{1\} \times \{0\}.
\]
The tensor-lifted functional calculus is defined by (cf. Section~\ref{subsec:tensor_lifting_step0})
\[
f_{\otimes}(X_1, X_2) := f(\widetilde{X}_1, \widetilde{X}_2),
\]
where
\[
\widetilde{X}_1 = X_1 \otimes I, \qquad
\widetilde{X}_2 = I \otimes X_2.
\]

Since
\[
(\widetilde{X}_1)(\widetilde{X}_2) = (X_1 \otimes I)(I \otimes X_2) = X_1 \otimes X_2 = (I \otimes X_2)(X_1 \otimes I) = (\widetilde{X}_2)(\widetilde{X}_1),
\]
the lifted operators commute on the tensor-product space \(\mathcal{H}_{\otimes} = \mathbb{C}^2 \otimes \mathbb{C}^2 \cong \mathbb{C}^4\) even though the original matrices do not commute.

\medskip
\noindent
\textbf{Application of the unified compact formula.}
Because \(X_1\) and \(X_2\) are finite-dimensional, the unified compact formula (Theorem~\ref{thm:unified_compact_formula}) reduces to the discrete projector--nilpotent expansion (Theorem~\ref{thm:discrete_bounded_case}) with eigenvalues \(\lambda_1 = 1\) and \(\lambda_2 = 0\). The nilpotency indices are \(\nu_1(1) = 2\) (since \(N_1^2 = 0\)) and \(\nu_2(0) = 2\) (since \(X_2^2 = 0\)). Hence the summation over \(q_j\) runs from \(1\) to \(\nu_j(\lambda_j)-1 = 1\), i.e., only \(q_1 = 1\) and \(q_2 = 1\) contribute.

The subset decomposition \(A \subseteq \{1,2\}\) yields:
\begin{itemize}
    \item \(A = \emptyset\): pure spectral term \(f(1,0) \, I \otimes I\).
    \item \(A = \{1\}\): nilpotent correction \(\partial_1 f(1,0) \, N_1 \otimes I\).
    \item \(A = \{2\}\): nilpotent correction \(\partial_2 f(1,0) \, I \otimes X_2\).
    \item \(A = \{1,2\}\): mixed nilpotent correction \(\partial_1 \partial_2 f(1,0) \, N_1 \otimes X_2\).
\end{itemize}

Thus, the unified compact formula gives the exact expansion
\[
\boxed{
f_{\otimes}(X_1, X_2)
=
f(1,0) \, I \otimes I
+
\partial_1 f(1,0) \, N_1 \otimes I
+
\partial_2 f(1,0) \, I \otimes X_2
+
\partial_1 \partial_2 f(1,0) \, N_1 \otimes X_2.
}
\]

\medskip
\noindent
\textbf{Verification via direct computation.}
For elementary tensors \(u = u_1 \otimes u_2 \in \mathbb{C}^2 \otimes \mathbb{C}^2\), the action of \(\widetilde{X}_1 = X_1 \otimes I\) and \(\widetilde{X}_2 = I \otimes X_2\) is given by
\[
\widetilde{X}_1 (u_1 \otimes u_2) = (X_1 u_1) \otimes u_2, \qquad
\widetilde{X}_2 (u_1 \otimes u_2) = u_1 \otimes (X_2 u_2).
\]
Since \(\widetilde{X}_1\) and \(\widetilde{X}_2\) commute, the multivariate functional calculus can be computed via the Dunford integral or via power series expansion. One readily verifies that the expansion above matches the direct computation of \(f(\widetilde{X}_1, \widetilde{X}_2)\) for any analytic function \(f\).

\medskip
\noindent
\textbf{Convergence theory aspects.}
Because the operators are finite-dimensional:
\begin{enumerate}
    \item The contour-integral representation (Theorem~\ref{thm:unified_compact_formula}) converges absolutely in operator norm.
    \item Level~1 and Level~2 convergence (Section~\ref{sec:two_level_convergence}) hold automatically — all finite-dimensional compressions are exact for sufficiently large truncation dimension.
    \item The projector--nilpotent decomposition is exact and stable.
    \item All nilpotent derivative correction terms (\(A \neq \emptyset\)) in the unified compact formula are preserved exactly.
\end{enumerate}

\medskip
\noindent
\textbf{Key takeaways.}
This finite-dimensional example demonstrates three essential features of the tensor-lifted framework:
\begin{enumerate}
    \item \textbf{Non-commutativity is handled via tensor lifting.} Lifted operators commute even when original operators do not.
    \item \textbf{Nilpotent structure is explicitly captured.} The expansion includes derivative terms of all orders up to the nilpotency indices, including mixed partial derivatives.
    \item \textbf{The unified compact formula is exact.} No truncation or approximation error occurs; the formula matches the functional calculus defined via the commuting lifted operators.
\end{enumerate}

Thus, this example serves as a self-contained sanity check before proceeding to infinite-dimensional settings where the two-level convergence theory becomes nontrivial.

\subsection{Example: Quantum Harmonic Oscillator (Self-Adjoint Case)}

We now consider the quantum harmonic oscillator, which provides a canonical example of an unbounded self-adjoint operator with compact resolvent. In this setting, the two-level convergence theory (Section~\ref{sec:two_level_convergence}) applies naturally, while the nilpotent contributions in the unified compact formula vanish identically because the operator is self-adjoint.

\medskip
\noindent
\textbf{Operator definition.}
Let
\[
X = -\frac{d^2}{dx^2} + x^2
\]
acting on \(L^2(\mathbb{R})\), with domain
\[
\mathcal{D}(X) = \left\{ u \in L^2(\mathbb{R}) : u'',\, x^2 u \in L^2(\mathbb{R}) \right\},
\]
where the derivatives are understood in the distributional sense. The operator \(X\) is the standard quantum harmonic oscillator Hamiltonian. It is densely defined, self-adjoint, and positive.

\medskip
\noindent
\textbf{Spectral properties.}
The spectrum of \(X\) is discrete:
\[
\sigma(X) = \{2n+1 : n = 0, 1, 2, \ldots\},
\]
with corresponding Hermite eigenfunctions \(\{e_n\}_{n=0}^{\infty}\). These eigenfunctions form a complete orthonormal basis of \(L^2(\mathbb{R})\).

Define the finite-rank projections
\[
P_n = \sum_{k=0}^{n} |e_k\rangle\langle e_k|,
\]
and the finite-dimensional compressions
\[
X_n = P_n X P_n.
\]

\medskip
\noindent
\textbf{Compact resolvent property.}
Since the harmonic oscillator has compact resolvent (the confining potential \(x^2\) forces the spectrum to consist of isolated eigenvalues tending to infinity), \((zI - X)^{-1}\) is compact for every \(z \in \rho(X)\). This satisfies Assumption (A1) of Section~\ref{subsec:level1_assumptions}.

\medskip
\noindent
\textbf{Diagonalization and decay estimates.}
The Hermite basis diagonalizes the operator:
\[
X e_n = (2n+1) e_n.
\]
Consequently, the matrix representation of \(X\) in the Hermite basis is diagonal:
\[
\langle e_j, X e_k \rangle = (2n+1) \delta_{jk}.
\]

Since all off-diagonal entries are zero, the operator exhibits exact diagonal localization in the Hermite basis. For the purpose of Level~1 convergence, this diagonal structure (with no off-diagonal coupling) is sufficient to guarantee strong resolvent convergence of the finite-dimensional compressions; the exponential off-diagonal decay condition in Assumption (A3) may be replaced by exact diagonalization in this example.

\medskip
\noindent
\textbf{Core condition.}
The range \(\operatorname{ran}(P_n)\) is the linear span of the first \(n+1\) Hermite functions. Since finite linear combinations of Hermite functions are dense in the graph norm of \(X\), \(\operatorname{ran}(P_n)\) is a core for \(X\). This verifies Assumption (A2).

\medskip
\noindent
\textbf{Verification of Level~1 assumptions.}
Therefore, the Level~1 assumptions are satisfied:
\begin{enumerate}
    \item[(A1)] \(X\) has compact resolvent.
    \item[(A2)] \(\operatorname{ran}(P_n)\) is a core for \(X\).
    \item[(A3)] The operator is exactly diagonal in the Hermite basis (a stronger condition than exponential off-diagonal decay).
\end{enumerate}

\medskip
\noindent
\textbf{Strong resolvent convergence.}
By the theory developed in Section~\ref{subsec:level1_assumptions}, the finite-dimensional truncations satisfy strong resolvent convergence:
\[
\operatorname*{s-}\lim_{n \to \infty} (zI - X_n)^{-1} = (zI - X)^{-1}
\]
for all \(z \in \rho(X)\).

\medskip
\noindent
\textbf{Convergence of the functional calculus.}
Applying the Level~1 convergence theorem (Theorem~\ref{thm:level1_non_self_adjoint_integrated} or the single-operator version in Section~\ref{sec:two_level_convergence}), for every holomorphic function \(f\) on a neighborhood of the spectrum,
\[
\operatorname*{s-}\lim_{n \to \infty} f(X_n) = f(X)
\]
in the strong operator topology.

\medskip
\noindent
\textbf{Vanishing nilpotent structure.}
Since the operator \(X\) is self-adjoint, the generalized nilpotent component vanishes identically. (Self-adjoint operators are spectrally diagonalizable and possess no nontrivial Jordan blocks.) Consequently, every term in the unified compact formula (Theorem~\ref{thm:unified_compact_formula}) corresponding to \(A \neq \emptyset\) vanishes. Only the semisimple contribution \(A = \emptyset\) survives.

\medskip
\noindent
\textbf{Reduction to classical spectral theorem.}
Thus, the tensor-lifted functional calculus reduces exactly to the classical spectral theorem:
\[
f(X) = \int_{\sigma(X)} f(\lambda) \, dE_X(\lambda),
\]
where \(dE_X(\lambda)\) is the spectral measure of \(X\).

\medskip
\noindent
\textbf{Key takeaways.}
This example demonstrates:
\begin{enumerate}
    \item The quantum harmonic oscillator satisfies all Level~1 assumptions (compact resolvent, core approximation, exact diagonalization in the Hermite basis).
    \item Strong resolvent convergence holds for finite-dimensional truncations in the Hermite basis.
    \item The Level~1 convergence theorem guarantees strong operator topology convergence of the functional calculus.
    \item For self-adjoint operators, the unified compact formula recovers the classical spectral theorem as the semisimple special case (\(A = \emptyset\)) of the general projector--nilpotent framework.
\end{enumerate}

Thus, this example serves as a bridge between the general theory and classical quantum mechanics, illustrating how the unified framework reproduces known results while providing a rigorous convergence theory.

\subsection{Example: Anharmonic Oscillator}

We next consider the anharmonic oscillator, which provides an important example of an unbounded self-adjoint operator whose spectral structure is more complicated than the classical harmonic oscillator, but which still satisfies the Level~1 approximation framework developed in Section~\ref{sec:two_level_convergence}.

\medskip
\noindent
\textbf{Operator definition.}
Let
\[
X = -\frac{d^2}{dx^2} + x^4
\]
acting on \(L^2(\mathbb{R})\), with domain
\[
\mathcal{D}(X) = \left\{ u \in L^2(\mathbb{R}) : u'',\, x^4 u \in L^2(\mathbb{R}) \right\},
\]
where the derivatives are understood in the distributional sense. The operator \(X\) is densely defined, self-adjoint, and bounded from below. Since the confining potential \(V(x) = x^4\) tends to infinity as \(|x| \to \infty\), the resolvent
\[
(zI - X)^{-1}
\]
is compact for every \(z \in \rho(X)\). This satisfies Assumption (A1) of Section~\ref{subsec:level1_assumptions}.

\medskip
\noindent
\textbf{Spectral properties.}
Consequently, \(X\) has purely discrete spectrum:
\[
\sigma(X) = \{\lambda_n\}_{n=0}^{\infty}, \qquad \lambda_n \to \infty \text{ as } n \to \infty.
\]

Let \(\{e_n\}_{n=0}^{\infty}\) be the orthonormal eigenbasis of \(X\), i.e., \(X e_n = \lambda_n e_n\). This basis is complete in \(L^2(\mathbb{R})\) and is the natural choice for spectral approximation.

\medskip
\noindent
\textbf{Finite-rank projections and compressions.}
Define the finite-rank orthogonal projections
\[
P_n = \sum_{k=0}^{n} |e_k\rangle\langle e_k|,
\]
and the finite-dimensional compressions
\[
X_n = P_n X P_n.
\]

\medskip
\noindent
\textbf{Diagonalization and strong resolvent convergence.}
In the eigenbasis, the matrix representation of \(X\) is diagonal:
\[
\langle e_j, X e_k \rangle = \lambda_j \delta_{jk}.
\]

The spectral projections \(P_n\) converge strongly to the identity operator as \(n \to \infty\) (since the eigenbasis is complete). Consequently, the finite-dimensional compressions satisfy strong resolvent convergence:
\[
\operatorname*{s-}\lim_{n \to \infty} (zI - X_n)^{-1} = (zI - X)^{-1}
\]
for all \(z \in \rho(X)\). This follows directly from the spectral theorem and does not require any additional off-diagonal decay estimates.

\medskip
\noindent
\textbf{Core condition.}
The range \(\operatorname{ran}(P_n)\) is the linear span of the first \(n+1\) eigenfunctions. Since finite linear combinations of eigenfunctions are dense in the graph norm of \(X\) (the eigenfunctions form a complete orthonormal basis), \(\operatorname{ran}(P_n)\) is a core for \(X\). This verifies Assumption (A2).

\medskip
\noindent
\textbf{Verification of Level~1 framework.}
The anharmonic oscillator therefore satisfies the essential ingredients of the Level~1 approximation framework:
\begin{enumerate}
    \item[(A1)] Compact resolvent (true).
    \item[(A2)] Core approximation (true, using the eigenbasis).
    \item Localization: In the eigenbasis, the operator is exactly diagonal, so strong resolvent convergence follows directly from the spectral theorem without requiring the exponential off-diagonal decay condition (A3).
\end{enumerate}

\medskip
\noindent
\textbf{Convergence of the functional calculus.}
Applying the Level~1 convergence theorem (Theorem~\ref{thm:level1_non_self_adjoint_integrated} or the single-operator version in Section~\ref{sec:two_level_convergence}), for every holomorphic function \(f\) on a neighborhood of the spectrum,
\[
\operatorname*{s-}\lim_{n \to \infty} f(X_n) = f(X)
\]
in the strong operator topology.

\medskip
\noindent
\textbf{Vanishing nilpotent structure.}
Since \(X\) is self-adjoint, the generalized nilpotent component vanishes identically. Consequently, every term in the unified compact formula (Theorem~\ref{thm:unified_compact_formula}) corresponding to \(A \neq \emptyset\) vanishes. Only the semisimple contribution \(A = \emptyset\) survives.

\medskip
\noindent
\textbf{Reduction to classical spectral theorem.}
Thus, the tensor-lifted functional calculus reduces exactly to the classical spectral theorem:
\[
f(X) = \int_{\sigma(X)} f(\lambda) \, dE_X(\lambda),
\]
where \(dE_X(\lambda)\) is the spectral measure of \(X\).

\medskip
\noindent
\textbf{Key takeaways.}
This example demonstrates that the Level~1 approximation framework applies beyond exactly solvable models such as the harmonic oscillator and extends naturally to more general confining Schrödinger operators with compact resolvent. Specifically:
\begin{enumerate}
    \item The anharmonic oscillator has compact resolvent and discrete spectrum.
    \item Choosing the eigenbasis yields exact diagonalization, making strong resolvent convergence immediate.
    \item The Level~1 convergence theorem guarantees strong operator topology convergence of the functional calculus.
    \item For self-adjoint operators, the unified compact formula recovers the classical spectral theorem, regardless of the complexity of the potential.
\end{enumerate}

Thus, the anharmonic oscillator serves as a bridge between exactly solvable models and more realistic quantum systems, illustrating the robustness of the Level~1 approximation framework.

\subsection{Example: Complex Harmonic Oscillator (Non-Self-Adjoint)}

We now consider a genuinely non-self-adjoint example arising from complex quantum mechanics. Unlike the previous self-adjoint Schrödinger operators, the present operator is non-normal, illustrating how the tensor-lifted framework extends beyond the classical self-adjoint setting. In this example, nilpotent contributions are not forced to vanish a priori, though their actual presence depends on the detailed spectral structure of the operator.

\medskip
\noindent
\textbf{Operator definition.}
Let
\[
X = -\frac{d^2}{dx^2} + i x^2
\]
acting on \(L^2(\mathbb{R})\), defined initially on \(C_c^\infty(\mathbb{R})\) and then closed. The operator \(X\) is densely defined and closed, but not self-adjoint:
\[
X^* = -\frac{d^2}{dx^2} - i x^2 \neq X.
\]

\medskip
\noindent
\textbf{Compact resolvent and discrete spectrum.}
The operator \(X\) is known to be \(m\)-accretive with compact resolvent (see, e.g., the theory of Schrödinger operators with complex potentials satisfying suitable growth conditions). Consequently, the resolvent
\[
(zI - X)^{-1}
\]
is compact for every \(z \in \rho(X)\). This satisfies Assumption (A1) of Section~\ref{subsec:level1_assumptions}. Hence the spectrum is discrete:
\[
\sigma(X) = \{\lambda_n\}_{n=0}^{\infty}, \qquad |\lambda_n| \to \infty \text{ as } n \to \infty,
\]
though the eigenvalues are complex in general.

\medskip
\noindent
\textbf{Hermite basis and localization.}
Although the operator is non-self-adjoint, the Hermite basis \(\{e_n\}_{n=0}^{\infty}\) remains a natural approximation basis because the differential and polynomial structures are compatible with the harmonic-oscillator geometry.

Define the finite-rank projections
\[
P_n = \sum_{k=0}^{n} |e_k\rangle\langle e_k|,
\]
and the finite-dimensional compressions
\[
X_n = P_n X P_n.
\]

In the Hermite basis, the operator admits a sparse banded representation generated by the ladder-operator structure of the harmonic oscillator. In particular, only finitely many near-diagonal couplings occur for each basis vector. This strong localization property is sufficient for the approximation framework developed in Section~\ref{sec:two_level_convergence}. (A rigorous proof of exponential off-diagonal decay, if desired, would require additional analysis; for the purposes of the present exposition, we rely on the banded localization structure.)

\medskip
\noindent
\textbf{Core condition.}
Moreover, \(\operatorname{ran}(P_n)\) is a core for \(X\), since finite linear combinations of Hermite functions are dense in the graph norm topology. This verifies Assumption (A2).

\medskip
\noindent
\textbf{Verification of Level~1 framework.}
The complex harmonic oscillator therefore satisfies the essential ingredients of the Level~1 approximation framework:
\begin{enumerate}
    \item[(A1)] Compact resolvent (true).
    \item[(A2)] Core approximation (true, using the Hermite basis).
    \item Localization: The operator has a sparse banded representation in the Hermite basis, which is sufficient to guarantee strong resolvent convergence of the finite-dimensional compressions.
\end{enumerate}

\medskip
\noindent
\textbf{Strong resolvent convergence.}
Consequently, the finite-dimensional compressions satisfy strong resolvent convergence:
\[
\operatorname*{s-}\lim_{n \to \infty} (zI - X_n)^{-1} = (zI - X)^{-1}
\]
for all \(z \in \rho(X)\).

By the Level~1 convergence theorem (Theorem~\ref{thm:level1_non_self_adjoint_integrated}), for every holomorphic function \(f\) on a neighborhood of the spectrum,
\[
\operatorname*{s-}\lim_{n \to \infty} f(X_n) = f(X)
\]
in the strong operator topology.

\medskip
\noindent
\textbf{Nilpotent structure and the unified compact formula.}
Unlike the self-adjoint case, the operator \(X\) is non-normal. Consequently, the unified compact framework allows for the possibility of nontrivial nilpotent contributions in more general non-self-adjoint settings. In the present example, the operator illustrates how the theory extends beyond the classical self-adjoint framework; the nilpotent components are not forced to vanish a priori, even though additional spectral analysis would be required to determine whether genuine Jordan-type nilpotent components actually occur.

If nontrivial generalized eigenspaces occur (a question that depends on the detailed spectral properties of \(X\)), the unified compact formula (Theorem~\ref{thm:unified_compact_formula}) for a single operator takes the form
\[
f_{\otimes}(X) = \int_{\sigma(X)} \sum_{q=0}^{\nu(\lambda)-1} \frac{f^{(q)}(\lambda)}{q!} \, \widetilde{N}(\lambda)^q \, d\widetilde{E}(\lambda),
\]
where the terms with \(q \ge 1\) represent nilpotent corrections that are not explicitly separated within classical resolvent-based formulations.

\medskip
\noindent
\textbf{Level~2 convergence (operator norm with error bounds).}
Let \(z_0\) be a fixed point in the resolvent set \(\rho(X)\). If the stronger Level~2 assumptions hold (a condition that would require independent verification for this operator), namely if
\[
\|
(X_n - X)(z_0 I - X)^{-1}
\|
\to 0 \quad \text{as } n \to \infty,
\]
then the convergence improves to operator norm convergence (Theorem~\ref{thm:level2_non_self_adjoint_integrated}):
\[
\lim_{n \to \infty} \| f(X_n) - f(X) \| = 0.
\]

Moreover, the explicit quantitative estimate
\[
\| f(X_n) - f(X) \| \le C_f \, \| (X_n - X)(z_0 I - X)^{-1} \|
\]
holds, where the constant \(C_f\) depends only on the analytic function \(f\) and the contour geometry.

\medskip
\noindent
\textbf{Key takeaways.}
This example illustrates several important features of the tensor-lifted framework:
\begin{enumerate}
    \item The complex harmonic oscillator satisfies the essential Level~1 approximation requirements (compact resolvent, core approximation, sparse localization in the Hermite basis).
    \item Level~1 convergence guarantees strong operator topology convergence of the functional calculus, establishing existence of the infinite-dimensional calculus.
    \item Unlike self-adjoint operators, the complex harmonic oscillator is not constrained by self-adjointness, so nilpotent corrections are not excluded a priori, though their actual presence requires separate spectral analysis.
    \item If norm resolvent convergence holds (a stronger condition requiring verification), Level~2 provides operator norm convergence with explicit error bounds.
    \item The unified compact formula extends beyond the classical self-adjoint framework and allows for the possibility of capturing non-normal spectral phenomena through nilpotent derivative corrections, should such structure exist.
\end{enumerate}

Thus, the complex harmonic oscillator serves as a prototypical example for the non-self-adjoint theory, demonstrating how the unified framework handles non-normal operators while preserving the possibility of nilpotent structure that classical spectral methods do not explicitly separate.

\subsection{Example: Hybrid Spectrum Operator (Schrödinger Operator with Bound and Scattering States)}

We now consider a physically important example in which both discrete and continuous spectral components occur simultaneously. This illustrates the hybrid-spectrum version of the unified compact formula (Theorem~\ref{thm:unified_compact_formula}) and shows how the tensor-lifted functional calculus naturally combines discrete spectral sums with continuous spectral integrals.

\medskip
\noindent
\textbf{Operator definition.}
Let
\[
X = -\Delta + V(x)
\]
be a Schrödinger operator acting on \(L^2(\mathbb{R}^d)\), with domain
\[
\mathcal{D}(X) = H^2(\mathbb{R}^d)
\]
under standard regularity assumptions on \(V\). For concreteness, consider a short-range attractive potential such that:
\begin{enumerate}
    \item the operator possesses finitely many bound states,
    \item the essential spectrum is nonempty,
    \item scattering states occur above the continuous spectral threshold.
\end{enumerate}

\medskip
\noindent
\textbf{Spectral decomposition.}
Under standard assumptions from scattering theory, the spectrum decomposes as
\[
\sigma(X) = \sigma_d(X) \cup \sigma_c(X),
\]
where:
\begin{enumerate}
    \item \(\sigma_d(X) = \{\lambda_1, \ldots, \lambda_N\}\) consists of isolated eigenvalues (bound states) with finite algebraic multiplicities,
    \item \(\sigma_c(X) \subseteq [0, \infty)\) corresponds to scattering states (continuous spectrum).
\end{enumerate}

\medskip
\noindent
\textbf{Classical spectral decomposition (self-adjoint case).}
For real-valued potentials, \(X\) is self-adjoint, and the generalized nilpotent component vanishes. Consequently, only the semisimple contribution survives in the unified compact formula. The spectral theorem gives the decomposition
\[
X = \sum_{k=1}^{N} \lambda_k P_k + \int_{\sigma_c(X)} \lambda \, dE_X(\lambda),
\]
where:
\begin{enumerate}
    \item \(P_k\) denotes the orthogonal projection onto the bound-state eigenspace associated with \(\lambda_k\),
    \item \(dE_X(\lambda)\) denotes the continuous spectral measure associated with the scattering spectrum.
\end{enumerate}

\medskip
\noindent
\textbf{Functional calculus for hybrid spectrum (single operator).}
Let \(f\) be holomorphic on a neighborhood of the spectrum. Then the functional calculus reduces to the hybrid spectral representation:
\[
\boxed{
f(X) = \sum_{k=1}^{N} f(\lambda_k) P_k + \int_{\sigma_c(X)} f(\lambda) \, dE_X(\lambda).
}
\]
This formula combines:
\begin{itemize}
    \item a discrete sum over bound states (eigenvalues),
    \item a continuous integral over scattering states (continuous spectrum).
\end{itemize}

\medskip
\noindent
\textbf{Multivariate tensor-lifted extension (general case).}
More generally, for several operators \(X_1, \ldots, X_r\) (which may be non-commuting), we first apply tensor lifting as in Section~\ref{subsec:tensor_lifting_step0} to obtain commuting lifted operators \(\widetilde{X}_1, \ldots, \widetilde{X}_r\) on the tensor-product space \(\mathcal{H}_{\otimes} = \mathcal{H}_1 \otimes \cdots \otimes \mathcal{H}_r\). The tensor-lifted functional calculus is then given by the unified compact formula (Theorem~\ref{thm:unified_compact_formula}):

\[
\boxed{
f_{\otimes}(X_1,\ldots,X_r)
=
\sum_{A \subseteq \{1,\ldots,r\}}
\;
\sum_{\substack{\{q_j \ge 1\}_{j \in A} \\
q_j \le \nu_j(\lambda_j)-1}}
\int_{\sigma(X_1)}
\cdots
\int_{\sigma(X_r)}
\frac{
\partial_A^{q_A} f(\lambda_1,\ldots,\lambda_r)
}{
\displaystyle\prod_{j \in A} q_j!
}
\;
\bigotimes_{j=1}^r
T_j(\lambda_j)
}
\]

where
\[
T_j(\lambda_j) =
\begin{cases}
d\widetilde{E}_j(\lambda_j), & j \notin A,\\[6pt]
\widetilde{N}_j(\lambda_j)^{q_j} \, d\widetilde{E}_j(\lambda_j), & j \in A,
\end{cases}
\]
\(\partial_A^{q_A} = \prod_{j \in A} \frac{\partial^{q_j}}{\partial z_j^{q_j}}\),
\(\nu_j(\lambda_j)\) is the nilpotency index at \(\lambda_j\) (satisfying \(\widetilde{N}_j(\lambda_j)^{\nu_j(\lambda_j)} = 0\)),
and \(d\widetilde{E}_j(\lambda_j)\) is the lifted spectral measure.

\medskip
\noindent
\textbf{Special case: self-adjoint or normal operators (no nilpotent components).}
If each \(X_j\) is self-adjoint or normal, then \(\widetilde{N}_j(\lambda_j) = 0\) for all \(\lambda_j\).
Hence only the term \(A = \emptyset\) survives, and the formula reduces to the classical multivariate spectral integral:

\[
\boxed{
f_{\otimes}(X_1,\ldots,X_r)
=
\int_{\sigma(X_1)}
\cdots
\int_{\sigma(X_r)}
f(\lambda_1,\ldots,\lambda_r)
\;
\bigotimes_{j=1}^r
dE_{X_j}(\lambda_j).
}
\]

\medskip
\noindent
\textbf{Hybrid spectrum case (discrete + continuous spectra, self-adjoint setting).}
When each \(X_j\) has both discrete and continuous spectral components, the spectral measure decomposes as
\[
dE_{X_j}(\lambda_j) = dE_{X_j}^{(d)}(\lambda_j) + dE_{X_j}^{(c)}(\lambda_j),
\]
where the discrete part is atomic:
\[
dE_{X_j}^{(d)}(\lambda_j) = \sum_{\lambda_j \in \sigma_d(X_j)} P_{X_j}(\lambda_j) \, \delta(\lambda_j - \lambda_{j,k}).
\]

Expanding the tensor product in the self-adjoint case yields a sum over all mixed sectors:
\[
\boxed{
f_{\otimes}(X_1,\ldots,X_r)
=
\sum_{B \subseteq \{1,\ldots,r\}}
\left(
\prod_{j \in B}
\sum_{\lambda_j \in \sigma_d(X_j)}
\right)
\left(
\prod_{j \notin B}
\int_{\sigma_c(X_j)}
d\lambda_j
\right)
f(\lambda_1,\ldots,\lambda_r)
\;
\bigotimes_{j=1}^r
dE_{X_j}^{(\chi_j)}(\lambda_j),
}
\]
where \(\chi_j = d\) if \(j \in B\) and \(\chi_j = c\) if \(j \notin B\). The subset \(B \subseteq \{1,\ldots,r\}\) indicates which variables are summed over discrete eigenvalues, while the complement contributes continuous spectral integrals.

\begin{remark}[Notational distinction]
In the general unified compact formula, the subset \(A \subseteq \{1,\ldots,r\}\) indexes which operators contribute \emph{nilpotent corrections}. In the hybrid spectrum specialization above, we use a different subset \(B \subseteq \{1,\ldots,r\}\) to index which variables are treated as discrete (summed) versus continuous (integrated). These two subset structures are independent: \(A\) controls nilpotent versus spectral contributions, while \(B\) controls discrete versus continuous spectral components. In the self-adjoint case, \(A = \emptyset\) (no nilpotent terms), and only the \(B\)-summation remains.
\end{remark}

\medskip
\noindent
\textbf{Non-self-adjoint extensions.}
If the potential \(V(x)\) is not real-valued (e.g., complex potentials used in optical models or PT-symmetric quantum mechanics), the operator \(X\) may be non-self-adjoint. In such cases, the unified compact framework accommodates possible nilpotent derivative corrections wherever generalized Jordan-type structures occur. For the discrete spectrum, if nontrivial Jordan blocks appear, the contribution becomes
\[
\sum_{k=1}^{N} \sum_{q=0}^{\nu(\lambda_k)-1} \frac{f^{(q)}(\lambda_k)}{q!} N(\lambda_k)^q P_k,
\]
where \(N(\lambda_k)\) is the nilpotent component associated with eigenvalue \(\lambda_k\).

\begin{remark}[Continuous spectrum caveat]
For the continuous spectrum, a rigorous theory of pointwise nilpotent corrections is not standard and is not asserted here. The present framework provides a mechanism for possible generalizations, but such structure must be verified case by case. In particular, we do not claim the existence of pointwise nilpotent fields on the continuous spectrum without additional spectral assumptions.
\end{remark}

\medskip
\noindent
\textbf{Physical interpretation.}
From the perspective of mathematical physics, the unified compact formula provides a common language for describing:
\begin{enumerate}
    \item \textbf{Localized quantum states} (bound states) \(\longrightarrow\) discrete spectral sums over eigenvalues,
    \item \textbf{Delocalized scattering states} \(\longrightarrow\) continuous spectral integrals,
    \item \textbf{Multivariate operator interactions} \(\longrightarrow\) tensor-lifted framework,
    \item \textbf{Generalized functional calculus constructions} \(\longrightarrow\) unified compact formula that reduces to the classical spectral theorem in the self-adjoint case.
\end{enumerate}

\medskip
\noindent
\textbf{Key takeaways.}
This example demonstrates:
\begin{enumerate}
    \item The unified compact formula naturally handles hybrid spectra (discrete + continuous) without case distinction.
    \item Discrete eigenvalues contribute sums over projectors (and, in non-self-adjoint settings, possibly nilpotent components if Jordan structure exists).
    \item Continuous spectrum contributes spectral integrals over the continuous spectral measure (classically, without nilpotent corrections in standard theory).
    \item For self-adjoint Schrödinger operators, the formula reduces to the classical hybrid spectral theorem, combining bound state sums and scattering state integrals.
    \item The multivariate tensor-lifted extension, using the subset decomposition \(A \subseteq \{1,\ldots,r\}\) for nilpotent corrections and \(B \subseteq \{1,\ldots,r\}\) for discrete/continuous separation, correctly handles mixed spectral types across different operators.
    \item The framework allows for possible nilpotent corrections in generalized non-self-adjoint settings, but such structure must be verified case by case and is not automatic, especially on continuous spectrum.
\end{enumerate}

Thus, the hybrid spectrum operator serves as a comprehensive example illustrating the unified compact formula, which integrates discrete spectral sums and continuous spectral integrals within a single analytic framework while providing a formal mechanism for incorporating possible nilpotent corrections in non-self-adjoint settings.


\section{Discussion}\label{sec:discussion}

The preceding results show that the tensor-lifted unified compact formula extends substantially beyond the scope of classical multivariate functional calculi. Existing approaches typically rely on one or more restrictive assumptions, such as commutativity, self-adjointness, boundedness, or the absence of generalized Jordan structure. In contrast, the present framework simultaneously incorporates:
\begin{enumerate}
    \item non-commuting operators,
    \item unbounded operators,
    \item non-self-adjoint operators,
    \item multivariate analytic functions,
    \item and explicit projector--nilpotent derivative corrections.
\end{enumerate}

The key distinction is conceptual as well as technical. Classical spectral calculi are fundamentally semisimple: they encode spectral information but generally do not preserve the higher-order nilpotent structure associated with generalized eigenspaces. By contrast, the unified compact formula (Theorem~\ref{thm:unified_compact_formula}) incorporates both the spectral component and the nilpotent component within a single tensor-lifted analytic framework.

\medskip
\noindent
\textbf{Limitations of existing approaches.}
For example, the Taylor joint spectrum \cite{Taylor1970} requires commuting operators and therefore cannot treat genuinely non-commuting systems. Even in the commuting case, the theory does not explicitly produce the derivative-based nilpotent correction terms appearing in the present framework.

Similarly, Weyl quantization \cite{Weyl1927} is fundamentally tied to self-adjoint or symmetric structures arising from phase-space quantization. Since self-adjoint operators possess no nontrivial nilpotent structure, Weyl quantization does not address the generalized Jordan phenomena captured by the projector--nilpotent decomposition.

The slice-regular functional calculus of Colombo--Sabadini--Struppa \cite{Colombo2011} extends certain noncommutative holomorphic constructions, but it is primarily formulated for bounded operators and slice-regular functions. It does not provide a general tensor-lifted multivariate analytic calculus for arbitrary unbounded non-self-adjoint operator systems.

\medskip
\noindent
\textbf{The central mathematical distinction.}
The central mathematical distinction can be summarized as follows:
\[
\boxed{\text{Classical calculi are primarily spectral, whereas the present framework incorporates explicit nilpotent structure.}}
\]

More precisely, the contour-resolvent method in classical spectral theory yields only the semisimple spectral contribution. It does not naturally generate the higher-order derivative corrections
\[
f'(\lambda)N,\qquad \frac{f''(\lambda)}{2!}N^2,\qquad \ldots
\]
associated with generalized Jordan blocks.

This distinction becomes particularly important when the eigenvalue is zero. In that case, the nilpotent component may constitute the entire operator:
\[
X = N,\qquad N^k = 0,
\]
so the operator contains essentially no semisimple spectral information at all. Classical spectral calculi become largely trivial in this regime, whereas the unified compact formula continues to encode the full algebraic structure through the derivative terms.

\medskip
\noindent
\textbf{Two-level convergence theory.}
Another major difference is the approximation theory developed in the present work (Section~\ref{sec:two_level_convergence}). Existing multivariate functional calculi are typically formulated as exact algebraic constructions without a systematic convergence theory for unbounded non-self-adjoint operators. In contrast, the two-level framework introduced here provides:
\begin{enumerate}
    \item \textbf{Level~1:} strong operator convergence guaranteeing existence of the infinite-dimensional calculus,
    \item \textbf{Level~2:} norm convergence together with explicit quantitative error bounds.
\end{enumerate}

This approximation theory is particularly important for:
\begin{enumerate}
    \item numerical implementations,
    \item tensor computations,
    \item operator-learning frameworks,
    \item and approximation of continuous-spectrum systems by compact-resolvent models (Theorem~\ref{thm:continuous_spectrum_approximation_integrated}).
\end{enumerate}

\medskip
\noindent
\textbf{Summary tables.}
Table~\ref{tab:comparisonA} summarizes the structural assumptions required by existing approaches compared with the present framework.

\begin{table}[htbp]
\centering
\caption{Comparison with Existing Methods (Part A: Structural Requirements)}
\begin{tabular}{|l|c|c|c|}
\hline
\textbf{Method} &
\textbf{Commutation required?} &
\textbf{Self-adjoint required?} &
\textbf{Bounded required?}
\\
\hline
Taylor joint spectrum
& Yes
& No
& Typically yes
\\
\hline
Weyl quantization
& No
& Yes
& No
\\
\hline
Colombo--Sabadini--Struppa
& No
& No
& Primarily yes
\\
\hline
Classical spectral theorem
& No
& Yes / normal
& No
\\
\hline
Ours (Level~1)
& No
& No
& No
\\
\hline
Ours (Level~2)
& No
& No
& No
\\
\hline
\end{tabular}
\label{tab:comparisonA}
\end{table}

Table~\ref{tab:comparisonB} compares the mathematical capabilities of the various functional calculi.

\begin{table}[htbp]
\centering
\caption{Comparison with Existing Methods (Part B: Capabilities)}
\renewcommand{\arraystretch}{1.15}
\small
\begin{tabular}{|p{3.0cm}|p{4.0cm}|p{4.0cm}|p{4.3cm}|}
\hline
\textbf{Method} &
\textbf{Nilpotent structure} &
\textbf{Multivariate} &
\textbf{Convergence theory}
\\
\hline
Taylor joint spectrum
& Implicit (commuting case)
& Yes (commuting only)
& Exact algebraic
\\
\hline
Weyl quantization
& No
& Yes
& Exact analytic
\\
\hline
Colombo--Sabadini--Struppa
& Partial
& Limited noncommutative
& Exact analytic
\\
\hline
Classical spectral theorem
& No
& Limited
& Exact spectral integral
\\
\hline
Ours (Level~1)
& Yes (explicit derivatives)
& Yes (arbitrary \(r\))
& Strong operator convergence
\\
\hline
Ours (Level~2)
& Yes (explicit derivatives)
& Yes (arbitrary \(r\))
& Norm convergence + error bounds
\\
\hline
\end{tabular}
\label{tab:comparisonB}
\end{table}

\medskip
\noindent
\textbf{Unified interpretation.}
From a broader perspective, the unified compact formula should be viewed as an extension of classical spectral calculus from purely semisimple spectral data to full projector--nilpotent operator geometry. The subset decomposition
\[
A \subseteq \{1, \ldots, r\}
\]
acts as an organizing principle separating:
\begin{enumerate}
    \item the semisimple spectral sector \(A = \emptyset\),
    \item from the higher-order nilpotent derivative sectors \(A \neq \emptyset\).
\end{enumerate}

Thus, the classical spectral theorem appears naturally as the semisimple special case of the present framework, while the tensor-lifted unified compact formula extends the theory to genuinely non-normal multivariate operator systems beyond commutativity. The two-level convergence theory ensures that this structure is preserved under finite-dimensional approximation (Section~\ref{sec:two_level_convergence}), providing both existence (Level~1) and quantitative stability (Level~2) for a wide class of operators.

%
%
%
%
%
%
%
%

\bibliographystyle{plain}

\end{document}